\documentclass{amsart}

\usepackage{amssymb}
\usepackage{url}
\usepackage{amscd}
\usepackage{texdraw}

\theoremstyle{plain}

\newtheorem{thm}{Theorem}

\newtheorem*{ml}{Main Lemma}

\newtheorem{cor}{Corollary}[section]
\newtheorem{lem}{Lemma}[section]
\newtheorem{rmk}{Remark}[section]

\newtheorem{defi}{Definition}[section]
\newtheorem{pro}{Proposition}[section]

\input txdtools

\begin{document}

\date{}

\title{On David type Siegel Disks of the Sine family}
\author{Gaofei Zhang}
\address{Department of  Mathematics, Nanjing University, Nanjing,   210093, P. R.
China} \email{zhanggf@nju.edu.cn}
\thanks{}
\begin{abstract}
In 2008 Petersen posed a list of questions on the application of
trans-quasiconformal Siegel surgery developed by Zakeri and himself.
In this paper we extend Petersen-Zakeri's idea  so that the surgery
can be applied to all the premodels which have
 no ``free critical points''. We explain how the idea is used in
solving three of the  questions posed by Petersen.  To present the
details of the idea,  we focus  on the solution of one of them: we
prove that for typical rotation numbers $0< \theta < 1$, the
boundary of the Siegel disk of $f_{\theta}(z) = e^{2 \pi i \theta}
\sin (z)$ is a Jordan curve which passes through exactly two
critical points $\pi/2$ and $-\pi/2$.
\end{abstract}

\maketitle


\section{Introduction}

The idea of trans-quasiconformal surgery was first used by
Haissinsky to make attracting basins into parabolic basins. It was
then introduced  by Petersen and Zakeri  to the study of Siegel
disks with typical rotation numbers \cite{PZ}.

Let us first sketch how Petersen-Zakeri's trans-quasiconformal
Siegel surgery works.  We say an irrational number $0< \theta < 1$
is of $\emph{David type}$ if $\log{a_{n}} = O(\sqrt n)$,  where
$[a_{1}, a_{2}, \cdots]$ is the continued fraction of $\theta$. It
is known that the set of David type irrational numbers in $[0, 1]$
has full Lebesgue measure. Let $0< \theta < 1$ be a David type
irrational number. Let $\Bbb T$ denote the unit circle. Consider the
degree-3 Blaschke product
$$
G(z) = e^{it} z^{2} \frac{z - 3}{1 - 3z}
$$
where $0< t < 2 \pi$ is chosen such that $G|\Bbb T: \Bbb T \to \Bbb
T$ is a homeomorphism of rotation number $\theta$.  It is known that
$G$ has a double critical point at $1$ and  has no other critical
points on $\Bbb T$.  By Yoccoz's linearization theorem, there is a
homeomorphism $h: \Bbb T \to \Bbb T$  such that $G|\Bbb T = h^{-1}
\circ R_\theta \circ h$ where $R_\theta: z \to e^{2 \pi i \theta} z$
is the rigid rotation given by $\theta$. By Herman's theorem, $h$ is
quasi-symmetric if and only if $\theta$ is of bounded type, that is,
$\sup\{a_n\}< \infty$. Since the set of bounded type irrational
numbers has zero Lebesgue measure, for typical rotation numbers
$\theta$, $h$
 is not quasi-symmetric and thus can not be quasiconformally
 extended
to the unit disk.  With the aid of Yoccoz's cell construction,
Petersen and Zakeri showed that if $\theta$ is of David type, then
$h$ has a trans-quasiconformal extension $H: \Delta \to \Delta$.
Here ``trans-quasiconformal" means that  the map $H: \Delta \to
\Delta$ satisfies the following conditions.
\begin{itemize}\item[1.] The map $H: \Delta \to \Delta$ is a homeomorphism in the Sobolev class ${
W_{\rm loc}^{1, 1}}$, \item[2.]   $H$ has a degenerate Beltrami
coefficient, that is, $\|\mu_H\|_{\infty} = 1$,
\item[3.]   there exist $M, \alpha > 0$ and $0< \epsilon_0 < 1$ such
that for any $0 < \epsilon < \epsilon_0$, the following inequality
holds,
\begin{equation}\label{dgr}
m(\{z \:|\: |\mu_{H}(z)| > 1 - \epsilon \} < M
e^{-\frac{\alpha}{\epsilon}}
\end{equation} where $m(\cdot)$ denotes the area with respect to the Euclidean metric
 on the plane or the area with respect to the spherical metric.
 \end{itemize}  Then as in
the quasiconformal surgery, define the premodel $\widehat{G}$  as
follows.
$$
\widehat{G}(z) =
\begin{cases}

 G(z) & \text{for $z \in \widehat{\Bbb C} \setminus \Delta$}, \\

H^{-1} \circ R_{\theta} \circ H (z) & \text{for $z \in \Delta$}.
\end{cases}
$$
Now  spread $\mu_H$ by the iterated inverse branches of
$\widehat{G}$ to all the drops and get a $\widehat{G}$-invariant
Beltrami differential $\mu$ on ${\Bbb C}$. Note that
$\|\mu\|_{\infty} = \|\mu_{H}\|_{\infty} = 1$. Thus unlike in
quasiconformal surgery, here one can not use measurable Riemann
mapping theorem to get a plane homeomrphism $\phi$ which solves the
Beltrami equation given by $\mu$.  The new idea  of this approach
 is to replace  the measurable Riemann mapping theorem by a theorem of David:  if  there exist constants $M, \alpha > 0$ and $0<
\epsilon_0 < 1$ such that $\mu$ satisfies the exponential decay
condition (\ref{dgr})  on the plane, then there is a plane
homeomorphism $\phi \in W_{\rm loc}^{1, 1}$ which solves the
Beltrami equation given by $\mu$.  Once $\phi$ is obtained,  one can
show that $\phi \circ \widehat{G} \circ \phi^{-1}$ is a quadratic
polynomial with a Jordan Siegel disk of rotation number $\theta$
whose boundary contains the unique finite critical point.  Note that
such a quadratic polynomial must be linearly conjugate to $e^{2 \pi
i \theta} z + z^2$.  So  the proof is reduced to showing that  $\mu$
satisfies  the integrability condition (\ref{dgr}) on the plane.
This is the core part of the Petersen-Zaker's proof. In \cite{PZ},
this step relies essentially on the delicate geometry of a puzzle
construction for the premodel $\widehat{G}$.   These puzzles    were
constructed  by Petersen in \cite{P1} and is now called Petersen
puzzles. In many situations, however, such puzzle construction is
not available. The following questions, which are examples of such
situations, were posed by Petersen in his 2008-lecture notes at
Liverpool University.

Q1. Is it true that for typical rotation numbers, the Siegel disk of
$e^{2 \pi i \theta} \sin(z)$ is a Jordan domain with the boundary
containing the two critical points $\pi/2$ and $-\pi/2$?

Q2. Is it true that for typical rotation numbers $\theta$ and $\tau$
with $\theta + \tau\ne 1$, the  two Siegel disks of the rational map
$$
f_{\theta, \tau}(z) =  \frac{z ^{2} + e^{2 \pi i \theta} z } {1 +
e^{2 \pi i \tau}z}$$ are Jordan domains with each boundary
containing exactly one critical point?

Q3. Is it true that for typical rotation numbers, the Siegel disk of
a cubic polynomial is a Jordan domain with the boundary containing
at least one critical point?

In the case of bounded type rotation numbers, the answers to all the
three questions are affirmative  and are obtained through
quasi-conformal surgery, see \cite{YZ}, \cite{Za1} and \cite{Zh1}.
In order to use Petersen-Zakeri's trans-quasiconformal Siegel
surgery to solve the questions for typical rotation numbers, one has
to overcome the problem that  Petersen's puzzle construction is not
available for all the three cases. The goal of this paper is to
develop a method by which one can verify the integrability condition
(\ref{dgr})  in certain  situations where Petersen puzzles are not
available. Roughly speaking, the method applies to the following
situation: {\sl  The premodel has  no ``free critical points", that
is, the forward orbit of every critical point either interests the
closure of the rotation disk(s), or is eventually periodic, or is
attracted to some attracting or parabolic cycles, or lands at some
essential singularity for a meromorphic function.}

In \cite{CZ}, \cite{Zh3} and the present paper, by using this
method, we solved all the above three questions. In this paper, we
solve Q1 by proving
\begin{thm}\label{thm1}
Let $0< \theta < 1$ be an irrational number of David type. Then the
boundary of the Siegel disk of $f_{\theta}(z) = e^{2 \pi i \theta}
\sin (z)$ is a Jordan curve which passes through exactly two
critical points $\pi/2$ and $-\pi/2$.
\end{thm}

An outline of the proof of Theorem~\ref{thm1} will be given in
$\S2$.

In  \cite{CZ} we proved that for David type  rotation numbers $0<
\theta, \tau < 1$ with $\theta + \tau \ne 1$, the two Siegel
quadratic polynomials $e^{2 \pi i \theta}z + z^2$ and $e^{2 \pi i
\tau}z + z^2$ are conformally matable. In particular, this solves
Q2.
\begin{thm}[Ch\'{e}ritat-Zhang, \cite{CZ}]\label{thm2}
For David type irrational numbers $0< \theta, \tau <1$ with $\theta
+ \tau \ne 1$, the quadratic rational map
$$f_{\theta, \tau}(z) =  \frac{z ^{2} +
e^{2 \pi i \theta} z } {1 + e^{2 \pi i \tau}z}$$ is the conformal
mating of the two Siegel quadratic polynomials $e^{2 \pi i \theta}z
+ z^2$ and $e^{2 \pi i \tau}z + z^2$. In particular, the boundaries
of the two Siegel disks of $f$ are Jordan curves containing a
critical point on each of them.
\end{thm}

In the case that both $\theta$ and $\tau$ are of bounded type,
Theorem~\ref{thm2} had been proved by Yampolsky and Zakeri in
\cite{YZ}. One of the key tools in Yampolsky-Zakeri's proof is  a
degree three Blaschke product $B$ such that \begin{itemize}
\item[1.]  $B$ has an indifferent fixed point at  infinity with
multiplier $e^{2 \pi i \tau}$, and \item[2.]  when restricted to the
unit circle, $B|\Bbb T: \Bbb T\to \Bbb T$ is a circle homeomorphism
with rotation number $\theta$ and with exactly one critical point at
$1$ (which is a double critical point).
\end{itemize} For the construction of such $B$, see $\S4$ of
\cite{YZ}. These models are now called Yampolsky-Zakeri's mating
models. With the aid of these models and quasiconformal surgery,
Yampolsky and Zakeri proved Theorem~\ref{thm2} for bounded type
irrational numbers $0< \theta, \tau< 1$ with $\theta + \tau \ne 1$.

But for  David type rotation numbers, we are not able to perform
Petersen-Zakeri's trans-quasiconformal surgery on Yampolsky-Zakeri's
mating models. The reason can be roughly sketched as follows. For
David type irrational numbers $\tau$,  the Yampolsky-Zakeri's mating
model $B$ has a Siegel disk centered at infinity whose boundary is
contained in the $\omega$-limit set of  a critical point of $B$. A
priori, the boundary of this Siegel disk could be very complicated,
and therefore, the forward orbit of the critical point might be very
wild. For instance, its $\omega$-limit set could even be dense in
the Julia set.  The existence of  such ``free critical point''
causes a big problem  in verifying  the integrability condition
(\ref{dgr}).

 The key idea in \cite{CZ}  is to
construct premodels with no ``free critical points"  so that  the
method developed in this paper can be applied.   The following is a
brief description of this construction.

For $R > 0$, let $\Bbb T_{R}$ denote the Euclidean circle $\{z\:|\:
|z| = R\}$.  Let $\sigma_1$ and $\sigma_R$ denote the reflection
about $\Bbb T$ and $\Bbb T_{R}$, respectively. In \cite{CZ} we prove
that  for any two irrational numbers $0 < \theta, \tau < 1$ with
$\theta + \tau \ne 1$, there exist an $R
> 1$, a  real number $0\le s < 1$ and
 four distinct points $a_{0}, b_{0}, \alpha_{0}, \beta_{0}$ in  the annulus $\{z\:|\:1 < |z|  < R^{2}\}$  with
 $\sigma_R (a_0) = \alpha_0$ and $\sigma_R(b_0) = \beta_0$,
 such that the following
infinite  Blaschke fraction
$$
 B(z) = e^{ 2 \pi i s} z \prod_{k= 0}^{\infty} \bigg{(}\frac{z
- a_{k} } {1 - \overline{a_{k}}  z}  \frac{ 1 - \overline{b_{k}}
z}{z - b_{k}}  \frac{z - \beta_{k}} {1 - \overline{\beta_{k}}  z}
\frac{ 1 - \overline{\alpha_{k}} z}{z - \alpha_{k}} \bigg{)}
$$
 where   $$a_{k}=
R^{2k} \cdot a_{0},\:\: b_{k} = R^{2k} \cdot b_{0},\:\: \alpha_{k}=
R^{2k} \cdot \alpha_{0}  \hbox{  and  } \beta_{k}= R^{2k} \cdot
\beta_{0}\:\:,  \forall k \ge 0,$$  satisfies the following
properties:
\begin{itemize}
\item[1.] $B \circ \sigma_1 = \sigma_1 \circ B$  and  $B \circ
\sigma_R = \sigma_R \circ B$, \item[2.] $B|\Bbb T: \Bbb T \to \Bbb
T$ and $B|\Bbb T_{R}: \Bbb T_{R} \to \Bbb T_{R}$ are both  circle
homeomorphisms with rotation numbers $\theta$ and $1 - \tau$,
respectively,
\item[3.]  $B$ has exactly one critical point on each of $\Bbb T$ and
$\Bbb T_{R}$, say $c_{1}$ and $c_{2}$, and both of them are double
critical points,
\item[4.] there exist two Jordan domains $U$ and $V$ contained in $$H
= \{z\:|\:1 < |z| < R\}$$ which are attached to $c_{1}$ and $c_{2}$
respectively, such that $$B: U \to \{z\:|\:|z| < 1\} \hbox{  and  }
B: V \to \widehat{\Bbb C} \setminus \{z\:|\: |z| \le R\}$$ are both
holomorphic isomorphisms, and $B: H \setminus \overline{U \cup V}
\to H$ is a two-to-one holomorphic covering map.
\end{itemize}

In \cite{CZ} we call such $B$ a holomorphic torus mapping since it
is obtained by iterating Thurston pull back map on the
Teichm\"{u}ller space modeled on a torus with finitely many points
marked.  Recently Cheritat found  several different constructions of
such bi-symmetric Blaschke fractions.

Now we can construct the desired premodel $\widehat{B}$ as follows.
Assume that both $\theta$ and $\tau$ are  David type irrational
numbers.   Let $h_{1}: \Bbb T \to \Bbb T$ and $h_{2}: \Bbb T_{R} \to
\Bbb T_{R}$ be two circle homeomorphisms such that $B|\Bbb T =
h_{1}^{-1} \circ R_{\theta} \circ h_{1}$ and $B|\Bbb T_{R} =
h_{2}^{-1} \circ R_{1-\tau}\circ h_{2}$ where $R_{\theta}$ and
$R_{1- \tau}$ are the rigid rotations given by $\theta$ and $1 -
\tau$, respectively.  Let $\Delta = \{z|\:|z| < 1\}$ and
$\Delta_{R}^{c} = \widehat{\Bbb C} \setminus \{z\: |\:|z| \le R\}$.
By Yoccoz's extension theorem proved in the appendix of \cite{PZ},
there exist David homeomorphisms $H_{1}: \Delta \to \Delta$ and
$H_{2}: \Delta_{R}^{c} \to \Delta_{R}^{c}$ which extend $h_{1}$ and
$h_{2}$, respectively (For the definition of David homeomorphism,
see $\S 3$). Now define the premodel $\widehat{B}$ as follows.
$$
\widehat{B}(z) =
\begin{cases}

 H_{1}^{-1} \circ R_{\theta} \circ H_{1}(z) & \text{for $z \in \Delta$}, \\
 B(z) & \text{for $1 \le |z| \le R$}, \\

H_{2}^{-1} \circ R_{1- \tau} \circ H_{2} (z) & \text{for $z \in
\Delta_{R}^{c}$}.
\end{cases}
$$
Let $\mu_{1}$ and $\mu_{2}$ be the Beltrami differentials of $H_1$
and $H_2$ on $\Delta$ and $\Delta_R^{c}$, respectively. Then spread
$\mu_1$ and $\mu_2$  by the iterated inverse branches of
$\widehat{B}$ to all the drops and get a $\widehat{B}$-invariant
Beltrami differential $\mu$ on $\widehat{\Bbb C}$.   It is clear
that $\widehat{B}$ has two critical points both of which  are
contained in  the boundaries of the two rotation  disks, $\Delta$
and $\Delta_{R}^{c}$.  This means that $\widehat{B}$ has no ``free
critical points".  So we can use the method developed in this paper
to show that $\mu$ satisfies the integrability condition (\ref{dgr})
on the sphere.  By the same argument as in the proof of Lemma 5.5 in
\cite{PZ}, it  follows that $\widehat{B}$ is topologically conjugate
to $f_{\theta, \tau}$. In particular, this implies  that for typical
rotation numbers $\theta$ and $\tau$ with $\theta + \tau \ne 1$, the
boundaries of the two Siegel disks of $f_{\theta, \tau}$ are Jordan
curves containing a critical point on each of them. This solves Q2.
By adapting the arguments in \cite{Ya} and  \cite{YZ}, we can
further prove that $f_{\theta, \tau}$ is the conformal mating of
$P_{\theta}$ and $P_{\tau}$.  This completes the proof of
Theorem~\ref{thm2}.

The answer to Q3 is also affirmative.  In fact we prove that it is
true for polynomial maps of all degrees \cite{Zh3}.
\begin{thm}[Zhang, \cite{Zh3}]\label{thm3}
All David type Siegel disks of polynomial maps are Jordan domains
with at least one critical point on their boundaries.
\end{thm}
The very rough idea of the proof of Therorem~\ref{thm3} is as
follows. Fix an integer $d \ge 3$ and a David type irrational number
$0 < \theta < 1$. Let $\mathcal{P}^{d}_{\theta}$ denote the class of
all the polynomial maps with degree $\le d$ and having a fixed
Siegel disk centered at the origin and with rotation number
$\theta$. Let $\mathcal{Q}^{d}_{\theta} \subset
\mathcal{P}^{d}_{\theta}$ be the subclass consisting of all those
polynomial maps  for which all the finite critical points are
contained in the boundary of the Siegel disk centered at the origin.
Using the method developed in this paper, we are able to show that
each polynomial map in $\mathcal{Q}^{d}_{\theta}$ is obtained by
performing Petersen-Zakeri's trans-quasiconformal Siegel surgery on
certain premodel with no ``free critical points" . These premodels
are produced from Blaschke products of degree $2d-1$ whose critical
points, except $0$ and $\infty$,  are all contained in $\Bbb T$.
Next we consider a family of functions which measure the oscillation
of the boundaries of the Siegel disks for the polynomial maps in
$\mathcal{P}^{d}_{\theta}$.  We  first prove that these functions
are ``maximized'' on the polynomial maps in
$\mathcal{Q}^{d}_{\theta}$. Then we prove that these oscillation
functions are uniformly bounded for all the polynomial maps in
$\mathcal{Q}^{d}_{\theta}$. Thus the ``oscillation'' of the boundary
of the Siegel disk of every polynomial map in
$\mathcal{P}^{d}_{\theta}$ is uniformly bounded. From this we deduce
that for every polynomial map in $\mathcal{P}^{d}_{\theta}$, the
boundary of the Siegel disk centered at the origin is a Jordan curve
passing through at least one critical point. This implies
Theorem~\ref{thm3}.

In this paper we focus on  the proof of Theorem~\ref{thm1}.

To learn more about Siegel disks of entire functions,  the reader
may refer to \cite{BF}, \cite{G}, \cite{Re} and the articles in the
references there.
\section{Outline of the proof of Theorem 1}

 Throughout the
paper, we use $\widehat{\Bbb C}$, $\Bbb C$, $\Bbb C^{*}$, $\Delta$,
and $\Bbb T$ to denote the Riemann sphere, the complex plane, the
complex plane with a puncture at the origin, the open unit disk, and
the unit circle, respectively. The following is an outline of the
proof of Theorem~\ref{thm1}.

In $\S3$, we introduce  $\emph{David homeomorphisms}$ and David's
integrability theorem.

In $\S4$, we present some basic results on the real bounds of
critical circle mappings which will be used in this paper.

In $\S5$, we construct an odd Blaschke fraction $G_{\theta}$ to
serve as the model map for $f_{\theta}$. The restriction of
$G_{\theta}$ on  $\Bbb T$ is an analytic  homeomorphism with
rotation number $\theta$ and two critical points $1$ and $-1$. Since
the framework  of Yoccoz's cell  construction  presented in
\cite{PZ} is made for analytic circle mappings with exactly one
critical point, we will
 transform $G_{\theta}$ to an intermediate model map which has exactly one critical point in $\Bbb T$.   The idea is
as follows.  Let $\Phi: {\Bbb C} \to {\Bbb C}$  be the map given by
$z \to z^{2}$. Let
$$
g_{\theta}(z) = \Phi \circ G_{\theta} \circ \Phi^{-1}(z).
$$
In Lemma~\ref{rotation number}, we will prove that $g_{\theta}$ is a
well-defined meoromorphic function  such that  \begin{itemize}
\item[1.] the restriction of $g_{\theta}$ to $\Bbb T$ is an analytic
circle mapping  with exactly one critical point at $1$,  and
\item[2.]
 the rotation number of $g_{\theta}|\Bbb T: \Bbb T \to \Bbb T$ is  $\tau \equiv 2 \theta \mod(1)$ (that
is, $\tau = 2 \theta$ if $0< \theta < 1/2$ and $\tau = 2 \theta -1$
if $1/2 < \theta < 1$).  \end{itemize}
In Lemma~\ref{arith}, we
prove that $\tau $ is also of David type. By Yoccoz's linearization
theorem \cite{Yo1}, there is a circle homeomorphism $h: \Bbb T \to
\Bbb T$ such that $h(1) = 1$ and
$$g_{\theta}|{\Bbb T}(z) = h^{-1} \circ R_{\tau} \circ h(z)$$
where $R_{\tau}: z \to e^{2\pi i \tau} z$ is the rigid rotation
given by $\tau$.

Following the appendix of  \cite{PZ},
 we introduce Yocooz's cell construction
in $\S6$ and  extend the circle homeomorphism $h$ to a David
homeomorphism
$$H: \Delta \to \Delta.$$ Let
$$
\nu_{H} = \frac{\overline{\partial} H }{\partial H}
$$
be the Beltrami differential of $H$ in $\Delta$.   Define
\begin{equation} \widetilde{g}_{\theta}(z) =
\begin{cases}
 g_{\theta}(z) & \text{for $z \in {\Bbb C} - \Delta$}, \\
 H^{-1}\circ R_{\alpha}\circ H(z)  & \text{
for $z \in \Delta$}.
\end{cases}
\end{equation}

Note that $\nu_{H}$ is $\widetilde{g}_{\theta}$-invariant in
$\Delta$.  By pulling back  $\nu_{H}$ through the iterations of
$\widetilde{g}_{\theta}$, we get a
$\widetilde{g}_{\theta}$-invariant Beltrami differential in the
whole complex plane. Let us denote this Beltrami differential by
$\nu$.   Let $\mu$ be the Beltrami differential in the complex plane
which is defined by pulling back $\nu$ through the square map
$\Phi$. Define
\begin{equation} \widetilde{G}_{\theta}(z) =
\begin{cases}
 G_{\theta}(z) & \text{for $z \in {\Bbb C} - \Delta$}, \\
\Phi^{-1}\circ H^{-1}\circ R_{\tau}\circ H \circ \Phi(z) & \text{
for $z \in \Delta$}.
\end{cases}
\end{equation}
We will show that $\widetilde{G}_{\theta}$ is well defined in
$\Delta$ and is continuous in $\Bbb C$. In Lemma~\ref{final} we show
that $\mu$ is $\widetilde{G}_{\theta}$-invariant. In
Lemma~\ref{equi}, we  show that the integrability of $\nu$ implies
that of $\mu$.

In $\S7$, we prove Theorem 1 by assuming the integrability of $\mu$.
 By David's integrability
theorem,  there is a unique homeomorphism $\phi: \widehat{\Bbb C}
\to \widehat{\Bbb C}$ in $W_{\rm loc}^{1,1}({\Bbb C})$ which fixes
$0$ and $\infty$ and maps $1$ to $\pi/2$ and such that
$$\overline{\partial} \phi = \mu
\partial \phi. $$ Then  using the same argument as in the proof of Lemma 5.5 of
\cite{PZ},  we prove that the map $T_{\theta}(z) = \phi \circ
\widetilde{G}_{\theta}\circ \phi^{-1}(z)$ is an entire function
(Lemma~\ref{PH}). From the construction above, $T_{\theta}$ has a
Siegel disk centered at the origin with rotation number $\theta$,
and moreover,  the boundary of the Siegel disk is a Jordan curve
containing exactly two critical points $\pi/2$ and $-\pi/2$. We then
prove that  $f_{\theta}(z) = T_{\theta}(z)$.  This is proved by
using a topological rigidity property of the sine family which was
proved  in \cite{DS}.  Theorem~\ref{thm1} then follows.

Since the integrability of $\nu$ implies that of $\mu$ by
Lemma~\ref{equi}, it remains to show the integrability of $\nu$.
This is the main task of the paper.
 In $\S8$, we extend the idea of Petersen-Zakeri's proof to verify the integrability of
 $\nu$. More precisely, we will prove that
 there exist constants
$M
> 0$, $\alpha
> 0$,  and $0< \epsilon_{0} < 1$, such that for any $0< \epsilon <
\epsilon_{0}$, the following inequality holds,
\begin{equation}\label{integrability-p}
area \{z\:\big{|}\: |\nu(z)| > 1 - \epsilon \} \le M
e^{-\alpha/\epsilon},
\end{equation}
where $area(X)$  denotes the spherical area of a measurable subset
$X \subset \widehat{\Bbb C}$. This is the heart of the paper.

The idea of our proof can be sketched as follows.  Let $Y_n$ be the
union of all the Yoccoz's cells of level $n$. Each Yoccoz cell is
like a closed trapezoid. For the construction of Yoccoz cells, see
$\S6$.  Then $\{Y_n\}$ is a sequence of nested closed annuli in
$\Delta$:
$$
Y_{N_0} \supset Y_{N_0+1} \supset \cdots Y_n \supset Y_{n+1} \supset
\cdots
$$  where $N_0 \ge 1$ is some fixed integer.  By  Yoccoz's extension theorem (cf. Theorem 6.5 of \cite{PZ} or Theorem~\ref{Yoccoz extension} of the this paper),
  there exist
constants $1 < \lambda, \: K < \infty$ such that
\begin{itemize}
 \item[1.] $H$ is $K$-quasiconformal in $\Delta \setminus {Y_{N_{0}}}$,
 \item[2.] for all $n \ge N_{0}$, the dilatation of $H$  in $\Delta\setminus
Y_{n+2}$ is not greater than $\lambda \cdot n$. \end{itemize}

Now let
$$
X = \{z \in {\Bbb C} \setminus \overline{\Delta}\:\big{|}\:
g_{\theta}^{k}(z) \in \Delta \hbox{ for some integer } k \ge 1\}.
$$
For each $z \in X$, let $k_{z} \ge 1$ be the least positive integer
such that $g_{\theta}^{k_{z}}(z) \in \Delta$.    Define
$$
X_{n} = \{z \in X \:\big{|}\: g_{\theta}^{k_{z}}(z) \in Y_{n}\}.
$$
\begin{rmk}\label{rk1}{\rm
From the definition of $X_n$, it is easy to see that if $z \in X_n$,
then all the points in the backward orbit of $z$ under $
\widetilde{g}_{\theta}(z)$ belong to $X_n$. }
\end{rmk}

Inspired by Petersen-Zakeri's proof ( cf. Theorem 4.15 of
\cite{PZ}), in  Proposition~\ref{poo8}, we  reduce  the proof of the
integrability of $\nu$ to

\begin{ml}\label{main lemma}
There exist $C > 0$, $0< \epsilon < 1$,  $0< \delta < 1$ and an
integer $N_{1} \ge N_{0}$  such that
\begin{equation}\label{pz-cci}
area(X_{n+2}) \le C \cdot \epsilon^{n} + \delta \cdot \:
area(X_{n}), \:\:\:\forall \:n >  N_1.
\end{equation}
\end{ml}
We remark here that   Petersen-Zakeri's proof of (\ref{pz-cci})
relies essentially on Petersen's puzzle construction for the
pre-model $\widehat{G}$. To overcome the problem caused by lack of
Petersen's puzzle construction for $\widetilde{g}_{\theta}$, we
introduce a new idea in this paper. The idea contains three key
ingredients which are described as follows.

The first ingredient of the idea   is  a variant of Vitali's
Covering Lemma. For $z \in {\Bbb C}$ and $r
> 0$, let $B_{r}(z)$ denote the Euclidean disk with radius $r$ and
centered at $z$.
\begin{defi}[$K$-bounded geometry]\label{bgp} Let $K > 1$ and $(U, V)$ be a
pair of sets in $\Bbb C$ such that $V \subset U$. We say $(U, V)$
has $K$-bounded geometry if there exist $x \in V$ and $r > 0$ such
that
$$
B_{r}(x) \subset V \subset U \subset B_{Kr}(x).
$$
\end{defi}
\begin{lem}\label{covering lemma}
Let $K > 1$ and $L = 8K + 9$. Then  for any finite family of pairs
of measurable sets $\{(U_{i}, V_{i})\}_{i \in \Lambda}$ all of which
have $K$-bounded geometry, namely, for each $i \in \Lambda$, there
exist $x_i \in V_i$ and $r_i
> 0$ satisfying  \begin{equation}\label{m-qq}
B_{r_i}(x_i) \subset V_i \subset U_i \subset
B_{Kr_i}(x_i),\end{equation}  there is a subfamily $\sigma_{0}$ of
$\Lambda$ such that all $B_{r_{j}}(x_{j}), j \in \sigma_{0}$, are
disjoint, and moreover,
$$
\bigcup_{i \in \Lambda} U_{i} \subset \bigcup_{j \in \sigma_{0}}
B_{Lr_{j}}(x_{j}).
$$
In particular, we have
$$
m \big{(}\bigcup_{i \in \Lambda} U_{i}\big{)} \le L^2 \cdot
m\big{(}\bigcup_{i \in \Lambda} V_{i}\big{)}
$$ where $m(\cdot)$ denotes the area with respect to the Euclidean
metric.
\end{lem} It is worth to note that the last assertion
of Lemma~\ref{covering lemma} is not true if we consider  spherical
area instead of Euclidean area. For instance, the pair $(B_{2R}(R),
B_{R/2}(R))$ has $K$-bounded geometry with $K = 4$. But as $R \to
\infty$, the spherical area of $B_{2R}(R)$ goes to $\infty$ and the
spherical area of $B_{R/2}(R)$ goes to $0$. Since we consider
spherical area in this paper, we need a particular variant of the
above lemma. Let  ${\rm diam}(\cdot)$ and ${\rm dist}(\cdot, \cdot)$
denote respectively  the   diameter  and the distance with respect
to the Euclidean metric.  Let  $area(\cdot)$  denote the area  with
respect to the spherical metric. Let $\Omega = {\Bbb C} \setminus
\overline{\Delta}$.
\begin{cor}\label{spherical area}{\rm
Let $K > 1$. Suppose $\{(U_{i}, V_{i})\}_{i \in \Lambda}$ is a
finite family of pairs of measurable sets   in $\Omega$ all of which
have $K$-bounded geometry.  If in addition
\begin{equation}\label{s-e}
{\rm diam}(U_{i})< K \cdot {\rm dist}(U_i, \Bbb T)\end{equation}
holds for every $i \in \Lambda$, then
$$
area(\bigcup_{i\in\Lambda}U_{i}) \le  \lambda(K)\cdot
area(\bigcup_{i\in\Lambda} V_{i})
$$
where $1< \lambda(K) < \infty$ is a constant depending only on $K$.}
\end{cor}
The proofs of Lemma~\ref{covering lemma} and
Corollary~\ref{spherical area} will be given in $\S8.2$.

 The second ingredient of the idea  is the concept of
``$K$-admissible pair''.
\begin{defi}\label{admiss}{\rm Let $1 < K  < \infty$ and $z \in
X_{n+2}$.  We say $z$ is associated to a $K$-admissible pair $(U,
V)$
 if $V \subset U \subset \Omega$ are two
open topological disks such that $z \in U$ and
\begin{itemize}
\item[1.] $V \subset X_{n}\setminus X_{n+2}$,
\item[2.] the pair $(U, V)$ has $K$-bounded geometry,
\item[3.] ${\rm diam}(U)< K \cdot {\rm dist}(U, \Bbb T)$.
\end{itemize}}
\end{defi}

The third ingredient of the idea  is a group of dynamically defined
domains.

Let $I \subset \Bbb T$ be an open interval. Let $\Bbb C^{*} = \Bbb C
\setminus \{0\}$ be the punctured plane. Set
\begin{equation}\label{slt}
\Omega_{I} = {\Bbb C}^{*} \setminus ({\Bbb T} \setminus I).
\end{equation}Then $\Omega_I$ is a hyperbolic Riemann surface.  For $d > 0$,
the hyperbolic neighborhood of $I$ is defined by
\begin{equation}\label{dhs}
\Omega_{d}(I) = \{z \in \Omega_{I}\:\big{|}\: d_{\Omega_{I}}(z, I) <
d\} \end{equation} where $d_{\Omega_{I}}(\cdot, \cdot)$ denotes the
hyperbolic distance in $\Omega_{I}$. We will show
\begin{lem}\label{hn-ps}
Let $d > 0$ be given.  Then when $I$ is small, $\Omega_{d}(I)$ is a
Jordan domain and is like the hyperbolic neighborhood of the slit
plane:
\begin{itemize}
\item[1.] $\partial \Omega_{d}(I) = \gamma_{int} \cup \gamma_{out}$ where $\gamma_{int}$ and $\gamma_{out}$
are real analytic curve segments both of which connect the two end
points of $I$. Moreover, $\gamma_{int}\setminus \partial I \subset
\Delta$ and $\gamma_{out}\setminus \partial I \subset \Bbb C
\setminus \overline{\Delta}$;
\item[2.]
$\gamma_{int}$ and $\gamma_{out}$ are symmetric about $\Bbb T$, and
each of them is like an arc segment of some Euclidean circle;
\item[3.] let $\alpha$ denote the exterior angles formed by $\gamma_{int}$ and $\Bbb T$,
$\gamma_{out}$ and $\Bbb T$, all of which are the same, then $ d =
\log \cot (\alpha/4). $\end{itemize}
\end{lem}

Note that  $\Omega_d(I)$ is divided by $I$ into two parts: one is in
the interior of $\Delta$ and the other one is in the exterior of
$\Delta$.  We only consider  the part which is in the exterior of
$\Delta$. We use $H_\alpha(I)$ to denote this part.  That is,
\begin{equation}\label{hndd}
H_\alpha(I) = \{z \in \Omega_d (I)\:|\: |z| > 1\}
\end{equation}
 where $\alpha$ is
determined by the formula $d = \log \cot (\alpha/4)$.
\begin{figure}
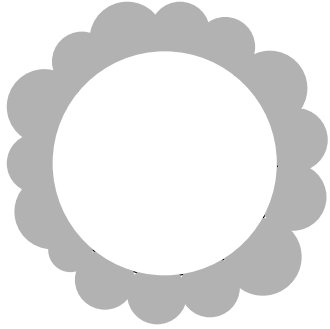

\bigskip
\begin{center}
\centertexdraw { \drawdim cm \linewd 0.02 \move(-2 1)

\move(-1.1 -1.75) \lcir  r:1.5

 \move(-2.8 -1.75) \fcir f:0.7 r:0.4
 \move(-2.7 -1.) \fcir f:0.7 r:0.5
  \move(-2.2 -0.4) \fcir f:0.7 r:0.4

  \move(-1.6 -0.1) \fcir f:0.7 r:0.5
  \move(-0.9 0) \fcir f:0.7 r:0.4
    \move(-0.3 -0.2) \fcir f:0.7 r:0.4
   \move(0.3 -0.75) \fcir f:0.7 r:0.5
  \move(0.64 -1.44) \fcir f:0.7 r:0.42
\move(0.6 -2.2) \fcir f:0.7 r:0.45

\move(0.2 -3) \fcir f:0.7 r:0.52

\move(-0.5 -3.4) \fcir f:0.7 r:0.4

\move(-1.2 -3.5) \fcir f:0.7 r:0.4

\move(-1.9 -3.3) \fcir f:0.7 r:0.4

\move(-2.6 -2.4) \fcir f:0.7 r:0.5

\move(-2.35 -2.9) \fcir f:0.7 r:0.3

\move(-1.1 -1.75) \fcir f:1 r:1.49

}
\end{center}
\vspace{0.2cm} \caption{The set $Z_n$}
\end{figure}

Recall that $g_{\theta}|\Bbb T: \Bbb T \to \Bbb T$ is a critical
circle mapping with rotation number $\tau$. Let $[a_1, \cdots, a_n,
\cdots]$ be the continued fraction of $\tau$ and $p_n/q_n$, $n \ge
0$, be the truncated continued fractions of $\tau$.  For $i \ge 0$,
let $x_i \in \Bbb T$ be the point such that $g_\theta^i(x_i) = 1$.
Take  $$0<\alpha < \pi/3$$ and let it  be fixed throughout the
paper. For $n
> 0$, let
$$
I_{n} = [1, x_{q_{n}}], \:I_{n+1} = [1, x_{q_{n+1}}].
$$
For $0 \le i \le q_{n+1} -1$, let $I_{n}^{i}\subset \Bbb T$ denote
the interval such that $g_{\theta}^{i} = I_{n}$. For $0\le i \le
q_{n} -1$, let $I_{n+1}^{i}\subset \Bbb T$ denote the interval such
that $g_{\theta}^{i}(I_{n+1}^{i}) = I_{n+1}$.  It is known that the
collection of the intervals $$I_n^i, 0 \le i \le q_{n+1}-1; \:\:
I_{n+1}^{i}, 0 \le i \le q_{n} -1,
$$ form a partition of $\Bbb T$. In $\S4$ we call this partition  the dynamical partition of level
$n$.

Define
\begin{equation}\label{ZN}
Z_{n}  =  \bigcup_{0\le i\le q_{n+1}-1} H_{\alpha}(I_{n}^{i}) \cup
\bigcup_{0\le i\le q_{n}-1} H_{\alpha}(I_{n+1}^{i}). \end{equation}
It is easy to see that $Z_n$ is the outer half of an open
neighborhood of $\Bbb T$. See Figure 1 for an illustration. As a
direct consequence of Lemma~\ref{hn-ps} and the real bounds for the
intervals of the dynamical partition (cf. Theorem~\ref{SH}),  it
follows  that there exist $C
> 1$ and $0 < \epsilon < 1$ such that
\begin{equation}\label{artes}
area(Z_n) < C \cdot \epsilon^n.
\end{equation}

By Corollary~\ref{spherical area} and (\ref{artes}), we reduce the
Main Lemma in Proposition~\ref{mred}  to
\begin{lem}\label{pre-lem}
There exist  $K > 1$  and $N_1 \ge N_0$ such that for all $n \ge
N_1$, if
 $z \in X_{n+2}$,  then either $z \in Z_{n}$, or $z$ is associated to
some $K$-admissible pair $(U, V)$.
\end{lem}

The proof of Lemma~\ref{pre-lem} is the main task of the paper and
is postponed until the end of the paper.   The proof of Theorem 1 is
completed once  Lemma~\ref{pre-lem} has been proved.


\section{David Homeomorphisms and David's integrability Theorem}
Let $\Omega \subset \widehat{\Bbb C}$ be a domain.  A Beltrami
differential $\mu = \mu(z)d\overline{z}/dz$ in $\Omega$ is a
measurable $(-1, 1)$-form such that $|\mu(z)| < 1$ almost everywhere
in $\Omega$. We  say $\mu$ is $\emph{integrable}$ if there is a
homeomorphism $\phi: \Omega \to \Omega'$ in $W_{\rm
loc}^{1,1}(\Omega)$ which solves the Beltrami equation
\begin{equation}\label{be}
\overline{\partial} \phi = \mu \partial \phi.
\end{equation}
The map $\phi$ is called a David homeomorphism. When
$\|\mu\|_{\infty} < 1$, the map $\phi$ is the classical
quasiconformal mapping.

Recall that $area(X)$ is used to denote the spherical area of a
subset $X \subset \widehat{\Bbb C}$.
\begin{thm}[David \cite{Da}]\label{David}
Let $\Omega \subset \widehat{\Bbb C}$ be a domain. Let $\mu$ be a
Beltrami differential in $\Omega$. Then $\mu$ is integrable if there
exist constants $M > 0$,  $\alpha > 0$, and $0< \epsilon_{0} < 1$,
such that for any $0< \epsilon < \epsilon_{0}$, the following
inequality holds,
$$
area \{z\:\big{|}\: |\mu(z)| > 1 - \epsilon \} \le M
e^{-\alpha/\epsilon}.
$$
Moreover, if $\mu$ is integrable, up to postcomposing a conformal
map, there is a unique solution $\phi: \Omega \to \Omega'$ in
$W_{\rm loc}^{1,1}(\Omega)$ which solves the Beltrami
equation~(\ref{be}). That is, if $\psi: \Omega\to \Omega''$ is
another such solution, then there is a conformal map $\sigma:
\Omega' \to \Omega''$ such that $\psi = \sigma \circ \phi$.
\end{thm}

\section{Real bounds} The materials in this section are  standard. We
provide them here just for the convenience of the readers. Our
presentation follows $\S2$  and the appendix of \cite{PZ}.

Let $f: {\Bbb T} \to {\Bbb T}$ be a real analytic homeomorphism
which preserves the orientation. We say $f$ is a $\emph{critical
circle mapping}$ if it has exactly one critical point at $1$.

Suppose $f$ is a critical circle mapping with an irrational rotation
number $\theta$. Let $a_i, i \ge 0$, be all the coefficients of the
continued fraction of $\theta$ and  $p_{n}/q_{n}, n \ge 0$, be all
the truncated continued fractions of $\theta$. By definition, we
have
$$
p_0 = 0, p_1 = 1;\:\: q_0 = 1, q_1 = a_1,
$$
and for all $n \ge 2$, we have
$$
p_n = a_n p_{n-1} + p_{n-2} \hbox{ and } q_n = a_n q_{n-1} +
q_{n-2}.
$$

Let $x_{0} = 1$. For $i = 1, 2, \cdots$, let $x_{i}\in \Bbb T$
denote backward iterate $f^{-i}(1)$, that is, the point in $\Bbb T$
such that $f^{i}(x_{i}) = 1$, and let $x_{-i}$ denote the forward
iterate $f^{i}(1)$.  Let $I_{n} = [1, x_{q_{n}}]$ denote the $n$-th
closest return interval under $f^{-1}$. For $i \ge 0$, let
$I_{n}^{i} \subset \Bbb T$ denote the interval such that
$f^{i}(I_{n}^{i}) = I_{n}$, that is, $I_{n}^{i} = [x_{i}, x_{q_{n}
+i}]$. Then the collection of the intervals
$$
I_{n}^{i}, \: 0 \le i \le q_{n+1}-1, \hbox{ and } I_{n+1}^{i}, \: 0
\le i\le q_{n}-1,
$$
defines a partition of $\Bbb T$ modulo the common end points. We
call such a partition a $\emph{dynamical partition}$ of level $n$.
It is not difficult to see that the set of all the end points in
this partition is
$$
\Pi_{n} = \{x_{i}\:\big{|} \:\: 0 \le i  \le  q_{n} + q_{n+1}-1\}.
$$

Let $K > 1$.  We say two intervals $I$ and $J$ are
$K$-$\emph{commensurable}$ if $|J|/K < |I| < K\cdot |J|$.
\begin{thm}[\'{S}wiatek-Herman, see \cite{dFdM}, \cite{P2} and \cite{PZ}]\label{SH}
Let $f: {\Bbb T} \to {\Bbb T}$ be a   critical circle mapping with
an irrational rotation number $\theta$.Then there is a constant $1 <
K < \infty$ depending only on $f$ such that for any $n \ge 1$,
\begin{itemize}\item[1.] any two adjacent intervals $I$ and $J$ in the dynamical
partition of $\Bbb T$ of level $n$  are $K$-commensurable, and
\item[2.]  the two intervals  $[x,
f^{q_{n}}(x)]$ and $[x, f^{-q_{n}}(x)]$ are $K$-commensurable  for
any point $x \in \Bbb T$.
\end{itemize}
\end{thm}
 Now let us consider another partition of $\Bbb T$. Let
$$
\Xi_{n}  = \{x_{i}\:\big{|}\: 0 \le i \le  q_{n+1}-1\}.
$$

The points in $\Xi_{n}$ separate $\Bbb T$ into disjoint intervals.
This partition arises in Yoccoz's cell construction(see the appendix
of \cite{PZ} or $\S6$ of this paper).  We  call it the $\emph{cell
partition}$ of level $n$. The following  lemma describes the basic
relation between the two partitions. The proof of the  lemma can be
found in the appendix of \cite{PZ}. For the convenience of the
readers, we present the proof here.
\begin{lem}\label{dyn-cell}
Each interval in ${\Bbb T}\setminus \Xi_{n}$ is either an interval
in ${\Bbb T}\setminus \Pi_{n}$ or the union of two adjacent
intervals in ${\Bbb T}\setminus \Pi_{n}$.
\end{lem}
\begin{rmk}{\rm
Note that the definition of the cell partition  here is a little
different from that in \cite{PZ} where the cell partition of level
$n$ is defined by the set of points $\mathcal{Q}_n = \{x_{i}\:
\big{|}\:\: 0 \le i \le q_{n}-1\}$. Therefore, the cells of level
$n$ in this paper correspond to the cells of level $n+1$ there.}
\end{rmk}
\begin{proof}
Assume that $[x_i, x_j]$ is an interval component of $\Bbb T
\setminus \Xi_{n}$. It suffices to prove that the interior of $[x_i,
x_j]$ contains at most one point in $\Pi_n$. Let us prove this by
contradiction. Suppose the interior of $[x_i, x_j]$ contains more
than one point in $\Pi_n$. Then we can take two  points $x_l$ and
$x_m$ in $\Pi_n$ which  are contained in the interior of $[x_i,
x_j]$ and are adjacent in $\Pi_n$.  Without loss of generality, we
may assume that $l < m$. Since both $x_l$ and $x_m$ are not in
$\Xi_n$, it follows that $m > l \ge q_{n+1}$. Since $x_l$ and $x_m$
are two adjacent points in $\Pi_n$, we have either $m - l = q_n$ or
$m- l = q_{n+1}$. This implies that
$$
m = l + (m-l) \ge q_{n+1} + q_n .
$$
But since $x_m \in \Pi_n$, we have $0 \le m \le q_n + q_{n+1}-1$.
This is a contradiction.  The proof of Lemma~\ref{dyn-cell} is
completed.

\end{proof}

  As an immediate
consequence of Theorem~\ref{SH} and Lemma~\ref{dyn-cell}, we have
\begin{cor}\label{c-g}
There is a constant $1 < K < \infty$ depending only on $f$ such any
two adjacent interval components in the cell partition of level $n$
are $K$-commensurable.
\end{cor}

\begin{lem}\label{geom}
There is a  $0< \delta < 1$ which depends only on $f$ such that for
any interval $I$ in ${\Bbb T} \setminus \Xi_{n+2}$, there is some
interval $J$ in ${\Bbb T} \setminus \Xi_{n}$ with $I \subset J$ and
$|I| < \delta |J|$.
\end{lem}
\begin{proof}
Let $I$ be an interval component of  ${\Bbb T} \setminus \Xi_{n+2}$.
Let $S$ and $J$ be the interval components of ${\Bbb T} \setminus
\Xi_{n+1}$  and ${\Bbb T} \setminus \Xi_{n}$ respectively  such that
$I \subset S \subset J$.

By Lemma~\ref{dyn-cell}, we have two cases. In the first case, $J$
is the union of two adjacent interval components of $\Bbb T
\setminus \Pi_n$.  In the second case, $J$ is  an interval component
of $\Bbb T \setminus \Pi_n$.

Suppose we are in the first case, that is, $J$ is the union of two
adjacent interval components of $\Bbb T \setminus \Pi_n$. Let us
denote these two adjacent interval components by $L$ and $R$. Since
$\Xi_{n+1} \supset \Pi_n$,  $S$ is contained either in $L$ or in
$R$. By Theorem~\ref{SH}, $|L| \asymp |R|$. So there is a uniform $0
< \delta < 1$ such that
$$
|I| \le |S| \le \max\{|L|, |R|\} < \delta |J|.
$$
The lemma follows in the first case.

Now suppose we are in the second case, that is,  $J$ is an interval
component of $\Bbb T \setminus \Pi_n$.  By the definition of $\Pi_n$
and $\Xi_n$,  it follows that $J = I_{n}^{i}$ for some $0 \le i <
q_{n+1} - q_n$.  Then    $J  = [x_i, x_{i + q_{n}}]$. Note  that
$x_{i + q_{n} + q_{n+1}}$ belongs to the interior of $[x_i,
x_{i+q_n}] = J$.  Since
$$
0 < i + q_{n} + q_{n+1} < q_{n+1} - q_n  + q_{n} + q_{n+1} = 2
q_{n+1} <  q_{n+2}+ q_{n+1} \le q_{n+3},
$$
it follows that
$$
x_{i + q_{n} + q_{n+1}} \in \Pi_{n+1} \subset \Xi_{n+2}.
$$
Note that  $I$ is an interval component of $J \setminus \Xi_{n+2}$.
Since $\Pi_{n+1} \subset \Xi_{n+2}$,  $I$ is contained in an
interval component of $J \setminus \Pi_{n+1}$. Since $x_{i + q_{n} +
q_{n+1}}$ belongs to the interior of $[x_i, x_{i+q_n}] = J$ and
$x_{i + q_{n} + q_{n+1}} \in \Pi_{n+1}$, there are at least two
interval components in $J \setminus \Pi_{n+1}$. Let us label these
interval components by order as $$J_{1}, \cdots, J_{l},$$ where $l
\ge 2$ is some integer. By Theorem~\ref{SH}, we have
$$
|J_{j}| \asymp |J_{j+1}|
$$
holds for all $1 \le j \le l-1$. It follows that there is a uniform
$0 < \delta < 1$ such that $|J_{j}| < \delta |J|$ for all $1 \le
j\le l$. Since $I$ is contained in  $J_{j}$ for some $1 \le j \le
l$, the lemma  follows in the second case. The proof of
Lemma~\ref{geom} is completed.
\end{proof}

 Recall that for $i = 1, 2, \cdots$, we use $x_i$ to denote the
backward iterate $f^{-i}(1)$ and $x_{-i}$ to denote the forward
iterate $f^{i}(1)$.
\begin{lem}\label{real bound}
Let $v = x_{-1} = f(1)$ denote the critical value of $f$.  Then

\begin{itemize}
\item[1.] $|[x_{q_{n+1}-1}, v]|\asymp |[v, x_{q_{n}+q_{n+1}-1}]|\asymp |I_{n}^{q_{n+1}-1}|$,
\item[2.] if  $I$ is an interval in the dynamical partition or cell partition of
of level $n$ such that  $v \notin I$, then  ${\rm dist}(v, I)
\succeq |I|$.
\end{itemize}
\end{lem}
\begin{proof}
 Note that $I_{n}^{q_{n+1} -1} = [x_{q_{n+1}-1}, v] \cup [v, x_{q_{n}+q_{n+1}-1}]$. Since  $[x_{q_{n+1}-1}, v] = f(I_{n+1})$ and
$[v, x_{q_{n}+q_{n+1}-1}] = f([1, x_{q_{n}+q_{n+1}}])$, to prove the
first assertion, it suffices to prove $|I_{n+1}|  \asymp |[1,
x_{q_{n}+q_{n+1}}]|$.    But this is obvious since $I_{n+2} \subset
[1, x_{q_{n}+q_{n+1}}] \subset I_{n}$ and $|I_{n+2}| \asymp
|I_{n+1}|\asymp |I_n|$ (cf. Theorem~\ref{SH}). This proves the first
assertion of Lemma~\ref{real bound}.

Now let us prove the second assertion. First suppose $I$ is an
interval of the dynamical partition of level $n$. There are two
cases. In the first case, $I$ and $I_{n}^{q_{n+1}-1}$ are adjacent
to each other. Then
$$
{\rm dist}(v, I) = \min\{|[x_{q_{n+1}-1}, v]|, |[v,
x_{q_{n}+q_{n+1}-1}]|\}.
$$
This, together with the first assertion we just proved, implies that
${\rm dist}(v, I) \asymp |I_{n}^{q_{n+1}-1}|$. Since $I$ and
$I_{n}^{q_{n+1}-1}$ are adjacent to each other by assumption, we
have $|I| \asymp |I_{n}^{q_{n+1}-1}|$. Thus ${\rm dist}(v, I)
\asymp|I|$.  In the second case, $I$ is not adjacent to
$I_{n}^{q_{n+1}-1}$. Let $S$ be the smaller component of $\Bbb T
\setminus (I \cup I_{n}^{q_{n+1}-1})$. Let $H$ denote the interval
in the dynamical partition of level $n$ which is adjacent to $I$ and
is contained in $S$. Then
$$
{\rm dist}(v, I) \succeq |S| \ge H \asymp |I|.
$$
This proves the second assertion in the case that $I$ is an interval
of the dynamical partition of level $n$.

Now suppose  $I$ is an interval of the cell partition of level $n$.
By Lemma~\ref{dyn-cell}, $I$ is either an interval in the dynamical
partition of level $n$ or is the union of two adjacent intervals of
dynamical partition of level $n$. The first case has just been
proved. Suppose it is the second case. That is, $I$ is the union of
two adjacent intervals in the dynamical partition of level $n$, say
$R$ and $S$. Then
$$
{\rm dist}(v, I) = \min\{{\rm dist}(v, R), {\rm dist}(v, S)\}.
$$
Since $R$ and $S$ are intervals of dynamical partition of level $n$,
we have  ${\rm dist}(v, R) \succeq |R|$ and ${\rm dist}(v, S)
\succeq |S|$. Since $|R| \asymp |S| \asymp |I|$, from the above
equation we get  ${\rm dist}(v, I) \succeq |I|$. This proves the
second assertion.

 The proof of Lemma~\ref{real bound} is completed.
\end{proof}


\section{The  premodel $G_{\theta}$ and the immediate premodel $g_{\theta}$}

In this section, we construct a premodel for $f_\theta$. The idea of
such type of construction was pioneered by A. Cheritat (see
\cite{C1}). Recall that $\Delta$ and $\Bbb T$ denote the unit disk
and the unit circle, respectively.

Let $T(z) = \sin(z)$. It follows that the map $T(z)$ has exactly two
critical values $1$ and $-1$. Let $D$ be the component of
$T^{-1}(\Delta)$ which contains the origin.  Let $$ R = \{x +
iy\:|\: -\pi/2 < x < \pi/2\:\hbox{  and } -100 < y < 100\}.
$$
\begin{lem}\label{basic domain}
$D$ is a Jordan domain which is symmetric about the origin and the
map $T|\partial D: \partial D \to {\Bbb T}$ is a homeomorphism.
Moreover, $D \subset R$ and $\partial D \cap \partial R = \{-\pi/2,
\pi/2\}$. In particular, $\pi/2$ and $-\pi/2$ are the only two
critical points of $T$ which are contained in $\partial D$.
\end{lem}
\begin{proof}
Let $K = (e^{100} + e^{-100})/2$.  Let us study what $T(R)$ looks
like. When $z$ runs through the straight segment $[\pi/2 + 100 i,
\pi/2]$ from $\pi/2 + 100 i$ to $\pi/2$, $T(z)$ runs through the
straight segment $[K, 1]$ from $K$ to $1$. When $z$ runs through the
straight segment $[\pi/2, \pi/2 - 100 i]$, $T(z)$ runs through the
straight segment $[1,  K]$ from $1$ to $K$. When $z$ runs through
the straight segment  $[\pi/2 - 100i, -\pi/2 - 100i]$ from $\pi/2 -
100i$ to $-\pi/2 - 100i$, $T(z)$ runs through a continuous curve
segment $\Gamma_{-}$ in the lower half plane where
$$
\Gamma_{-} = \{ \sin (t - 100i)\:|\: \pi/2 \ge   t \ge -\pi/2\}
$$
looks like the lower half circle $\{z\:|\: |z| = K \hbox{ and } 0\ge
\arg(z) \ge -\pi\}$. Note that for $\pi/2 \ge t_1 > t_2 \ge -\pi/2$,
$$\sin(t_1 - 100i) - \sin(t_2 - 100 i) = 2 \sin (\frac{t_1 -
t_2}{2}) \cos(\frac{t_1+ t_2}{2} - 100 i).$$  Since $0< \frac{t_1 -
t_2}{2} \le \pi/2$, the first factor on the right hand of the above
equation is not equal to zero. It is clear that the  absolute value
of the second factor is large. Thus $\sin(t_1 - 100i) \ne \sin(t_2 -
100 i)$. So the curve segment $\Gamma_-$ does not intersect with
itself.

Note that  $T$ is odd. So  when $z$ runs through the straight
segment $[-\pi/2 - 100 i, -\pi/2]$, $T(z)$ runs through the straight
segment $[-K, -1]$ from $-K$ to $-1$; and when $z$  runs through the
straight segment $[-\pi/2, -\pi/2 + 100 i, ]$, $T(z)$ runs through
the straight segment $[ -1, -K]$ from $-1$ to $-K$; and when  $z$
runs through the straight segment  $[-\pi/2 + 100i, \pi/2 + 100i]$
from $-\pi/2 + 100i$ to $\pi/2 +100i$, $T(z)$ runs through a
continuous curve segment $\Gamma_{+}$ in the upper half plane where
$$
\Gamma_{+} = \{ \sin (t + 100i)\:|\:  -\pi/2  \le   t \le \pi/2\}
$$
looks like the upper half circle $\{z\:|\: |z| = K \hbox{ and } \pi
\ge  \arg(z) \ge 0\}$ and does not intersect with itself.

Let  $\Gamma = \Gamma_{+} \cup \Gamma_{-}$. Then $\Gamma$ is a
Jordan curve which is like the Euclidean circle  with center at the
origin and  radius $K$. Let $\Omega$ be the domain bounded by
$\Gamma$ and containing the origin. Let
$$
\Sigma = \Omega \setminus ([1, K] \cup [-K, -1]).
$$
From the above argument, we see that when $z$ runs through $\partial
 R$ one time, $T(z)$ runs through $\partial \Sigma$ one time also.
 It follows that
 $$
 T: R \to \Sigma
 $$
 is a holomorphic homeomorphism. Since $\Delta \subset \Sigma$ with
  $ \partial \Delta \cap \partial \Sigma=\{1, -1\}$,
 there is a Jordan domain, say $D \subset R$,  such that $T: D \to
 \Delta$ is a homeomorphism, and moreover,  $$\partial D \cap \partial R=\{\pi/2, -\pi/2\}.$$  Since $T(0) = 0$, we have   $0 \in D$. Since $T$ is odd and $\Delta$ is symmetric about the origin,
 $D$ must be symmetric about the origin, that is, $D = -D$.    The holomorphic isomorphic  $T: D \to \Delta$ can be homeomorphically extended to $\partial D$. Thus
  $T: \partial D
 \to \partial \Delta$ is a homeomorphism. Since $D \subset R$,   $\pi/2$ and $-\pi/2$
 are  the only critical points of $T$ which are contained in
 $\partial D$.  The proof of Lemma~\ref{basic domain} is completed.
\end{proof}

For $k \in \Bbb Z$, let $D_{k} = \{z+k\pi\big{|} z \in D\}$  and
$R_{k} =  \{z+k\pi\big{|} z \in R\}$.  Then  $D_{0} = D$, $R_{0} =
R$ and $D_k \subset R_k$ for all $k \in \Bbb Z$.   As a direct
consequence of Lemma~\ref{basic domain} we have
\begin{cor}\label{dis-con}
The domains $D_{k}, k \in {\Bbb Z}$,  are all the components of
$T^{-1} (\Delta)$. For any $k \in \Bbb Z$, $\partial D_{k} \cap
\partial D_{k+1} = \{k \pi+ \pi/2\}$, and moreover,  $\partial D_{i}
\cap
\partial D_{j}= \emptyset$ for all $i, j \in {\Bbb Z}$ with $|i - j| >
1$.
\end{cor}

Let $\psi: \widehat{\Bbb C} - \overline{\Delta} \to \widehat{\Bbb C} - \overline{D}$
be the Riemann map such that $\psi(\infty) = \infty$ and $\psi(1) = \pi/2$. Since $\Delta$
and $D$ are both symmetric about the origin, we have
\begin{lem}\label{odd}
$\psi$ is odd.
\end{lem}
For $z \in \widehat{\Bbb C}$, let $z^{*}$ denote the symmetric image
of $z$ about the unit circle. That is, $z^{*} = 1/\bar{z}$ if  $z
\ne 0$ and $z^{*} = \infty$ if $z = 0$. Define
\begin{equation}\label{ff-rr}
G(z) =
\begin{cases}
 T \circ \psi(z) & \text{for $z \in {\Bbb C} - \Delta$}, \\
[(T \circ \psi)(z^{*})]^{*} & \text{ for $z \in \Delta -\{0\}$}.

\end{cases}
\end{equation}

From (\ref{ff-rr}) it follows that $G$ is   meromorphic in ${\Bbb
C}^{*}$ and has exactly  two essential singularities at $0$ and
$\infty$.  Let us show that  $G$ has no asymptotic values.  If this
were not true, then there would be a curve $\gamma: [0, \infty) \to
{\Bbb C}^{*}$ such that $\gamma(t) \to 0$ or $\infty$ and
$G(\gamma(t))$ converges to some complex number. Since $G$ is
symmetric about the unit circle, we may assume that $\gamma(t) \to
\infty$ as $t \to \infty$. Since $\psi$ maps $\infty$ to $\infty$,
it follows that $T$ has an asymptotic value. But this is a
contradiction since $T$ has no asymptotic values (cf. Lemma 1 of
\cite{Zh1}). So $G$ has no asymptotic values.

From (\ref{ff-rr}) it follows that $G$ is holomorphic in $\Bbb C
\setminus \overline{\Delta}$ and thus has no poles in $\Bbb C
\setminus \overline{\Delta}$.

Let us  determine the set of  critical points of $G$ as follows.
Note that the set of critical points of $T$ is given by $\{k\pi +
\pi/2\:|\: k \in \Bbb Z\}$, and except  $\pi/2$ and $-\pi/2$ which
are contained in $\partial D$, all the other critical points of $T$
are contained in $\Bbb C \setminus \overline{D}$.  So the set
$$
\{\psi^{-1}(k\pi + \pi/2)\:|\: k\ne 0, -1, k \in \Bbb Z\}
$$
contains all the critical points of $G$ in $\Bbb C \setminus
\overline{\Delta}$, and its symmetric image, the set
$$
\{[\psi^{-1}(k\pi + \pi/2)]^{*}\:|\: k\ne 0, -1, k \in \Bbb Z\}
$$
contains  all the critical points of $G$ in $\Delta \setminus
\{0\}$. Since $1 = \psi^{-1}(\pi/2)$ and $-1 = \psi^{-1}(-\pi/2)$,
it follows that $-1$ and $1$ are the only two critical points of $G$
in the unit circle. From (\ref{ff-rr}) it follows that $G: \Bbb T
\to \Bbb T$ is an analytic homeomorphism. Let us summarize these
into the following lemma.

\begin{lem}\label{G}
$G: \Bbb C^{*} \to \widehat{\Bbb C}$  is an odd meromorphic function
such that
\begin{itemize}
\item[1.] $G$ is symmetric about the unit circle, that is,
 $G (z^{*}) = [G(z)]^{*}$  for all $z \in \Bbb C^{*}$, \item[2.] $G$ has two essential singularities at  $0$ and
$\infty$  and has no asymptotic values, and moreover, $G$ has no
poles outside the unit disk,

\item[3.] the set of the critical points of $G$ in $\Bbb C \setminus \overline{\Delta}$
 is given by $ \{\psi^{-1}(k\pi + \pi/2)\:|\: k\ne 0, -1, k \in
\Bbb Z\}$, and the set of the critical points of $G$ in $\Delta
\setminus \{0\}$  is given by  $ \{[\psi^{-1}(k\pi +
\pi/2)]^{*}\:|\: k\ne 0, -1, k \in \Bbb Z\}$. The only two critical
points of $G$ in the unit circle are $-1$ and $1$ both of which are
double critical points,
\item[4.] $G(1) = 1$ and $G(-1) = -1$ are the only two critical
values of $G$, that is, all the other critical points of $G$ are
mapped either to $1$ or to $-1$,
\item[5.]  $G|{\Bbb T}: {\Bbb T}\to {\Bbb T}$ is a
real analytic circle homeomorphism.
\end{itemize}
\end{lem}

Let $0< \theta< 1$ be the David type irrational number in Theorem 1.
Since $G|{\Bbb T}: {\Bbb T} \to {\Bbb T}$ is a critical circle
homeomorphism, by Proposition 11.1.9 of \cite{KH}, we get
\begin{lem}\label{exist}
There exists a unique $t \in [0, 1)$ such that $e^{2 \pi it} \cdot
G|{\Bbb T}: {\Bbb T}\to {\Bbb T}$ is a critical circle homeomorphism
of rotation number $\theta$.
\end{lem}
Let $t \in [0, 1)$ be the number given in Lemma~\ref{exist}. Let us
define a meromorphic function $G_{\theta}: \Bbb C^{*} \to
\widehat{\Bbb C}$ by setting
\begin{equation}\label{ff-ll}
G_{\theta}(z) = e^{2 \pi it} \cdot G(z) \quad \forall z \in \Bbb C
^{*}.
\end{equation}

 By Yoccoz's linearization theorem, it follows that
 $G_{\theta}|\Bbb T: \Bbb T \to \Bbb T$ is homeomorphically conjugate to
 the rigid rotation $z \mapsto e^{2 \pi i \theta} z$.
 As we mentioned in the Introduction of
the paper, since $G_{\theta}$ has two critical points on the unit
circle,  Yoccoz's cell construction, which is presented in the
appendix of \cite{PZ}, may not be used directly to extend the
conjugacy map to a David homeomorphism of the  unit disk to itself.
To  avoid this problem, we will introduce an immediate premodel map
$g_{\theta}$ as follows.

Let $\Phi: {\Bbb C} \to {\Bbb C}$ be the square map given by
$\Phi(z) = z^{2}$.  Note that for each $z$ in $\Bbb C$,  since
$G_{\theta}$ is odd, the value $\Phi \circ G_{\theta} \circ
\Phi^{-1}(z)$ does not depend on the choice of the branch of
$\Phi^{-1}$ at $z$. Thus we can define a meromorphic function
$g_{\theta}: \Bbb C^{*} \to \widehat{\Bbb C}$  by setting
\begin{equation}\label{easier-m}
g_{\theta}(z) =  \Phi \circ G_{\theta} \circ \Phi^{-1}(z).
\end{equation}
\begin{lem}\label{rotation number}
$g_{\theta}$ is a meromorphic function such that
\begin{itemize}
\item[1.] $g_{\theta}$ is symmetric about the unit circle, that is,
 $g_\theta (z^{*}) = [g_\theta(z)]^{*}$  for all $z \in \Bbb C^{*}$,
 \item[2.] $g_{\theta}$ has two essential singularities at  $0$ and
$\infty$  and has no asymptotic values, and moreover, $g_{\theta}$
has no poles outside the unit disk, \item[3.] a point $\omega \in
\Bbb C^{*}$ is a critical point of $g_\theta$  if and only if it is
the  $\Phi$-image of either a critical point, or a pole, or a zero
of $G$,

\item[4.] the point $1$ is the only critical point of $g_{\theta}$
in the unit circle, and it is a double critical point. Moreover,
$g_\theta$ has exactly three critical values: $0$, $\infty$ and
$g_\theta(1)$. In particular, $g_{\theta}(1)$ is the unique critical
value of $g_{\theta}$ in $\Bbb C^{*}$,
\item[5.]
$g_{\theta}|\Bbb T: \Bbb T \to \Bbb T$ is a critical circle
homeomorphism with rotation number  $\tau \equiv 2 \theta \mod(1)$.
\end{itemize}
\end{lem}
\begin{proof}
The first four assertions follow  directly from Lemma~\ref{G} and
the construction of $g_{\theta}$.  The only point one needs to think
about is that the $\Phi$-images of the zeros and poles of
$G_{\theta}$ are critical points of $g_{\theta}$ (This is because
$0$ and infinity are critical points of the square map $\Phi$).

  Let us prove the last assertion. Let $\eta: \Bbb T \to \Bbb T$ be the
circle homeomorphism such that $\eta(1) = 1$ and $\eta \circ
(G_{\theta}|\Bbb T) \circ \eta^{-1} = R_{\theta}$. Since
$G_{\theta}$ is odd,  it follows that $\eta$ is odd also.

Let $I$ denote the anticlockwise arc from $1$ to $g_{\theta}(1) =
(G_{\theta}(1))^{2}$. Let  $$\Lambda_{n}  = \{g_{\theta}^{k}(1) =
(G_{\theta}^{k}(1))^{2},\:\: 0 \le k \le n\} \hbox{  and  } \Psi_{n}
= \{G_{\theta}^{k}(1),\:\: 0 \le k \le n\}.$$ Let $P_{n}$ denote the
numbers of the points in $\Lambda_{n}$  which are contained in $I$.
Since  $R_{\theta}: \Bbb T \to \Bbb T$ preserves the Lebesgue
measure and is uniquely ergodic (cf. Proposition 4.2.1 of
\cite{KH}), we have
$$
\lim_{n\to \infty}P_n/n = \tau.
$$
Our proof is divided into  two cases.

 In the
first case, $0< \theta < 1/2$.    Let $J$ denote the anticlockwise
arc from $1$ to $G_{\theta}(1)$  and  $-J$ denote the anticlockwise
arc from $-1$ to $-G_{\theta}(1)$.  Then $\eta(J)$ is the
anticlockwise arc from $1$ to $e^{2 \pi i \theta}$ and $\eta(-J)$ is
the anticlockwise arc from $-1$ to $-e^{2 \pi i \theta}$. Let
$Q_{n}^{+}$ and $Q_{n}^{-}$ denote the numbers of the points in
$\Psi_{n}$ which are contained in $J$ and $-J$, respectively. Again
since $R_{\theta}: \Bbb T \to \Bbb T$ preserves the Lebesgue measure
and is uniquely ergodic, we have
$$
\lim_{n\to \infty} Q_{n}^{+}/n = \lim_{n\to \infty} Q_{n}^{-}/n =
\theta.
$$   From the assumption that  $0< \theta < 1/2$,
it follows that  $\eta(J)$, which is the anticlockwise arc from $1$
to $e^{2 \pi i \theta}$,  is properly contained in the upper half of
the unit circle. Since $\eta(1) = 1$ and $\eta(-1) = -\eta(1) = -1$,
$\eta$ preserves the upper half of the unit circle. Thus $J$ is
properly contained in the upper half of the unit circle. In
particular, this implies
$$J \cap (-J) = \emptyset.$$
 Note that $$g_{\theta}^{k}(1) =  (G^{k}_{\theta}(1))^{2} \in I \hbox{  if and only if
}G_{\theta}^{k}(1) \in J \cup (-J)$$ Since $J \cap (-J) =
\emptyset$, we have $P_{n} = Q_{n}^{+} + Q_{n}{-}$ and
$$
\tau = \lim_{n\to \infty} P_{n}/n = \lim_{n\to \infty} [(Q_{n}^{+}+
Q_{n}^{-}]/n = 2\theta.
$$

In the second case, $1/2< \theta < 1$. Again  let $J$ denote the
anticlockwise arc from $1$ to $G_{\theta}(1)$  and $Q_{n}^{+}$
denote the numbers of the points in $\Psi_{n}$ which are contained
in $J$. Similarly, we have
$$
\lim_{n\to \infty} Q_{n}^{+}/n  = \theta > 1/2.
$$  Since $1/2 < \theta < 1$,
 $e^{2 \pi i \theta}$ belongs to the lower half of the unit circle.
 Since $G_{\theta}(1) = \eta^{-1}(e^{2 \pi i \theta})$  and $\eta$ preserves the
 lower half of the unit circle,
$G_{\theta}(1)$ is contained in the lower half of the unit circle.
Thus $-G_{\theta}(1)$ is contained in the upper half of the unit
circle. Let $S$ denote the anticlockwise arc from $1$ to
$-G_{\theta}(1)$ and $-S$ denote the anticlockwise arc from $-1$ to
$G_{\theta}(1)$.   Note that $S$ is a sub-arc of $J$ and $J -S$ is
the anticlockwise arc from $-G_{\theta}(1)$ to $G_{\theta}(1)$,
which is a half of the unit circle.  Let  $T_{n}^{+}$ and
$T_{n}^{-}$ denote the numbers of the points in $\Psi_{n}$ which are
contained in $S$ and $-S$,  respectively.  Let $D_{n}$ denote the
number of the points in $\Psi_{n}$ which are contained in $J-S$.
Since $\eta$ is odd,  $\eta$ maps $J-S$ to a half of the unit
circle. Again sine $R_{\theta}: \Bbb T \to \Bbb T$ preserves the
Lebesgue measure and is uniquely ergodic,  we have
$$
\lim_{n\to \infty}\frac{D_{n}}{n} = \frac{1}{2}.
$$  Since $T_{n}^{+} = Q_{n}^{+} - D_{n}$, we  have
$$
\lim_{n\to \infty} T_{n}^{+}/n = \lim_{n\to \infty} Q_{n}^{+}/n  -
\lim_{n\to\infty}D_{n}/n =  \theta - 1/2.
$$  Since $\eta$ is odd, $\eta(-S) = -\eta(S)$. So
$\eta(S)$ and $\eta(-S)$ have the same length. We thus have
$$
\lim_{n\to \infty} T_{n}^{-}/n  = |\eta(-S)| = |\eta(S)| =
\lim_{n\to \infty} T_{n}^{+}/n = \theta - 1/2.
$$
 Note that in the second case   $$g_{\theta}^{k}(1) =
(G^{k}_{\theta}(1))^{2} \in I \hbox{  if and only if
}G_{\theta}^{k}(1) \in S \cup (-S)$$  Since $S \cap (-S) =
\emptyset$, we have $P_{n} = T_{n}^{+} + T_{n}^{-}$.  We thus have
$$
\lim_{n\to \infty} P_{n}/n = \lim_{n\to \infty} [(T_{n}^{+}+
T_{n}^{-}]/n = 2\theta - 1.
$$
The lemma follows.

\end{proof}


\begin{lem}\label{arith}
 Let $0<
\theta < 1$ be an irrational number of David type. Let $0< \tau < 1$
be the irrational number such that
$$
\tau \equiv 2 \theta \mod(1).
$$
Then $\tau$ is also of David type.
\end{lem}
\begin{proof}
Let $[b_{1}, \cdots, b_{n}, \cdots]$, $s_{n}/t_{n}$,  and $[a_{1},
\cdots, a_{n}, \cdots,]$, $p_{n}/q_{n}$, be the continued fractions
and convergents of $\theta$ and $\tau$, respectively. Let $n \ge 4$.
Note that if both $t_{n-1}$ and $t_{n}$ are odd, then   $t_{n} -
t_{n-1}$ must be even. So there is at least one even integer among
the three integers  $t_{n-1}$, $t_{n}$ and   $t_{n} - t_{n-1}$.   We
claim that there exists an even integer $L= 2m$ among $t_{n-1},
t_{n}$ and $t_{n} - t_{n-1}$ and an integer $N \ge 0$ such that the
inequality
\begin{equation}\label{arith ine}
|L \theta - N| = |2m \theta - N|< |2y \theta - x|
\end{equation}
holds for all integers $x \ge 0$ and  $0 < y < m$.

Let us prove the claim.  If one of $t_{n-1}$ and $t_{n}$ is even, we
can take $L$ to be it,  and  let  $N$  be $s_{n-1}$ if $L = t_{n-1}$
and be $s_{n}$ if $L = t_{n}$. Then the claim is obviously true.
Otherwise, both $t_{n-1}$ and $t_{n}$ are odd integers. Then let $L
= t_{n} - t_{n-1}$ and let $N \ge 0$ be the integer such that the
left hand of (\ref{arith ine}) obtains the minimum.  Note that
$t_{n} = t_{n-2} + b_{n} t_{n-1}$. If $b_{n} = 1$, then $L  = t_{n}
- t_{n-1}= t_{n-2}$. In this case, $N$ must be $s_{n-2}$ and the
claim is obviously true. Otherwise, we have $b_{n}
> 1$ and thus $t_{n} - t_{n-1}
> t_{n-1}$. Then the claim  follows  from the property of the closest returns: since   the only possible
integers $s$ and $t$ such that $t < t_{n} - t_{n-1}$ and $|(t_{n} -
t_{n-1}) \theta - N| \ge |t \theta - s|$ are $s_{n-1}$ and
$t_{n-1}$. But $t_{n-1}$ is odd and can not be equal to $2y$ for any
integer $y$, hence (\ref{arith ine}) also holds in the later case.
The claim has been proved.

From (\ref{arith ine}) and $\tau \equiv 2 \theta \mod(1)$ , it
follows that there exists some integer $M \ge 0$ such that
\begin{equation}\label{appro}
|m \tau - M| < |y \tau   - x|
\end{equation}
holds for all integers  $x \ge 0$ and  $0 < y < m$. This implies
that \begin{equation}\label{reed} m = q_{k_{0}}\end{equation} for
some $k_{0} \ge 0$.

For each $n \ge 5$, let $k$ be the unique integer such that
\begin{equation}\label{ch-k}
q_{k} < t_{n+1} \le q_{k+1}.
\end{equation}
Since $L$ is one of the three integers $t_n$, $t_{n-1}$ and $t_n -
t_{n-1}$,  we have
$$t_{n-2} \le L \le t_{n}.$$ We thus have  $$m = L/2 <
t_{n+1}.$$  From (\ref{ch-k}) it follows that  $k$ is the largest
integer such that $q_{k} < t_{n+1}$. Since $m = q_{k_{0}}$ for some
$k_{0} \ge 0$ by (\ref{reed}) and $q_{k_{0}}  = m < t_{n+1}$ by the
last inequality, it follows  from (\ref{ch-k})  that $k_{0} \le k$
and thus
\begin{equation}\label{oj}
m = q_{k_{0}} \le q_{k}. \end{equation} Since $L \ge t_{n-2}\ge
t_{n-4} + t_{n-3}
> 2 t_{n-4}$, we have $m = L/2
> t_{n-4}$.  This, together with (\ref{oj}), implies
\begin{equation}\label{mmjn}
q_{k} >  t_{n-4}.
\end{equation}
From (\ref{ch-k}) and (\ref{mmjn}) it follows that for every $n \ge
5$, there is some $q_{k}$ such that
\begin{equation}\label{rrtp}t_{n-4} < q_{k} < t_{n+1}.\end{equation}

Now for every $l \ge 1$,  let $n \ge 1$ be the least integer such
that $q_{l} < t_{n+1}$. It is clear that $n \ge 9$ for all $l$
large. Since for every $n \ge 5$, by (\ref{rrtp}), there is at least
one $q_{k}$ in $(t_{n-4}, t_{n+1})$, and since $n \ge 1$ is the
least integer such that $q_{l} < t_{n+1}$, it follows that
\begin{equation}\label{rlnl}
n  <  5l + 10.
\end{equation}
The constant $10$ in the right hand of the above inequality may not
be optimal. Again, by (\ref{rrtp}), there is some $q_{l'}$ with
$q_{l'} \in (t_{n-9}, t_{n-4})$. We thus have $q_{l'} < q_{l}$. This
means $l' \le l-1$ and thus $q_l' \le q_{l-1}$. So we have
$$
q_{l-1} \ge q_{l'} >   t_{n-9}.
$$
Therefore, we have
$$
a_{l} \le  q_{l}/q_{l-1}  <  t_{n+1}/t_{n-9}.
$$
This, together with (\ref{rlnl})  and the fact that $t_k/t_{k-1} <
b_k$,  implies that
$$
\log a_{l} <  \log (t_{n+1}/t_{n-9}) \le \sum_{n-8 \le k \le n+1}
\log b_{k}  \le C \sqrt n \le C' \sqrt l
$$
holds for all $l \ge 1$ large, where $C, C' > 0$ are some constants
depending only on $\theta$. The lemma follows.
\end{proof}

\section{Yoccoz's Cell Construction}
By Lemma~\ref{rotation number}, the map $g_{\theta}|\Bbb T: \Bbb T
\to \Bbb T$ is a critical circle homeomorphism with rotation number
$\tau$.  By Lemma~\ref{arith}, $\tau$ is of David type.  By Yoccoz's
linearization theorem, $g_{\theta}|\Bbb T$ is homeomorphically
conjugate to the rigid rotation $z \mapsto e^{2 \pi i \tau} z$. by
Yoccoz's extension theorem(see \cite{Yo2} or Theorem 6.5 of
\cite{PZ}), the conjugacy can be extended to a David homeomorphism
of the  unit disk to itself.   The construction of the extension is
based on  a partition of an inner half neighborhood of the unit
circle into cells.  The goal of this section is to sketch the
construction following Petersen-Zakeri's presentation in \cite{PZ}
with only some minor differences in terminologies.

 For $i \in {\Bbb Z}$, let $x_{i}$
denote the backward iterate $(g_{\theta}|\Bbb T)^{-i}(1)$, that is,
the point in $\Bbb T$ such that $g_{\theta}^{i}(x_{i}) = 1$. Let
$\{a_{1}, a_{2}, \cdots, a_{n}, \cdots \}$ denote the coefficients
of  the continued fraction of $\tau$. For $n \ge 0$, let
$p_{n}/q_{n}$ denote the $n$-th truncated continued fraction of
$\tau$. For any two distinct points $z, w \in \Bbb T$, let $d(z, w)$
denote the length of the smaller arc connecting $z$ and $w$.
Consider the cell partition of  $\Bbb T$  of level $n$  introduced
in $\S2$,
$$
\Xi_{n}  = \{x_{i} \: \big{|}\: 0\le i \le  q_{n+1} -1\}.
$$   Take
\begin{equation}\label{constant-n} N_{0} \ge 1\end{equation} large
enough such that for all $n \ge N_{0}$, $d(x_{i}, x_{j})< 1$ holds
for any two adjacent points $x_{i}$ and $x_{j}$ in $\Xi_{n}$.  From
now on we always assume that $n \ge N_0$.

 For each $x_{i} \in
\Xi_{n}$, let $y_{i}$ be the point on the radial segment $[0,
x_{i}]$ such that
$$
|y_{i}-x_{i}| = d(x_{r}, x_{l})/2
$$
where $x_{r}$ and $x_{l}$ denote the two  points immediately to the
right and left of $x_{i}$ in $\Xi_{n}$.

\begin{defi}[Yoccoz's cells]\label{b-s}{\rm
Let $x_{i}$ and $x_{j}$ be any two adjacent points in $\Xi_{n}$.
Connect $y_{i}$ and $y_{j}$ by a straight segment. Then the three
straight segments $[x_{i}, y_{i}]$, $[y_{i}, y_{j}]$, $[x_{j},
y_{j}]$,  and the  arc segment  $[x_{i}, x_{j}]$ bound a domain in
$\Delta$. We call the closure of this domain  a $\emph{cell of
level}$ $n$. }
\end{defi}

From Definition~\ref{b-s}, it follows that the union of all the
Yoccoz's cells of level $n$ is a  closed  annulus. Let us denote
this annulus by $Y_{n}$. From the construction we see that the outer
boundary component  of $Y_n$ is $\Bbb T$, and the  inner boundary
component of $Y_n$  is the union of finitely many straight segments,
and moreover,
$$
Y_{N_{0}}  \supset Y_{N_{0}+1} \supset \cdots \supset Y_n \supset
Y_{n+1} \supset \cdots .
$$
\begin{figure}
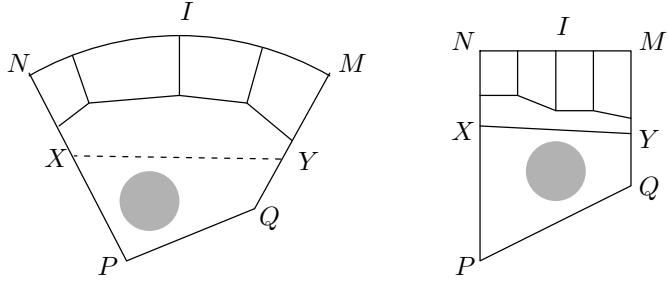

\bigskip
\begin{center}
\centertexdraw { \drawdim cm \linewd 0.02 \move(-2 1)

\move(4  -0.2) \lvec(6  -0.2) \lvec(6 -2)\lvec(4 -3)\lvec(4 -0.2)
\move(4.5 -0.2)\lvec(4.5 -0.8)\lvec(5 -1)\lvec(5 -0.2)\move(5.5
-0.2)\lvec(5.5 -1) \lvec(6 -1.1) \move(5.5 -1) \lvec(5 -1) \move(4.5
-0.8)\lvec(4 -0.8)  \move(4 -1.2)\lvec(6 -1.3)  \move(5 0)
\htext{$I$}

\move(3.65 -3.2) \htext{$P$} \move(6.1 -2.2) \htext{$Q$} \move(6.1
-0.2) \htext{$M$}  \move(6.1 -1.5) \htext{$Y$} \move(3.61 -1.4)
\htext{$X$} \move(3.61 -0.2) \htext{$N$}

\move(0 -4) \larc r: 4 sd:60 ed:120

\move(0   -0.8) \lvec(0.9   -0.9) \lvec(1.5 -1.4)

\move(1.1 -0.17) \lvec(0.9   -0.9)

\move(0 0) \lvec(0   -0.8) \lvec(-1.2   -0.9) \lvec(-1.6   -1.2)

\move(-0.4 -2.2) \fcir f:0.7 r:0.4

\move(5 -1.8) \fcir f:0.7 r:0.4

\move(-1.42 -0.26) \lvec(-1.2  -0.9)

\move(1 -2.3)    \lvec(2 -0.5)

\move(-0.7 -3)    \lvec(-2 -0.5) \move(-1.1 -3.2)
 \htext{$P$}  \move(1.05 -2.6)
 \htext{$Q$}

 \move(-0.7 -3) \lvec(1 -2.3)

\move(1.57   -1.8)   \htext{$Y$}

\move(-1.8   -1.75) \htext{$X$} \move(0   0.2) \htext{$I$}

\move(2.1 -0.5) \htext{$M$}

\move(-2.3 -0.5) \htext{$N$}

\lpatt(0.067 0.1)  \move(1.35 -1.64) \lvec(-1.4 -1.6)

 }
\end{center}
\caption{The geometry of the cells.  }
\end{figure}

 Let $E$ be a
cell of level $n$.  Let $I$ be the arc segment to which $E$ is
attached. That is, $I = E \cap \Bbb T$.  Suppose $P, Q, M$ and $N$
are the four vertices of $E$. When $n$ is large,  $I$ is small and
like a straight segment. The two radial  edges  of $E$  are
perpendicular to $I$. So we may regard $E$ as a trapezoid.  See
Figure 2 for an illustration.

In Corollary~\ref{c-g} and Lemma~\ref{geom}, let $f$ be
$g_\theta|\Bbb T$ and  let $1 < K < \infty$, $0< \delta < 1$ be the
constants there.  Let
$$
\sigma = \max\{\delta, \frac{K}{1 + K}\}. $$  In Figure 2 let $X \in
PN$ and $Y \in QM$ be the points such that
$$ |NX| = \sigma |PN| \hbox{  and  } |MY| = \sigma |QM|.$$ Then the straight segment $XY$ divide $E$ into two
trapezoids: $XYMN$ and $PQYX$.
\begin{lem}[Geometry of the cells]\label{good geometry} There exists
an $1 < L < \infty$ depending only on $g_{\theta}$ such that the
four edges of $E$ are $L$-commensurable. Moreover,  there is a
Euclidean disk $B$ such that (1) ${\rm diam}(B) \asymp |I|$, (2)
${\rm dist}(B, \Bbb T) \asymp |I|$ and (3) $B \cap Y_{n+2} =
\emptyset$.

\end{lem}
\begin{proof}
Note that when $E$ is large, $I$ is like a straight segment, and
thus  $E$ is like a trapezoid with $I$ being the top edge and the
two radial edges perpendicular to $I$. Let us first prove $PN$ and
$QM$ are $L$-commensurable to $I$ for some  $1< L < \infty$
depending only on $g_\theta$. It suffices to prove this for $PN$
since the argument applies to $QM$ also.  By definition, we have
\begin{equation}\label{fw1}
|PN| = \frac{1}{2}(|I| + |I'|)
\end{equation}
where $I'$ is one of the adjacent interval component of  $I$ in the
cell partition of level $n$. By Corollary~\ref{c-g}, $I$ and $I'$
are $K$-commensurable for some  $1< K< \infty$ depending only on
$g_{\theta}$.  From (\ref{fw1}) we have $$ \frac{2}{1 + K} |I| \le
\frac{1 + K^{-1}}{2} |I|  < |PN| < \frac{1+K}{2} |I|. $$ This
implies that $PN$ is $\frac{1+K}{2}$-commensurable with $I$.  Using
the same argument one can show that  $QM$ is
$\frac{1+K}{2}$-commensurable with $I$ also. Now let us consider
$PQ$.  Note that $I$ is like a straight segment for all $n$ large
enough and both $PN$ and $QM$ are perpendicular to $I$, we have
$$
{\rm dist}(PN, QM) \asymp |I|
$$
where ${\rm dist}(PN, QM)$ denotes the Euclidean distance between
$PN$ and $QM$. Since $|PQ| \ge {\rm dist}(PN, QM)$, we have
$$
|PQ|  \succeq |I|.
$$
On the other hand, we have
$$
|PQ|  < |PN| + |I| + |QM| \asymp |I|.
$$
This implies that $|PQ|\asymp|I|$.  This proves the first assertion
of the lemma by taking an $L > 1$ large enough.

Recall that for $n$ large,  $I$ is like a straight segment, and $PN$
and $QM$ are like two parallel straight segment both of which are
perpendicular to $I$.  So for $n$ large, $I$ is like the projection
of $PQ$ to the unit circle and thus $$|I| \asymp |PQ|
\cdot\sin(\angle NPQ) \asymp |PQ| \cdot \sin(\angle PQM).$$ Since
$|PQ| \asymp |I|$ by the first assertion of the lemma, the two
angles $\angle NPQ$ and $\angle PQM$ are bounded away from $0$ and
$\pi$. That is, there is a uniform $0< \eta < \pi$ such that
\begin{equation}\label{comp-f}
\eta < \angle NPQ, \angle PQM < \pi - \eta.
\end{equation}
Note that
\begin{equation}\label{comp-s}
|PX| = (1-\sigma) |PN| \asymp |I|\hbox{  and  } |QY| = (1-\sigma)
|QM|\asymp |I|.
\end{equation}
Since $PX$ and $QY$ are almost parallel to each other, by
(\ref{comp-f}), (\ref{comp-s})  and a compactness argument, we can
deduce that $PQYX$ contains a Euclidean disk, say $B$,  such that
${\rm diam}(B) \asymp {\rm diam} (PQYX) \asymp |I|$.   Since $PN$
and $QM$ are perpendicular to $I$, we have
$$
{\rm dist}(B, \Bbb T) \asymp {\rm dist}(B, I) \succeq \min\{|NX|,
|MY|\} \asymp |I|.
$$

It remains to prove that $Y_{n+2} \cap B= \emptyset$.  Since $B
\subset PQYX$, it suffices to prove $Y_{n+2} \cap PQYX = \emptyset$.
Let $E' \subset E$ be an arbitrary cell of level $n+2$. It suffices
to prove that all the four vertices of $E'$ are above the straight
segment $XY$. See the figure on the right hand of Figure 2 for an
illustration. Suppose $Z$ is a vertex of $E'$. We may assume that $Z
\notin I$ since otherwise there is nothing to prove. Then there are
two cases. In the first case, $Z$ lies in $PN$ or $QM$.  In the
second case, $Z$ is within the two radial segments $PN$ and $QM$.

Suppose we are in the first case. Let us only consider the case that
$Z \in PN$. The same argument will work for the case that $Z \in
QM$. Suppose $J$ is the adjacent interval of $I$ in the cell
partition of level $n$ such that they have a common end point $N$.
Let $I'$ and $J'$ be the two adjacent intervals in the partition of
level $n+2$ such that $I' \subset I$ and $J' \subset J$. Then  $N$
is  the common end point of $I'$ and $J'$. By the construction of
the cells, we have
\begin{equation}\label{dlp}
|NZ| = \frac{|I'|+ |J'|}{2}.
\end{equation}
Since $I' \subset I$ and $J' \subset J$,   we have  $|I'| < \delta
|I|$ and $|J'|< \delta |J|$ by Lemma~\ref{geom}. By the construction
of cells, we have
\begin{equation}\label{ddre}|PN| = \frac{|I|+|J|}{2}. \end{equation}
From (\ref{dlp}),  (\ref{ddre}) and the definition of $\sigma$, we
have
$$|NZ| < \delta |PN| \le \sigma|PN| = |NX|.$$
This proves the first case.

In the second case, $Z$ is within the two radial segments $PN$ and
$QM$.  Let $W$ be the interior point of $I$ such that $WZ$ is a
radial edge of $E'$. Let $I'$ and $I''$ be the two adjacent
intervals  in the cell partition of level $n+2$ so that $I'$ and
$I''$ have a common end point $W$. Then $I'$ and $I''$ are both
contained in $I$. So we get

\begin{equation}\label{kk1}
|WZ| = \frac{|I'| + |I''|}{2} \le \frac{|I|}{2}.
\end{equation}

Note that in (\ref{ddre})  we have $|J| > |I|/K$ by
Corollary~\ref{c-g}. So  we get $|PN| > \frac{K+1}{2K} |I|$.  This,
together with the definition of $\sigma$, implies
\begin{equation}\label{kk2}
|NX| = \sigma |PN| >  \frac{K}{1 + K} \cdot \frac{K+1}{2K} |I| =
\frac{|I|}{2}. \end{equation} Using the same argument we can show
that
\begin{equation}\label{kk3}|MY|
> \frac{|I|}{2}.\end{equation}   From (\ref{kk1}-\ref{kk3}), it follows
that
$$
|WZ| < \min\{|NX|, |MY|\}.
$$
This implies that $Z$ is above the straight segment $XY$. The proof
of Lemma~\ref{good geometry} is completed.
\end{proof}

By Lemma~\ref{c-g} and Theorem~\ref{SH}, there exist $1 <  C <
\infty$ and $0< \delta < 1$ such that for each interval $I$ in the
cell partition of level $n$, we have $|I| < C \cdot \delta^{n}$.
Since the four edges of each cell are $L$-commensurable, it follows
that all the cells of level $n$ are contained in the annulus
$$
\{z\:|\:  1 - L\cdot C\cdot  \delta^{n} < |z| < 1\}.
$$ From this it follows that there exist
 $1 < C_{1} < \infty$ and $0< \delta_1 < 1$ depending only on $C$, $L$ and $\delta$ such that
\begin{equation}\label{area-in}
area(Y_{n}) < C_{1} \delta_{1}^{n}
\end{equation}
for all $n \ge N_{0}$.

Let $R_{\tau}: \Bbb T \to \Bbb T$ be the rigid rotation given by $z
\mapsto e^{2 \pi i \tau} z$. Let $h: {\Bbb T} \to {\Bbb T}$ be the
homeomorphism such that $h(1) = 1$ and $g_{\theta}|{\Bbb T}(z) =
h^{-1} \circ R_{\tau} \circ h(z)$.  Since $\log{a_{n}}  = O(\sqrt
{n})$, by Yoccoz's extension theorem (see \cite{Yo2} or Theorem 6.5
of \cite{PZ}), we have
\begin{lem}\label{Yoccoz extension} Let $N_{0} \ge 1$ be the
constant taken in (\ref{constant-n}). Then there exist a
homeomorphism $H: \overline{\Delta} \to \overline{\Delta}$ and  a
constant $1 < \lambda < \infty$ such that
\begin{itemize}
\item[1.] $H|\Bbb T = h$, \item[2.]
for all $n \ge N_{0}$, the dilatation of $H$  in $\Delta\setminus
Y_{n+2}$ is not greater than $\lambda \cdot n$. \end{itemize}
\end{lem}

By composing $H$ with a quasiconformal homeomorphism of the unit
disk to itself which fixes the unit circle and maps $H(0)$ to $0$,
we may assume that $H(0) = 0$. Note that after this modification,
$H$ still satisfies the properties in Lemma~\ref{Yoccoz extension}
with probably different  $\lambda$.  Let
$$
\nu_{H} = \frac{\overline{\partial} H }{\partial H}
$$
be the Beltrami differential of $H$ in $\Delta$. Define
\begin{equation} \widetilde{g}_{\theta}(z) =
\begin{cases}
 g_{\theta}(z) & \text{for $z \in {\Bbb C} - \overline{\Delta}$}, \\
 H^{-1}\circ R_{\tau}\circ H(z)  & \text{
for $z \in \overline{\Delta}$}.
\end{cases}
\end{equation}
It follows that $\nu_{H}$ is $\widetilde{g}_{\theta}$-invariant. Let
$\nu$ denote the Beltrami differential in the whole complex plane
which is obtained by the pull back of $\nu_{H}$ through the
iterations of $\widetilde{g}_{\theta}$.

Define the premodel $\widetilde{G}_{\theta}: \Bbb C \to \Bbb C$ for
$f_{\theta}$ by
\begin{equation} \label{dg}
 \widetilde{G}_{\theta}(z) =
\begin{cases}
 G_{\theta}(z) & \text{for $z \in {\Bbb C} - {\Delta}$}, \\
\Phi^{-1}\circ H^{-1}\circ R_{\tau}\circ H \circ \Phi(z) & \text{
for $z \in{\Delta}$}.
\end{cases}
\end{equation}
Let us show that one can choose the branch of $\Phi^{-1}$ in the
above definition so that  $ \widetilde{G}_{\theta}: \Bbb C \to \Bbb
C$ is continuous.   Since $H(0) = 0$ by the previous assumption and
$\Phi$ is the square function, if $z$ runs through a closed curve in
$\Delta$ which turns around $0$ one time, the value
$$
H^{-1}\circ R_{\tau}\circ H \circ \Phi(z)
$$
will describe a closed curve which turns around $0$ two times, and
thus
$$
\Phi^{-1} \circ H^{-1}\circ R_{\tau}\circ H \circ \Phi(z)
$$
describe a closed curve which turns around $0$ one time. This
implies that  by continuous extension $\widetilde{G}_{\theta}$ is
well defined  in $\Delta$. From (\ref{easier-m}) it follows that
$\widetilde{G}_{\theta}$ is continuous in $\Bbb T$ and thus
continuous in $\Bbb C$.

 Let $\mu$ be the Beltrami differential in the complex plane
which is obtained by the pull back of $\nu$ through the square map
$\Phi$.
\begin{lem}\label{final}
The map $\widetilde{G}_{\theta}$ is odd.  The Beltrami differential
$\mu$ is $\widetilde{G}_{\theta}$-invariant, and moreover, $\mu(z) =
\mu(-z)$.
\end{lem}
\begin{proof}
For all $z$ outside  the unit disk, the odd property of
$\widetilde{G}_{\theta}$ at $z$ follows from that of $G_{\theta}$.
Now assume that $z \in \Delta$. Let $w = H^{-1}\circ R_{\alpha}\circ
H \circ \Phi(z)$. When $z$ runs through the following half circle
which connects $z$ and $-z$,
$$
\{\zeta\:  \big{|}\: |\zeta| = |z|, \arg z \le \arg \zeta \le \arg z
+ \pi\},
$$
the value  $$H^{-1}\circ R_{\alpha}\circ H \circ \Phi(\zeta)$$
describes a simple closed curve  $\Gamma$ which  turns around the
origin exactly one time.  Since the branch of $\Phi^{-1}$ is chosen
by continuous extension,   $\Phi^{-1}(w)$ will  change to be its
negative after $w$ runs through $\Gamma$ one time. This implies that
$\widetilde{G}_{\theta} (-z) = - \widetilde{G}_{\theta}(z)$. This
proves the first assertion of the lemma.

Now let us verify that $\mu$ is $\widetilde{G}_{\theta}$-invariant.
Take an arbitrary $z \in \Bbb C$.  First let us assume $z \in \Bbb C
\setminus \overline{\Delta}$.  Let $E$ denote the infinitesimal
ellipse at $z$ which is associated to $\mu(z)$. By the definition of
$\mu$, the square map $\Phi$ maps $E$ to an infinitesimal ellipse
$T$ at $\Phi(z)$ which is associated to $\nu(\Phi(z))$. Since $\nu$
is $\widetilde{g}_{\theta}$-invariant and $g_{\theta} =
\widetilde{g}_{\theta}$  outside  the unit disk, $g_{\theta}$ maps
$T$ to the infinitesimal ellipse $S$ at $g_{\theta}(\Phi(z))$ which
is associated to $\nu$. By the definition of $\mu$, the square root
map $\Phi^{-1}$ maps $S$ to the infinitesimal ellipse $W$ at
$$
\widetilde{G}_{\theta}(z) =  G_{\theta}(z) = \Phi^{-1} (g_{\theta} (
\Phi(z)))
$$
which is associated to $\mu(\widetilde{G}_{\theta}(z))$.  This
proves that  $\mu$ is $\widetilde{G}_{\theta}$-invariant  for $z \in
\Bbb C \setminus \overline{\Delta}$.  Now assume that $z \in
\Delta$. Note that $\nu$ is $H^{-1} \circ R_{\alpha} \circ
H$-invariant. Again the proof is by tracking the images of the
infinitesimal ellipses and the argument is exactly the same as in
the case that $z \in \Bbb C \setminus \overline{\Delta}$.

Now let us prove $\mu(z) = \mu(-z)$. This is equivalent to prove
that the infinitesimal ellipses at $z$ and $-z$ are symmetric about
the origin. But this is obvious since both of them are mapped by the
square map to the same infinitesimal ellipse of $\nu$ at $z^2$.
\end{proof}


The heart of the paper is to verify the integrability of $\mu$, that
is, the existence of constants $M
> 0$, $\alpha
> 0$,  and $0< \epsilon_{0} < 1$  such that for any $0< \epsilon <
\epsilon_{0}$, the following inequality holds,
\begin{equation}\label{integrability}
area \{z\:\big{|}\: |\mu(z)| > 1 - \epsilon \} \le M
e^{-\alpha/\epsilon}.
\end{equation}
where $area(\cdot)$ denotes the area of a measurable set with
respect to the spherical metric.  The next lemma says that the
integrability of $\mu$ follows from that  of $\nu$.

\begin{lem}\label{equi}
If $\nu$ satisfies (\ref{integrability-p}) for some $M, \alpha > 0$
and $0< \epsilon_0 < 1$,  then  $\mu$ satisfies
(\ref{integrability})  with the same $0< \epsilon_{0}< 1$ but
possibly different constants $M
> 0$ and $\alpha
> 0$.
\end{lem}
\begin{proof}
Recall that  $\Phi: z \to z^{2}$ is the square map.   Since $\mu$ is
the pull back of $\nu$ by $\Phi$ and $\Phi$ preserves the
dilatation, it is sufficient to prove that there exists a $C
> 0$ such that for any measurable set $E \subset {\Bbb C}$,
the following inequality holds,
\begin{equation}\label{growth}
area(\Phi^{-1}(E)) < C area(E)^{1/2}.
\end{equation}
To show this, let $E_{1} = E \cap \overline{\Delta}$ and $E_{2} = E
\cap ({\Bbb C} \setminus \Delta)$. It is sufficient to prove the
existence of a universal constant $1 < C < \infty$ such that
(\ref{growth}) holds for both $E_{1}$ and $E_{2}$. Since the
transform $\zeta = 1/z$ commutes with $\Phi$ and preserves the the
spherical area form
$$\frac{\frac{i}{2} dz \wedge d\bar{z}}{(1 + |z|^{2})^{2}},$$ and
maps $E_{2}$ to some subset of $\overline{\Delta}$,  it suffices to
prove (\ref{growth}) holds for $E_{1}$. Note that in
$\overline{\Delta}$, the Euclidean area is equivalent to the
spherical area.  That is, there exists a universal constant  $1 <
C_{1} < \infty$ such that for any measurable subset $E_{1} \subset
\overline{\Delta}$, we have
$$area(E_{1})/C_{1}\le m(E_{1}) \le C_{1} area(E_{1})$$ where
$m(\cdot)$ denotes the Euclidean area.  Thus it is sufficient to
prove the existence of a universal constant $1 < C_{2} < \infty$
such that  for any measurable subset $E_{1} \subset
\overline{\Delta}$, \begin{equation}\label{iem} m(\Phi^{-1}(E_{1}))
< C m(E_{1})^{1/2}.\end{equation} To see this, let $\zeta = \Phi(z)$
where $\zeta = x + iy$ and $z = s + it$. Since $\Phi:
\Phi^{-1}(E_{1}) \to E_{1}$ is a two-to-one map, we have
\begin{equation}\label{obv} m(E_{1}) = \int_{E_{1}}dxdy =
\frac{1}{2}\int_{\Phi^{-1}(E_{1})}|\Phi'(z)|^{2} dsdt =
2\int_{\Phi^{-1}(E_{1})}|z|^{2} dsdt.
\end{equation}
For $ 0 \le r \le 1$, let
$$
F(r) = m (\Phi^{-1}(E_1) \cap \{z \:|\: |z| \le r\}).
$$
It is clear that $F: [0, 1] \to [0, m(\Phi^{-1}(E_{1}))]$ is  a
continuous and monotone function.  Let $dF(r)$ denote the finite
Borel measure on $[0, 1]$ induced by $F$.  It is clear that
\begin{equation}\label{mdo}
dF(r) \le 2 \pi r dr.\end{equation}   Let $0  \le r_{0} \le 1$ be
such that
$$\pi r_{0}^{2} = m(\Phi^{-1}(E_1)).$$  From (\ref{obv}) we have
\begin{equation}\label{tob}
m(E_{1}) = 2\int_{\Phi^{-1}(E_{1})}|z|^{2} dsdt = 2\int_{0}^{1}
r^{2} dF(r)
$$
$$
=   2\int_{0}^{r_{0}} r^{2} (dF(r) - 2 \pi r dr) + 2\int_{0}^{r_{0}}
r^{2} \cdot 2 \pi r dr + 2\int_{r_{0}}^{1} r^{2} dF(r)
\end{equation}
By (\ref{mdo}) it follows that $(dF(r) - 2 \pi r dr)$ is a
non-positive measure.  Thus we have
\begin{equation}\label{dd4}
2\int_{0}^{r_{0}} r^{2} (dF(r) - 2 \pi r dr)  \ge 2 r_{0}^{2}
\int_{0}^{r_{0}}  (dF(r) - 2 \pi r dr) = 2r_{0}^{2} \int_{0}^{r_{0}}
dF(r) - 2 \pi r_{0}^{4}.
\end{equation}
Since $dF(r)$ is non-negative, we have
\begin{equation}\label{ee44}
2\int_{r_{0}}^{1} r^{2} dF(r) \ge 2 r_{0}^{2} \int_{r_{0}}^{1}
dF(r).
\end{equation}
From (\ref{dd4}), (\ref{ee44}) and the choice of $r_0$, we have
$$
2\int_{0}^{r_{0}} r^{2} (dF(r) - 2 \pi r dr) + 2\int_{r_{0}}^{1}
r^{2} dF(r)
$$
$$
\ge 2r_{0}^{2} \int_{0}^{r_{0}} dF(r) - 2 \pi r_{0}^{4} +
2r_0^2\int_{r_{0}}^{1}  dF(r)
$$
$$
= 2r_{0}^{2} \int_{0}^{1} dF(r) - 2 \pi r_{0}^{4} = 2 r_{0}^{2}
m(\Phi^{-1}(E_{1})) - 2\pi r_{0}^{4} = 0.
$$

This, together with (\ref{tob}), implies
$$
m(E_{1}) \ge 2\int_{0}^{r_{0}} r^{2} (2 \pi r dr) = \pi r_{0}^{4} =
\pi^{-1} (\pi r_0^2)^{2} =  \pi^{-1} m(\Phi^{-1}(E_{1}))^{2}.
$$
That is,
$$
m(\Phi^{-1}(E_{1})) \le \sqrt{\pi} m(E_{1})^{1/2}.
$$
This proves (\ref{iem}) and  completes the proof of
Lemma~\ref{equi}.
\end{proof}

\section{Proof of the Main Theorem}
Assuming $\nu$ satisfies (\ref{integrability-p}).  By
Lemma~\ref{equi} $\mu$ satisfies (\ref{integrability}).  By
Theorem~\ref{David} we have a David homeomorphism $\phi: \Bbb C \to
\Bbb C$ which solves the Beltrami equation  given by $\mu$ and fixes
$0$ and  infinity, and maps $1$ to $\pi/2$.
\begin{lem}\label{odd homeo}
The map $\phi$ is odd.
\end{lem}
\begin{proof}
By Lemma~\ref{final}, $\mu(z) = \mu(-z)$. Consider the map
$\tilde{\phi}(z) = \phi(-z)$. It follows that $\phi$ and
$\tilde{\phi}$ has the same Beltrami differential. By
Theorem~\ref{David}, it follows that $\tilde{\phi} \circ \phi^{-1}$
is a conformal map in the plane. Since it fixes $0$ and $\infty$, it
follows that $(\tilde{\phi} \circ \phi^{-1})(z) = az$ for some $a
\ne 0$.  Substituting in  $z = \phi(\zeta)$, we get  $\phi(-\zeta) =
a \phi(\zeta)$ for all $\zeta \in \Bbb C$.  Then  $\phi(\zeta) =
\phi(-(-\zeta)) = a \phi(-\zeta)$. Thus  $\phi(-\zeta) = a
\phi(\zeta) = a \phi(-(-\zeta)) = a^{2} \phi(-\zeta)$ for all
$\zeta$. This implies that $a^{2} = 1$. Clearly $a \ne 1$ since
$\phi$ is a homeomorphism of the plane. It follows that $a = -1$ and
thus $\phi(-z) = -\phi(z)$. The lemma has been proved.
\end{proof} Let $\widetilde{G}_\theta$ be defined as in (\ref{dg}).  Define
$T_{\theta}: \Bbb C \to \Bbb C$ by setting
\begin{equation}\label{nnff}
T_{\theta}(z) = \phi \circ \widetilde{G}_{\theta} \circ \phi^{-1}(z)
\hbox{  for all } z \in \Bbb C.
\end{equation}
\begin{lem}\label{PH}
$T_{\theta}$ is an odd entire function.
\end{lem}
The proof uses completely the same argument as in the proof of Lemma
5.5 of \cite{PZ}.
\begin{proof}
Let $X$ denote the set which consists of the origin and all the
critical points of $\widetilde{G}_{\theta}$.  To show that
$T_\theta$ is an entire function, it is sufficient to show that the
map $\phi \circ \widetilde{G}_{\theta}$ belongs to ${\rm
W^{1,1}_{loc}}({\Bbb C} \setminus X)$.  In fact, if $\phi \circ
\widetilde{G}_{\theta}$ belongs to ${\rm W^{1,1}_{loc}}({\Bbb C}
\setminus X)$, then in any small open neighborhood $U$ of a point in
$\Bbb C \setminus X$,  by Lemma~\ref{final}, the Beltrami
differential of $\phi \circ \widetilde{G}_{\theta}$ and $\phi$ are
both equal to $\mu$.  Thus by Theorem~\ref{David} $\phi \circ
\widetilde{G}_{\theta} = \sigma \circ \phi$  where $\sigma$ is a
conformal map defined on $\phi(U)$. This implies that $T_{\theta}$
is holomorphic in the complex plane except the points in $\phi(X)$.
But it is clear that for any point $z \in \phi(X)$, there is a
neighborhood $W$ of $z$ such that $T_{\theta}$ is bounded in
$W\setminus \{z\}$. It follows that all the points in $\phi(X)$ are
removable. So $T_{\theta}$ is an entire function. Now  let us show
that the map $\phi \circ \widetilde{G}_{\theta}$ belongs to ${\rm
W^{1,1}_{loc}}({\Bbb C} \setminus X)$.

Firstly, we have
\begin{equation}\label{tt1}
\phi \circ \widetilde{G}_{\theta} \in {\rm W^{1,1}_{loc}}({\Bbb C}
\setminus (X \cup \overline{\Delta})).\end{equation} This is because
$\widetilde{G}_{\theta}$ is holomorphic in ${\Bbb C} \setminus (X
\cup \overline{\Delta})$ and $\phi \in {\rm W^{1,1}_{loc}}(\Bbb C)$.

Secondly,  we have  \begin{equation}\label{tt2}\phi \circ
\widetilde{G}_{\theta} \in {\rm W^{1,1}_{loc}}({\Delta \setminus
\{0\}}).\end{equation} To see this, write $\phi \circ
\widetilde{G}_{\theta} = \phi\circ \Phi^{-1} \circ H^{-1} \circ
R_{\tau} \circ H \circ \Phi$ in $\Delta$ (cf. (\ref{dg})). Take an
arbitray poitn $z \in \Delta \setminus \{0\}$. Let $U$ be a small
open disk centered at $z$  such that $\Phi$ maps $U$
homeomorphically onto a Jordan domain  $\Phi(U)$.   It follows that
$H \circ \Phi$ belongs to ${\rm W^{1,1}_{loc}}(U)$. Since $R_{\tau}$
is the rigid rotation given by $\tau$, $R_{\tau} \circ H \circ \Phi$
also belongs to ${\rm W^{1,1}_{loc}}(U)$. In particular, $R_{\tau}
\circ H \circ \Phi$ maps $U$ homeomorphically onto a Jordan domain
$V \subset \Delta \setminus \{0\}$. Thus $H^{-1}(V)$ is a Jordan
domain in $\Delta \setminus \{0\}$ and the two branches of
$\Phi^{-1}$ are both well defined and conformal in $H^{-1}(V)$. Thus
for either of the two branches of $\Phi^{-1}$, $\phi\circ \Phi^{-1}$
belongs to  $W^{1,1}_{{\rm loc}}(H^{-1}(V))$.   By Lemma~\ref{final}
$H$ has the same Beltrami differential as $\phi \circ \Phi^{-1}$ in
$H^{-1}(V)$. Since  $H$ belongs to $W^{1,1}_{{\rm loc}}(H^{-1}(V))$
also, it follows from Theorem~\ref{David} that $\phi \circ \Phi^{-1}
\circ H^{-1}$ is conformal in $V$.  It follows that
$$\phi \circ \widetilde{G}_{\theta} = (\phi \circ \Phi^{-1} \circ
H^{-1} ) \circ (R_{\alpha} \circ  H \circ \Phi) \in {\rm
W^{1,1}_{loc}}(U).$$ Since $z$ is an arbitrary point in $\Delta
\setminus \{0\}$, (\ref{tt2}) follows.

 It remains to prove that for
every small open disk $U$ centered at the point in ${\Bbb T}
\setminus \{1, -1\}$, $\phi \circ \widetilde{G}_{\theta} \in {
W^{1,1}_{\rm loc}}({U})$.   From (\ref{tt1}) and (\ref{tt2}) it
follows that  $\phi \circ \widetilde{G}_{\theta}$ is almost
differentiable in $U$. Therefore
\begin{equation}\label{Jac}
\int_{U} {\rm Jac}(\phi \circ \widetilde{G}_{\theta}) \le area\:
((\phi \circ \widetilde{G}_{\theta}) (U)) < \infty.
\end{equation}
This implies that ${\rm Jac}(\phi \circ \widetilde{G}_{\theta}) \in
L^{1}(U)$.   It is sufficient to prove that $\partial (\phi \circ
\widetilde{G}_{\theta}) \in L^{1}(U)$ and thus
$\overline{\partial}(\phi \circ \widetilde{G}_{\theta}) \in
L^{1}(U)$(Then the distributive partial derivatives coincide with
the ordinary partial derivatives in $U$ and are thus integrable in
$U$). But this follows from the following argument. Since $\mu_{\phi
\circ \widetilde{G}_{\theta}} = \mu$ almost everywhere in $U$, we
have
 $$
 |\partial (\phi \circ \widetilde{G}_{\theta})|^{2} =
\frac{Jac(\phi \circ \widetilde{G}_{\theta})}{1-|\mu_{\phi \circ
\widetilde{G}_{\theta}}|^{2}} \le  \frac{Jac(\phi \circ
\widetilde{G}_{\theta})}{1-|\mu_{\phi \circ
\widetilde{G}_{\theta}}|}  =  \frac{Jac(\phi \circ
\widetilde{G}_{\theta})}{1-|\mu|}
$$ and therefore,
$$
|\partial (\phi \circ \widetilde{G}_{\theta})| \le \frac{{\rm
Jac}(\phi \circ \widetilde{G}_{\theta})^{1/2}}{(1-|\mu|)^{1/2}}.
$$
Since $\mu$ satisfies the exponential growth condition
(\ref{integrability}),
 the measurable function $1/(1- |\mu|)$ is
integrable in $U$. This, together with (\ref{Jac}) and Cauchy
inequality, implies the integrability of $\partial (\phi \circ
\widetilde{G}_{\theta})$ in $U$. This proves that $T_{\theta}$ is an
entire function.

The odd property of $T_{\theta}$ follows from the odd property of
$\widetilde{G}_{\theta}$ and $\phi$, see Lemmas~\ref{final}  and
~\ref{odd homeo}.
\end{proof}

\begin{defi}\label{pt}{\rm
Two maps $f: {\Bbb C} \to {\Bbb C}$ and $g: {\Bbb C} \to {\Bbb C}$
are called topologically equivalent  if there exist two
homeomorphisms $\theta_{1}$ and $\theta_{2}$ of the complex plane
such that $f = \theta_{2}^{-1} \circ g \circ \theta_{1}$}.
\end{defi}
\begin{lem}[Lemma 1, \cite{DS}]\label{DS}
Let $f$ be an entire function. If $f$ is topologically equivalent to
the map  $z \mapsto \sin(z)$, then $f(z) = a + b\sin(cz + d)$ where
$a, b, c, d \in\Bbb C$, and $b, c \ne 0$.
\end{lem}
 For a proof of Lemma~\ref{DS}, see \cite{DS}.
\begin{lem}\label{top-equ}
Let $f: {\Bbb C} \to {\Bbb C}$ and $g: {\Bbb C} \to {\Bbb C}$ be two
continuous maps such that  $f = g$ on the outside of the unit disk.
If in addition, $f: \overline{\Delta} \to \overline{\Delta}$ and $g:
\overline{\Delta} \to \overline{\Delta}$ are both homeomorphisms,
then $f$ and $g$ are topologically equivalent to each other.
\end{lem}
\begin{proof}
Define $\theta_{2}(z) = z$ for $z \notin \Delta$ and $\theta_{2}(z)
= g^{-1} \circ f(z)$ for $z \in \Delta$. It follows that
$\theta_{2}: {\Bbb C} \to \Bbb C$ is a homeomorphism. Let
$\theta_{1} = id$. Then $f = \theta_{1}^{-1} \circ g \circ
\theta_{2}$. The Lemma follows.
\end{proof}
Let $\psi: \widehat{\Bbb C} - \overline{\Delta} \to \widehat{\Bbb C}
- \overline{D}$ be the  map defined right before Lemma~\ref{odd}.
Let $\eta : \widehat{\Bbb C} \to \widehat{\Bbb C}$ be a homeomorphic
extension of $\psi$. As before let $T(z) = \sin (z)$. By
Definition~\ref{pt}  $T$ is topologically equivalent to $T \circ
\eta $. Let $t \in [0, 1)$ be the number in Lemma~\ref{exist}.
Define a map $S: \Bbb C \to \Bbb C$ by
\begin{equation}\label{S}
S(z) = e^{2\pi it} \cdot (T\circ \eta)(z), \quad \forall z \in \Bbb
C.
\end{equation}
\begin{lem}\label{tran}
 $S$ is topologically equivalent to  both $T$ and
 $\widetilde{G}_{\theta}$.
\end{lem}
\begin{proof}
The  topological equivalence between $S$ and $T$ follows directly
from Definition~\ref{pt} and (\ref{S}).  From (\ref{ff-rr}),
(\ref{ff-ll}), (\ref{dg}) and the definition of $S$, it follows that
$S$ and $\widetilde{G}_{\theta}$ coincide with each other in the
outside of the unit disk, and are both homeomorphisms in the inside
of the unit disk. By Lemma~\ref{top-equ}, $S$ and
$\widetilde{G}_{\theta}$ are topologically equivalent with each
other.
\end{proof}

\begin{lem}\label{sin}
$T_{\theta}$ is topologically equivalent to $T$.
\end{lem}
\begin{proof}
From (\ref{nnff}) it follows that $T_{\theta}$ is topologically
equivalent to $\widetilde{G}_{\theta}$. By Lemma~\ref{tran}
$\widetilde{G}_{\theta}$  and $T$
 are topologically equivalent.  The lemma
follows.
\end{proof}

Now it is the time to prove the Main Theorem.
\begin{proof}
 By Lemmas~\ref{DS} and ~\ref{sin}, it follows that $T_{\theta}(z) = a + b\sin(cz +
d)$ holds for all $z \in \Bbb C$ where $a, b, c, d \in \Bbb C$ and
$b, c \ne 0$. Since $T_{\theta}$ is odd by Lemma~\ref{PH}, we get
\begin{equation}\label{odd-sin}
a + b\sin(c z + d) \equiv -a + b\sin(c z - d).
\end{equation}
Now by differentiating both sides of (\ref{odd-sin}), we get
$$
\cos(cz + d) \equiv \cos(cz - d).
$$
 It follows that
$$
\sin (d)\sin(cz) \equiv 0.
$$
Since $c \ne 0$, it follows that $d = k\pi$ for some integer $k$.
Therefore, we may assume that $T_{\theta}(z) = a + b \sin(cz)$ for
some $b, c \ne 0$. Since $T_{\theta}(0) = 0$, it follows that $a =
0$. This implies that $T_{\theta}(z) = b \sin (cz)$.

Since $T_{\theta}'(\pi/2) = 0$, it follows that $c$ is some odd
integer. By changing the sign of $b$, we may assume that $c$ is
positive. Suppose $c = 2 l +1$ for some integer $l \ge 0$. Let
$\Omega_{0}$ be the Siegel disk of $T_{\theta}$ centered at the
origin. For $k \in \Bbb Z$, let
$$
\Omega_{k} = \{z+ k\pi\big{|}\: z \in \Omega_{0}\}.
$$ Since
$T_{\theta}$ is odd by Lemma~\ref{PH}, $\Omega_{0}$ is symmetric
about the origin. That is, $\Omega_0 = - \Omega_0$.  It follows that
$T_{\theta}(\Omega_{k}) = (-1)^k \Omega_{0} = \Omega_0$. Therefore
each $\Omega_{k}$ is a component of $T_{\theta}^{-1}(\Omega_{0})$.

Let $D_{k}, k \in {\Bbb Z}$, be the domains  in
Corollary~\ref{dis-con}. By Corollary~\ref{dis-con}, the domains
$D_k$, $k \in \Bbb Z$, are all the components of $T^{-1}(\Delta)$.
 Recall that $D = D_{0}$.  Let $\psi: \widehat{\Bbb C}
\setminus \overline{\Delta} \to \widehat{\Bbb C}\setminus
\overline{D}$ be the map defined immediately after
Corollary~\ref{dis-con}.  We can extend $\psi$ to a homeomorphism
from $\widehat{\Bbb C}$ to itself. Let us still denote it by $\psi$.
Then $\psi (\Delta) = D_0$.  By the definitions of $G$, $G_{\theta}$
and  $\widetilde{G}_{\theta}$ ( cf. (\ref{ff-rr}), (\ref{ff-ll}) and
(\ref{dg})),  the domains $\psi^{-1} (D_{k})$, $k  \in \Bbb Z$, are
all the components of $\widetilde{G}_{\theta}^{-1}(\Delta)$. Let
\begin{equation}\label{dpi}
\widetilde{\Omega}_{k} = \phi \circ \psi^{-1} (D_{k}), \:\:k \in
\Bbb Z. \end{equation} Then  $\widetilde{\Omega}_0 = \Omega_0$ is
the Siegel disk of $T_{\theta}$ centered at the origin. From
(\ref{nnff})  it follows  that the domains $\widetilde{\Omega}_k$
are all the components of $T_{\theta}^{-1}(\widetilde{\Omega}_0) =
T_{\theta}^{-1}(\Omega_0)$. From (\ref{dpi}) and
Corollary~\ref{dis-con}, we have
\begin{itemize}
\item[1.] every $\partial \widetilde{\Omega}_{k}$ contains exactly two critical
points of $T_{\theta}$, and in particular, $\partial
\widetilde{\Omega}_0 = \partial \Omega_0$  contains $\pi/2$ and
$-\pi/2$,
\item[2.] $\partial \widetilde{\Omega}_{k} \cap
\partial \widetilde{\Omega}_{j} = \emptyset$ if $|k-j| > 1$,
\item[3.]  any critical point of $T_{\theta}$  belongs to the boundaries of exactly two domains, $\partial \widetilde{\Omega}_{k}$ and $\partial
\widetilde{\Omega}_{k+1}$, where  $k \in {\Bbb Z}$ is some integer,
and for every $k \in {\Bbb Z}$, $\partial \widetilde{\Omega}_{k}
\cap
\partial \widetilde{\Omega}_{k+1}$ contains a single point which is a critical
point of $T_{\theta}$.

\end{itemize}

Since every $\Omega_k$ is a component of
$T_{\theta}^{-1}(\Omega_0)$,  every $\Omega_{k}$ must be identical
with some $\widetilde{\Omega}_{j}$.

 $\bold{Claim:}$ $\Omega_{k} =
\widetilde{\Omega}_{k}$ for all $k \in {\Bbb Z}$.   Let us prove the
$\bold{Claim}$. We have already known that $\Omega_{0} =
\widetilde{\Omega}_{0}$. Since $-\pi/2 \in \partial \Omega_0$ and
$\Omega_1 = \Omega_0 + \pi$, we get $\pi/2 \in \partial \Omega_{1}$.
Since  $\pi/2 \in
\partial D_{1}$ and $\phi \circ \psi^{-1}(\pi/2) = \pi/2$, we get $\pi/2 \in \phi\circ \psi^{-1} (\partial D_1) = \partial \widetilde{\Omega}_1$.
By the third property above,  among $\partial
\widetilde{\Omega}_{k}s$, only  $\partial \widetilde{\Omega}_{0}$
and $\partial \widetilde{\Omega}_{1}$ contain $\pi/2$.  Since
$\Omega_1$ must be identical with one of the
$\widetilde{\Omega}_{k}s$,  $\Omega_1$ must be identical with either
$\widetilde{\Omega}_0$ or $\widetilde{\Omega}_1$.  Since
${\Omega}_1$  can not be identical with $\widetilde{\Omega}_0 =
\Omega_0$, it must be identical with $\widetilde{\Omega}_1$.

Now assume for some $ k \ge 1$, we have $\Omega_i =
\widetilde{\Omega}_i$ for $0 \le i \le k$.  Let us prove
$\Omega_{k+1} = \widetilde{\Omega}_{k+1}$. By the second property
above, among the $\partial \widetilde{\Omega}_{i}s$, only $\partial
\widetilde{\Omega}_{k-1}$ and $\partial \widetilde{\Omega}_{k+1}$
intersect with $\partial \widetilde{\Omega}_{k} = \partial
\Omega_k$. Since $\partial \Omega_{k+1}$ intersect with $\partial
\widetilde{\Omega}_{k} = \partial \Omega_{k}$ at the point $\pi/2 +
k \pi$, it follows that ${\Omega}_{k+1}$ must be identical with
either $\partial \widetilde{\Omega}_{k-1}$ or  $\partial
\widetilde{\Omega}_{k+1}$. Since $\partial \widetilde{\Omega}_{k-1}
= \Omega_{k-1}$ by assumption, we get ${\Omega}_{k+1} =
\widetilde{\Omega}_{k+1}$. By induction we get $\Omega_k =
\widetilde{\Omega}_{k}$ for all $k \ge 0$.  Using the same argument
we can get $\Omega_k = \widetilde{\Omega}_{k}$ for all $k \le 0$.
This implies that
$$
\Omega_k = \widetilde{\Omega}_{k} \hbox{  for all  } k \in \Bbb Z.
$$
The $\bold{Claim}$  has been proved.

Note that  $\partial \widetilde{\Omega}_{k}\cap
\partial \widetilde{\Omega}_{k+1} = \partial \Omega_k \cap \partial
\Omega_{k+1} = \{\pi/2 + k \pi\}$ for all $k \in \Bbb Z$. From the
property (3) above every critical point of $T_{\theta}$ is contained
in the boundary of some $\widetilde{\Omega}_k$. It follows that the
set of the critical points of $T_{\theta}$ is equal to
$$\{\pi/2 + k \pi\:\big{|}\: k \in {\Bbb Z}\}.$$ This implies that
$c = 1$. It follows that $b = e^{2 \pi i \theta}$ and therefore
$T_{\theta}(z) = f_{\theta}(z)$. By the construction, the Siegel
disks of $T_{\theta}$ centered at the origin is a Jordan curve
passing through exactly two critical points, $\pi/2$ and $-\pi/2$.
This completes the proof of the Main Theorem.

\end{proof}
\section{The integrability of $\nu$}
\subsection{Reduce the proof of the integrability of $\nu$ to the Main Lemma}
\begin{pro}\label{poo8}{\rm
The Main Lemma implies the integrability of $\nu$.}
\end{pro}
Let us first prove an immediate lemma.
\begin{lem}\label{power law}
If there exist $C_2 > 0$, $0< \delta_2 < 1$ and an integer $N  \ge
N_{0}$ such that $area(X_{n}) < C_2 \cdot  \delta_2^{n}$ holds for
all $n \ge N$, then $\nu$ satisfies the
condition~(\ref{integrability-p}) for some constants $M, \alpha > 0$
and $0< \epsilon_0 < 1$.
\end{lem}
\begin{proof}
Suppose there exist $C_2 > 0$, $0< \delta_2 < 1$ and an integer $N
\ge N_{0}$ such that $area(X_{n}) < C_2 \cdot \delta_2 ^{n}$ holds
for all $n \ge N$.    Let $\lambda > 1$ be the constant in
Lemma~\ref{Yoccoz extension}. Let
\begin{equation}\label{choice-e} \epsilon_{0} = \min\{
 \frac{2}{\lambda N + 1}, \frac{1}{1 +
\lambda}\}.
\end{equation}

Let  $0 < \epsilon < \epsilon_{0}$.  Since $\epsilon_0 \le
\frac{2}{\lambda N + 1}$ by (\ref{choice-e}),   there exists a
unique integer $n \ge N$  such that
\begin{equation}\label{tk}
\frac{2}{\lambda (n+1) + 1} \le  \epsilon <  \frac{2}{\lambda n +
1}.
\end{equation}
From the left hand of (\ref{tk}) we have
$$
n \ge \frac{2}{\lambda \epsilon} - \frac{1}{\lambda} -1.
$$
Since $0< \epsilon < \epsilon_0$ and $\epsilon_0 \le \frac{1}{1 +
\lambda}$ by (\ref{choice-e}) we have
$$
\frac{1}{\epsilon} > \frac{1}{\epsilon_0} \ge  1 + \lambda.
$$
From the above two inequalities we have
\begin{equation}\label{fgaq}
 n \ge  \frac{1}{\lambda\epsilon} + \frac{1}{\lambda\epsilon} - \frac{1}{\lambda} - 1   =
 \frac{1}{\lambda \epsilon} + \frac{1}{\lambda}(\frac{1}{\epsilon} - 1 - \lambda)
\ge \frac{1}{\lambda \epsilon}.
\end{equation}

 Now suppose $z \in \Bbb C$ is an arbitrary  point such that $|\nu(z)|
> 1 - \epsilon$.  Since $\epsilon < \frac{2}{\lambda n + 1}$ by the right hand of
(\ref{tk}) we have
$$
|\nu(z)| > \frac{\lambda n -1}{\lambda n + 1}.
$$
This implies that the dilatation at $z$ is greater than $\lambda n$.
From the  third assertion of Lemma~\ref{Yoccoz extension} and the
definition of  $X_n$, we have
$$
z \in Y_{n+2} \cup X_{n+2}.
$$
Since $z$ is an arbitrary point such that $|\nu(z)|
> 1 - \epsilon$, the above inclusion relation implies $$
area(\{z|\:  |\nu(z)| > 1 - \epsilon \} \le area (Y_{n+2}) +
area(X_{n+2}). $$  By (\ref{area-in}) we have $$area (Y_{n+2})
 < C_{1} \cdot  \delta_{1}^{n+2} $$ where $1< C_1 < \infty$ and $0<
\delta_1 < 1$ are some constants depending only on $g_\theta$.  By
the assumption of the lemma, we have $$area(X_{n+2})  <  C_{2} \cdot
\delta_{2}^{n+2}. $$

Take $M = C_1 + C_2$ and $\alpha = \frac{1}{\lambda} \cdot \log
\delta^{-1}$ where $\delta = \max\{\delta_1, \delta_2\}$. Note that
$n \ge \frac{1}{\lambda \epsilon}$ by (\ref{fgaq}). From all the
above we get
$$
area(\{z|\:  |\nu(z)| > 1 - \epsilon \}\le area (Y_{n+2}) +
area(X_{n+2}) $$$$ <  C_{1} \delta_{1}^{n+2} + C_{2}
\delta_{2}^{n+2} <  M\delta^{n}  = M e^{- n \log \frac{1}{\delta}}
\le M e^{- \frac{\log \frac{1}{\delta}}{\lambda \epsilon}}  = M
e^{-\frac{\alpha}{\epsilon}}.
$$
 This completes the
proof of Lemma~\ref{power law}.
\end{proof}

Let us prove Proposition~\ref{poo8} now.

\begin{proof}
Assume the condition of the Main Lemma holds. We may assume that
$\epsilon
> \sqrt{\delta}$.  This is because  otherwise we may let $\epsilon =
\frac{1+\sqrt{\delta}}{2}$. Then $\epsilon > \sqrt{\delta}$ and the
condition (\ref{pz-cci}) still holds. Let
$$
\sigma = \frac{C}{ \epsilon^{2} - \delta}.
$$
Since $\epsilon > \sqrt{\delta}$ we have $\sigma > 0$. From
(\ref{pz-cci}) we have
\begin{equation}\label{ddse}
area(X_{n+2}) - \sigma \epsilon^{n+2} \le \delta ( area(X_{n}) -
\sigma \epsilon^{n}).
\end{equation}
Note that  $area(\Bbb C) = \pi$. In particular, we have
$area(X_{N_{1}}) \le \pi$.

Let $n \ge N_1$. Then there are two cases.

In the first case, $n  = N_{1} + 2k$ for some integer $k \ge 0$.
From (\ref{ddse}) we have
$$
area(X_{n}) \le \delta^{k}(area(X_{N_{1}}) -\sigma \epsilon^{N_{1}})
+ \sigma \epsilon^{n}
$$
$$
\le \pi \delta^{k}  + \sigma \epsilon^{n}
$$
Note that $k = (n-N_{1})/2$ and  $\sqrt\delta < \epsilon$ by
assumption. We thus have
$$
area(X_{n}) \le \pi \delta^{k}  + \sigma \epsilon^{n} <  \pi
\delta^{-N_{1}/2} \epsilon^{n} + \sigma \epsilon^{n} < (\pi
\delta^{-N_{1}/2} + \sigma) \epsilon^{n}.
$$
In the second case,  $n  = N_{1} +1 + 2k$ for some integer $k \ge
0$. Using the same reasoning we have
$$
area(X_{n}) \le \pi \delta^{k}  + \sigma \epsilon^{n} <  \pi
\delta^{-(N_{1}+1)/2} \epsilon^{n} + \sigma \epsilon^{n} < (\pi
\delta^{-(N_{1}+1)/2} + \sigma) \epsilon^{n}.
$$
Let $C_{2} =  \pi \delta^{-(N_{1}+1)/2} + \sigma$ and $\delta_{2} =
\epsilon$.  Thus in both the cases, we  have
$$
area(X_{n}) < C_{2} \delta_{2}^{n}.
$$
for all $n \ge N_{1}$. The integrability of $\nu$ then follows from
Lemma~\ref{power law}.  This completes the proof of
Proposition~\ref{poo8}.
\end{proof}

\subsection{Proofs of Lemma~\ref{covering lemma} and Corollary~\ref{spherical area}}
Let us first prove Lemma~\ref{covering lemma}.

\begin{proof}

Let us simply denote $B_{r_{i}}(x_{i})$ by $B_{i}$ for all $i \in
\Lambda$.   Let $\Lambda_0 \subset \Lambda$  be the subset which
consists of all those $i$ such that $B_{i}$ is maximal, that is,
$B_{i}$ is not contained in any other $B_{j}$.   Let $\Sigma$ be the
class which consists of all the non-empty subsets of $\Lambda_0$
such that for every $\sigma \in \Sigma$, the sets
$$
B_{i},\:\: i \in \sigma
$$ are disjoint with each other. Clearly any subset of $\Lambda$
which contains exactly one element must belong to $\Sigma$.    This
means that $\Sigma$ is not empty. Since $\Lambda$ and thus
$\Lambda_0$ is finite, $\Sigma$ is a finite set.  Thus there exist a
$\sigma_{0} \in \Sigma$ such that
$$
m (\bigcup_{i \in \sigma_{0}} B_{i}) = \max_{\sigma\in\Sigma} m
(\bigcup_{i \in \sigma} B_{i})
$$
where $m(\cdot)$ denotes the Euclidean area.

Let $L = 8 K + 9$. Now let us prove that  for any $i \in \Lambda$,
there is some $j \in \sigma_{0}$ with $U_{i} \subset
B_{Lr_{j}}(x_{j})$.  Since $U_i \subset B_{Kr_i}(x_i)$ by
(\ref{m-qq}), it suffices to prove that
\begin{equation} \label{provd}B_{Kr_i}(x_i) \subset B_{Lr_{j}}(x_{j})  \hbox{     for
some  }j \in \sigma_0.\end{equation}  We need only prove
(\ref{provd}) for $i \in \Lambda_0$. This is because  if $i \notin
\Lambda_0$, then there is an $i' \in \Lambda_0$ such that
$B_{r_i}(x_i) \subset B_{r_{i'}}(x_{i'})$. Since $K > 1$,  it
follows that  $B_{Kr_i}(x_i) \subset B_{Kr_{i'}}(x_{i'})$.  If
(\ref{provd}) holds for some $j \in\sigma_0$ and $i' \in \Lambda_0$,
then $B_{Kr_{i}}(x_{i}) \subset B_{Kr_{i'}}(x_{i'}) \subset
B_{Lr_j}(x_j)$.

We  may further assume that $i \notin \sigma_{0}$.  This is because
if $i \in \sigma_{0}$, by (\ref{m-qq}) we have $U_{i} \subset
B_{Kr_{i}}(x_{i})\subset B_{Lr_{i}}(x_{i})$. So in the following we
assume that $i \in \Lambda_0 \setminus\sigma_0$.

By the maximal property of $\sigma_{0}$, the disk $B_{i}$ must
intersect at least one $B_{l}$ for some $l \in \sigma_{0}$. Let
$$
\Theta = \{l \in \sigma_{0}\:\big{|}\: B_{i} \cap B_{l} \ne
\emptyset\}.
$$
It follows from the maximal property of $\sigma_{0}$ again that
\begin{equation}\label{cover area}
m(B_{i}) \le m(\bigcup_{l\in \Theta} B_{l}).
\end{equation}
(This is because otherwise, one may use $B_{i}$ to replace all the
disks $B_{l}, l \in \Theta$, then the total Euclidean area will be
increased, and this contradicts with the maximal property of
$\sigma_{0}$)

By the assumption on $\Lambda_0$ in the beginning of the proof, $B_l
\nsubseteq B_i$ and $B_i \nsubseteq B_l$ for any $l \in \Theta$. For
$l \in \Theta$,  since $B_i \cap B_l \ne \emptyset$, it follows that
the boundary circle of $B_{l}$ intersects the boundary circle of
$B_{i}$. This implies  that
\begin{equation}\label{edp}
r_{i} \le 8  \max _{l \in \Theta} r_{l}.
\end{equation}
Because otherwise, the union of $B_{l}$, $l \in \Theta$,  would be a
proper subset of the annulus
$$\{z\:\big{|}\: \frac{3}{4} r_{i} < |z - x_{i}| < \frac{5}{4}
r_{i}\},$$ whose  Euclidean area  is equal to that  of $B_{i}$. This
contradicts  with (\ref{cover area}).  Since $\Theta$ is a finite
set, there is a $j \in \Theta$  such that
\begin{equation}\label{che}
r_{j} = \max _{l \in \Theta} r_{l}.\end{equation}  Let $d(\cdot,
\cdot)$ denote the distance with respect to the Euclidean metric.
Let $z \in B_{Kr_{i}}(x_{i})$ be an arbitrary point.  Since
 $r_i \le 8 r_j$ by (\ref{edp}) and (\ref{che}), we have
$$
d(z, x_i) < K r_i < 8K r_j.
$$
Since $B_i$ intersects $B_j$ we have
$$
d(x_i, x_j) < r_i + r_j \le 8 r_j +  r_j = 9 r_j.
$$
It follows that
$$
d(z, x_j) \le d(z, x_i) + d(x_i, x_j) < (8K + 9) r_j = L r_j.
$$
Since $z$ is an arbitrary point in $B_{Kr_{i}}(x_{i})$,
(\ref{provd}) follows.
 The proof of Lemma~\ref{covering lemma} is completed.
\end{proof}

Recall that $\Omega = {\Bbb C} \setminus \overline{\Delta}$  and
${\rm diam}(\cdot)$, ${\rm dist}(\cdot, \cdot)$ denote respectively
the diameter and the distance   with respect to the Euclidean
metric. Now let us prove Corollary~\ref{spherical area}.

\begin{proof}
Let $\Lambda_0 \subset \Lambda$ and $\sigma_{0} \subset \Lambda_0$
be given as in the proof of Lemma~\ref{covering lemma}.   Let $L =
8K + 9$. Then for any $i \in \Lambda_0$, from the proof of
Lemma~\ref{covering lemma}, there is
 some $j \in \sigma_{0}$ such that
\begin{itemize}
\item[1.] $B_i \cap B_j\ne \emptyset$,
\item[2.] $B_{Kr_i}(x_i) \subset B_{Lr_j}(x_j)$.
\end{itemize}  Now for each $j \in \sigma_0$, let $\Theta_j \subset \Lambda_0$ be the subset consisting of all the  $i \in
\Lambda_0$ so that the above two properties hold. Then $\bigcup
\Theta_j = \Lambda_0$.  For each $i \in \Theta_j$, let $\tau_{i, j}
\subset \Lambda$ be the subset which consists of all those $l$ such
that $B_l \subset B_i$.  In particular, $i \in \tau_{i,j}$. For $j
\in \sigma_0$, let
$$\Delta_j = \bigcup_{i \in \Theta_j} \tau_{i,j}.$$ Then
$$
\bigcup_{j \in \sigma_0} \Delta_j = \Lambda.
$$
For $j \in \sigma_0$, let
$$
X_j = \bigcup_{l \in \Delta_j} U_l.
$$  We thus  have
\begin{equation}\label{dis-u}
\bigcup_{i \in \Lambda} U_i = \bigcup_{j\in\sigma_0} X_j.
\end{equation}
Claim 1. $$X_j \subset B_{Lr_j}(x_j), \quad \forall j \in
\sigma_0.$$ Let us prove the claim. Let $z \in X_j$ be an arbitrary
point. Then $z \in U_l$ for some $l \in \Delta_j$. Then $l \in
\tau_{i, j}$ for some $i \in \Theta_j$. Thus $B_l \subset B_i$ and
$B_{Kr_i} \subset B_{Lr_j}(x_j)$. Since $K > 1$ and $B_l \subset
B_i$, we can easily deduce that $B_{Kr_l}(x_l) \subset
B_{Kr_i}(x_i)$.  Thus $B_{Kr_l}(x_l) \subset B_{Lr_j}(x_j)$. So $z
\in U_l \subset B_{Kr_l}(x_l) \subset B_{Lr_j}(x_j)$.  Since $z$ is
an arbitrary point in $X_j$, the Claim 1 follows.

From the Claim 1, we have
\begin{equation}\label{eue}
m(X_j) \le L^2 \cdot m(B_{r_j}(x_j))
\end{equation}
where $m(\cdot)$ denotes the area with respect to the Euclidean
metric.

Claim 2. There is a constant $\eta(K)$ depending only on $K$ such
that
\begin{equation}\label{c2c2}
\frac{1 + |\omega|^2}{1 + |\zeta|^2} < \eta(K), \forall \zeta \in
X_j \hbox{  and  } \forall \omega \in B_{r_j}(x_j).
\end{equation} Note that $\zeta \in U_l$ for some $l \in \tau_{i,j}$
and $i \in \Theta_j$. Thus $U_l \cap U_i \ne \emptyset$ and $U_i
\cap U_j \ne \emptyset$. This means  to prove (\ref{c2c2}) it
suffices to prove that there exists a constant $\ell(K)$ such that
\begin{equation}\label{imc}
\frac{1 + |z|^2}{1 + |\xi|^2} < \ell(K), \forall k \in \Lambda
\hbox{  and  } \forall z, \xi \in U_k.
\end{equation}
Because one can take $\omega' \in U_j \cap U_i$ and $\zeta' \in U_i
\cap U_l$, then using (\ref{imc}) to $\omega, \omega'\in U_j$,
$\omega', \zeta' \in U_i$, and $\zeta', \zeta \in U_l$, we have
$$
\frac{1 + |\omega|^2}{1 + |\zeta|^2} = \frac{1 + |\omega|^2}{1 +
|\omega'|^2}\cdot \frac{1 + |\omega'|^2}{1 + |\zeta'|^2}\cdot
\frac{1 + |\zeta'|^2}{1 + |\zeta|^2} < \ell^3(K).
$$ This implies (\ref{c2c2}) by taking $\eta(K) = \ell^3(K)$. Now
let us prove (\ref{imc}).

We may assume that  $|\xi| \le |z|$ since otherwise the left hand of
(\ref{imc}) is less than $1$. Since $U_k \subset \Omega$ we have
$|\xi|
> 1$. Using the fact that $|\xi|
> 1$, we have
$$
\frac{1 + |z|^{2}}{1 + |\xi|^{2}}  <  \frac{1 + |z|^{2}}{ |\xi|^{2}}
<  1 + \big{(}\frac{|z|}{|\xi|}\big{)}^{2}  <  1 + \big{(}\frac{|z|
+|\xi| -1}{|\xi|}\big{)}^{2}  =  1 + \big{(}1 + \frac{|z|
-1}{|\xi|}\big{)}^{2}.
$$
Since
$$
\frac{|z|- 1}{|\xi|} < \frac{|z|-1}{|\xi| -1} = 1 + \frac{|z|-
|\xi|}{|\xi|-1} \le 1 + \frac{{\rm diam}(U_k)}{{\rm dist}(U_k, \Bbb
T)} \le 1 + K.
$$ The last inequality comes from (\ref{s-e}).
We thus have
$$
\frac{1 + |z|^{2}}{1 + |\xi|^{2}}  \le  1 + (2 + K)^{2}.
$$ This proves (\ref{imc}) by taking $\ell(K) = 1 + (2 + K)^{2}$ and
the Claim 2 follows.

 Since the spherical  area form is given by
$$\frac{|dz|^{2}}{(1+ |z|^{2})^{2}},$$ from (\ref{eue}) and (\ref{c2c2}), it follows that
$$
area(X_{j}) \le L^2 \cdot \eta^{2}(K)  \cdot area
(B_{r_{j}}(x_{j})).
$$
Since  $B_{r_j}(x_j) \subset V_j$ and all $B_{r_{j}}(x_{j})$, $j \in
\sigma_{0}$, are disjoint, from (\ref{dis-u})
Corollary~\ref{spherical area} then follows by taking
$$\lambda(K) = L^{2} \cdot \eta^{2}(K).$$

\end{proof}
\subsection{Proof of Lemma~\ref{hn-ps}}

\begin{proof}  Let $$\Sigma =
\Bbb C \setminus (\Bbb R \setminus  (-1, 1))$$ denote the slit
plane.  For $d > 0$, let $U_d((-1, 1))$ denote the hyperbolic
neighborhood of $(-1, 1)$ in $\Sigma$ with hyperbolic distance $d$.
It is known that the boundary of $U_d((-1,1))$ is the union of two
arc segments of Euclidean circles which are symmetric about $\Bbb
R$, and moreover, the exterior angle $\alpha$ formed by the boundary
arc of $U_d((-1,1))$ and $\Bbb R$ is determined by the formula $d =
\ln \cot(\frac{\alpha}{4})$ (cf. \cite{Ya}).

For an open arc $I \subset \Bbb T$, let $\Omega_{I}$ denote  the
hyperbolic Riemann surface defined by (\ref{slt}).  Without loss of
generality,  we may assume that $i$ is the mid-point of $I$. Since
$\Sigma$ is simply connected, there exists a unique holomorphic
covering map
$$
\pi: \Sigma \to \Omega_{I}
$$
with  $\pi(0) = i$ and $\pi'(0) > 0$.

Claim 1.  $\pi$ maps $(-1, 1)$ homeomorphically to $I$, and
moreover, $\pi$ is symmetric in the following sense
\begin{equation}\label{symmetric} \pi(\bar{z}) = (\pi(z))^{*},
\:\:\forall \:z \in \Sigma,\end{equation} where $w^{*} = 1/\bar{w}$
denotes the symmetric image of $w$ about $\Bbb T$. Let us prove the
Claim 1 as follows. Considering the  map $\pi^*: \Sigma \to
\Omega_I$ given by
$$
\pi^*: z \mapsto (\pi (\bar{z}))^{*}.
$$
It is easy to check that $\pi^*$ is a holomorphic covering map  with
$\pi^{*}(0) = i$ and $(\pi^{*})'(0) > 0$. Thus $\pi^{*} = \pi$. This
implies (\ref{symmetric}). In particular, $\pi(z) \in I$ for $z \in
(-1, 1)$. Since $\pi$ is a local isomorphism, $\pi$ must maps $(-1,
1)$ homeomorphically to $I$. This proves the Claim 1.

Recall that $i$ is the mid-point of $I$. Let $B_{1}(i)$ denote the
Euclidean disk with center $i$ and radius $1$.  Let $S =
B_{1}(i)\cap \Bbb T$. Then $\overline{I} \subset S$ when $I$ is
small.

Claim 2. There exist $a$ and $b$ with $ a< -1 < 1 <
  b$ and a simply connected  domain $U\subset \Bbb C$  such that
\begin{itemize}
\item[1.] $U \cap \Bbb R = (a, b)$,
\item[2.] $\pi$ can be holomorphically extended across $(a, -1] \cup [1, b)$, and  $\pi: U \to
B_{1}(i)$ is a holomorphic isomorphism which
 maps $(a, b)$ homeomorphically to $S$.
 \end{itemize}  Let us prove the Claim 2 as follows.   Let $\sigma$
denote  the branch of the inverse of $\pi$ which
 sends $I$ to $(-1, 1)$.  Since $B_{1}(i) \cap \Omega_I$ is simply connected,  $\sigma$ can be holomorphically
 extended to $B_{1}(i)\cap \Omega_I$.

 Let us first prove that
 $\sigma$ can be continuously extended to $S \setminus I$ from
 both  the inside and the outside of $\Delta$. We prove this by contradiction. Suppose this
 were not true. Then there would be a point $z \in S \setminus I$
 and a continuous curve $\gamma: [0, 1) \to B_{1}(i)\cap \Omega_I$
 such that as $t \to 1$, $\gamma(t) \to z$ and
 $(\sigma \circ  \gamma)(t)$ accumulates on a non-trivial boundary segment of
 $\Sigma$.  Let $\tau:  \Sigma \to \Delta$ and $\eta: \widehat{\Bbb C} \setminus (\Bbb T \setminus I)
\to \Delta$ be two holomorphic isomorphisms. Then
$$
h = \eta \circ \pi \circ \tau^{-1}: \Delta \to \Delta
$$
is a bounded and holomorphic map.  It is easy to see that $h(\Delta)
= \Delta \setminus \{\eta(\infty), \eta(0)\}$. In particular, $h$ is
not a constant. Let $\xi: [0, 1) \to \Delta$ be a continuous curve
given by
$$
\xi(t) = \tau \circ \sigma \circ \gamma(t), \quad \forall t \in
[0,1).
$$
Note that  each point in $\partial \Sigma$, except $1$ and $-1$,
corresponds to two equivalent classes of null chains of cross cuts
in $\Sigma$: one of them is contained in the upper half plane, and
the other one is contained in the lower half plane.  Each of the two
points $1$ and $-1$ corresponds to exactly one equivalent class of
null chains of cross cuts. Thus each point in $\partial \Sigma
\setminus \{1, -1\}$ represents two distinct primes ends, and each
of the two points $1$ and $-1$ represents exactly one prime end.
Since as $t \to 1$, $(\sigma \circ \gamma)(t)$ accumulates on a
non-trivial boundary segment of $\Sigma$, and since $\tau: \Sigma
\to \Delta$ is a holomorphic isomorphism, by Caratheodory's theory
on prime ends, it follows that as $t \to 1$, $\xi(t) = \tau \circ
\sigma \circ \gamma (t)$ accumulates to some non-trivial arc segment
of $\Bbb T$. Now we will show that as $t \to 1$, $h(\xi(t))$
converges to a point in $\Bbb T$.  To see this, note that
 for the
domain $\widehat{\Bbb C} \setminus (\Bbb T \setminus I)$, each point
in $\Bbb T \setminus I$, except the two end points of $I$,
represents two distinct prime ends (one of them corresponds to the
null chain of cross cuts inside $\Delta$, and the other one
corresponds to the null chain of the cross cuts outside $\Delta$),
and each of the two end points of $I$ represents exactly one prime
end. Since as $t \to 1$, $\gamma(t)$ converges to a point in $z \in
S \setminus I \subset \Bbb T \setminus I$, and since $\eta:
\widehat{\Bbb C} \setminus (\Bbb T \setminus I) \to \Delta$ is a
holomorphic isomorphism, again by Caratheodory's theory on prime
ends,
$$h(\xi(t))= \eta \circ \pi \circ \tau^{-1} \circ \tau \circ \sigma
\circ \gamma (t) = \eta \circ \gamma(t) $$  converges to some point
in $\Bbb T$.   Since $h$ is a bounded and holomorphic function in
$\Delta$, and $\xi(t)$ accumulates on some non-trivial arc segment
of $\Bbb T$ as $t \to 1$,  this would imply that $h$ is a
constant(cf. Lemma 4.3, Chapter XIV of \cite{SL}). But this is a
contradiction. This implies that $\sigma$ can be continuously
extended to $S \setminus I$ from both the inside and the outside of
$\Delta$. Since $\pi$, and thus $\sigma$ is symmetric about $\Bbb
T$, the two extensions must coincide on $S \setminus I$. Thus
$\sigma$ can be extended to a holomorphic function in $B_1(i)$.

 Let us prove that
$\sigma$ is injective in $B_{1}(i)$. Suppose this were not true.
Since $\sigma$ is injective in $B_{1}(i) \cap \Omega_I$,  we would
either have two distinct points $z, z' \in S$ with  $\sigma(z) =
\sigma(z')$ or have two  points  $z' \in S$ and $z \in B_1(i)
\setminus S$ with $\sigma(z) = \sigma(z')$.  If the first
possibility occurs, since $\sigma(S) \subset \Bbb R$ by symmetry,
there would be a $\zeta \in S$ such that $\sigma'(\zeta) = 0$. This
implies that in a small neighborhood of $\zeta$, we could find two
distinct points $\omega$ and $\omega'$ in $B_{1}(i) \cap \Omega_I$
such that $\sigma(\omega) = \sigma(\omega')$. This is a
contradiction with the injectivity of $\sigma$ in $B_1(i) \cap
\Omega_I$.  If the second possibility occurs,  by symmetry we would
have $\sigma(z^*) = \overline{\sigma(z)} = \overline{\sigma(z')} =
\sigma(z') = \sigma(z)$ (The equation $\overline{\sigma(z')} =
\sigma(z')$ comes from the fact that $ \sigma(z') \in \Bbb R$).
Since both $z$ and $z^{*}$ belong to $B_1(i) \cap \Omega_I$,   this
contradicts with the injectivity of $\sigma$ in $B_1(i) \cap
\Omega_I$.  So $\sigma$ is injective in $B_{1}(i)$.  Let $U =
\sigma(B_1(i))$. Then $U$ is a simply connected domain.  By
symmetry, $U$ is symmetric about $\Bbb R$ and $U \cap \Bbb R =
\sigma(S)$ is an interval containing $[-1, 1]$.   Let $(a, b) =
\sigma(S)$.  The inverse $\sigma^{-1}$ realizes the extension of
$\pi$ to $U$. This proves the Claim 2.

From the Claim 2,  the domain $U = \sigma(B_{1}(i))$ contains $[-1,
1]$ and
$$
{\rm mod} (\sigma(B_{1}(i)) \setminus [-1, 1]) = {\rm mod} (B_{1}(i)
\setminus \bar{I}).
$$
Note  that ${\rm mod} (B_{1}(i) \setminus \bar{I})$, and thus ${\rm
mod} (\sigma(B_{1}(i)) \setminus [-1, 1])$,  can be arbitrarily big
provided that $I$ is small enough. So for any $R > 0$,
$\sigma(B_{1}(i))$ contains the disk $\{z\:|\: |z| < R\}$ provided
that $I$ is small.  Since ${\rm diam}(U_{d}((-1, 1))) \le C(d)$
where $C(d)
> 0$ is some constant depending only on $d$ (cf. \cite{Ya}),  for $d
> 0$ given, we have
\begin{equation}\label{ood}
\overline{U_{d}((-1, 1))} \subset \sigma(B_{1}(i))
\end{equation}
and  $${\rm mod} (\sigma(B_{1}(i)) \setminus \overline{U_d((-1,
1))})$$ can be arbitrarily  big provided that $I$ is small enough.
  Since $\pi: \Sigma \to \Omega_I$ preserves the hyperbolic
metric, and $\pi: \sigma(B_{1}(i)) \to B_{1}(i)$ is a holomorphic
isomorphism,  from (\ref{ood})  it follows that $$ \pi(U_{d}([-1,
1]) = \Omega_d(I) \subset B_1(i).$$   It follows that $\Omega_d(I)$
is symmetric and simply connected. Since $\partial \Omega_d(I)$ is
the $\pi$-image of the union of two arc segments of Euclidean
circles which are symmetric about $\Bbb R$, by the symmetry of $\pi$
(see (\ref{symmetric})), $\partial \Omega_d(I)$ is the union of two
real analytic curve segments which are symmetric about $\Bbb T$.
Since
$${\rm mod} (\sigma(B_{1}(i)) \setminus \overline{U_d((-1, 1))})$$
can be arbitrarily  big provided that $I$ is small enough,  the
distortion of
 $\pi$ in $\overline{U_{d}((-1, 1))}$ can be arbitrarily  small provided that $I$ is small enough. This shows that the domain
 $\Omega_{d}(I)$ is more and more like $U_{d}((-1, 1))$ as $I$ becomes  smaller and smaller. Since $\pi$ is conformal, the
 four exterior angles  formed by the boundary of $\Omega_{d}(I)$ and the unit circle, all of which are the same and denoted as $\alpha$,
 are
 equal to those formed by the boundary of $U_{d}((-1, 1))$ and the real
 line.  Lemma~\ref{hn-ps} follows.
\end{proof}

\subsection{Reduce the Main Lemma to Lemma~\ref{pre-lem}}


\begin{pro}\label{mred}{\rm
Lemma~\ref{pre-lem} implies the Main Lemma.}
\end{pro}
\begin{proof}
Let $R > 1$. Define
$$
X_{n+2}^{R} = \{z \in X_{n+2} \:\big{|}\:  1  \le  |z| \le  R  \}.
$$
It is clear that $X_{n+2}^R$ is a Lebesgue measurable set. For any
$\eta > 0$, by a standard result in Lebesgue measure theory (cf.
Page 127, Ex. 3.43 of \cite{MW}), there is a closed set $F \subset
X_{n+2}^R$ such that \begin{equation}\label{rct} area(X_{n+2}^{R}
\setminus F) < \eta.\end{equation}  Since $Z_{n}$ is open and $F$ is
a bounded and closed set, it follows that $F\setminus  Z_{n}$ is a
compact set.

For any point $x \in F \setminus Z_{n}$,  by Lemma~\ref{pre-lem},
$x$ is associated to some $K$-admissible pair $(U_i, V_i)$ for some
uniform $1 < K < \infty$. This implies that the sets $U_i s$ form an
open cover of $F \setminus Z_{n}$. Since $F \setminus Z_{n}$ is a
compact set, we have finitely many pairs $(U_{i}, V_{i})$, $i \in
\Lambda$, such that
\begin{equation}\label{oneh}
F \setminus Z_{n} \subset \bigcup_{i \in \Lambda} U_{i},
\end{equation}

By Theorem~\ref{SH}   and Lemma~\ref{hn-ps},  there exist constants
$L_1, L_2
> 1$ and $0 < \epsilon < 1$ independent of $n$ such that for any
interval $I$ in the dynamical partition of level $n$, we have (1)
$|I| < L_1 \cdot \epsilon^n$  and (2)  ${\rm diam}(H_{\alpha}(I))
<L_2 \cdot |I|$.  From the definition of $Z_n$ (cf. (\ref{ZN})) that
$Z_n$ is contained in the annulus $\{z\:|\: 1 < |z| < 1 + L_1 \cdot
L_2 \cdot \epsilon^n\}$. Thus we get
\begin{equation}\label{m-z}
area(Z_{n}) < C \epsilon^{n}
\end{equation} where $1 < C < \infty$  is some constant independent
of $n$.

We now claim that there is a $0< \delta < 1$ such that
\begin{equation}\label{m-e-1}
area (X_{n+2}) \le C \epsilon^{n}  +  \delta \:area(X_{n}).
\end{equation}
Let us prove the claim now.  From (\ref{rct}) and (\ref{oneh}),  we
have
\begin{equation}\label{m-e-2}
area ({X_{n+2}^{ R}}) \le area(Z_{n})  + area(\bigcup_{i \in
\Lambda} U_{i}) + \eta.
\end{equation}
By  Corollary~\ref{spherical area}, we have
\begin{equation}\label{m-e-3}
area(\bigcup_{i \in \Lambda} U_{i}) \le \lambda(K) \cdot
area(\bigcup_{i \in \Lambda} V_{i}).
\end{equation}
Since $V_i \subset X_{n} \setminus X_{n+2}$ for all $i \in \Lambda$,
we have
$$
area(\bigcup_{i \in \Lambda} V_{i}) \le area(X_{n}) - area(X_{n+2})\
\le area(X_{n}) - area(X_{n+2}^R).
$$
This, together with (\ref{m-e-3}), implies that
\begin{equation}\label{rree}
area(\bigcup_{i \in \Lambda} U_{i}) \le \lambda(K) \cdot(area(X_{n})
- area(X_{n+2}^R)).
\end{equation}
From (\ref{m-e-2}) and (\ref{rree}) we get $$(1 + \lambda(K))\cdot
area ({X_{n+2}^{ R}}) \le area(Z_{n}) + \lambda(K) \cdot area(X_{n})
+ \eta. $$  From (\ref{m-z}) and the above inequality we have
\begin{equation}\label{finale}
area ({X_{n+2}^{ R}})  \le \frac{\lambda(K)}{1 + \lambda(K)} \cdot
area(X_{n}) + \frac{C}{1 + \lambda(K)} \cdot \epsilon^n +
\frac{\eta}{1 + \lambda(K)}.
\end{equation}
Since $\eta > 0$ and $R > 1$   are arbitrary,  by letting  $R \to
\infty$ and $\eta \to 0$, we get
$$
area ({X_{n+2}})  \le \frac{\lambda(K)}{1 + \lambda(K)} \cdot
area(X_{n}) + \frac{C}{1 + \lambda(K)} \cdot \epsilon^n.
$$
The Main Lemma follows.

\end{proof}

\subsection{Proof of Lemma~\ref{pre-lem}}
For $0< \gamma < \pi/2$, let $S_{\gamma}$ denote the cone spanned at
$1$ and outside the unit disk such that the two angles formed by the
two sides of $S_\gamma$  and the unit circle are both equal to
$\gamma$. For $\wp
> 0$, let $$C_{\wp} = \{z\:|\: 1 < |z| < 1 + \wp \hbox{ and }
g_{\theta}(z) \in \Delta\}. $$  Note that  $g_{\theta}: \Bbb T \to
\Bbb T$ is a homeomorphism and the point $1$ is the only critical
point of $g_{\theta}$ in the unit circle which has local degree $3$.
The proof of the following lemma is easy and we leave it to the
reader.
\begin{lem}\label{obv}
For $\wp > 0$ small, $C_{\wp}$ is like the part of $S_{\pi/3}$ which
is contained in the annulus $\{z\:|\: 1 < |z| < 1 + \wp\}$.
Moreover, for any $\epsilon
> 0$, we have $C_{\wp} \subset S_{\pi/3 + \epsilon}$ provided that
$\wp > 0$ is small enough.
\end{lem}

\begin{figure}
\bigskip
\begin{center}
\centertexdraw { \drawdim cm \linewd 0.02 \move(-1 0)

\move(0 -2) \lcir  r:1.5

\move(1.5 -2) \lvec(1.8 -1.8) \move(1.5 -2) \lvec(1.8 -2.2)
\move(1.5 -2) \larc r: 0.38 sd:330 ed:30

\move(1.2 -2.1) \htext{$1$} \move(1.51 -2)\lvec(1.8 -1.3) \move(1.51
-2)  \lvec(1.8 -2.7)

\move(1.7 -1.3) \htext{$V_n$} \move(1.7 -3.1) \htext{$W_n$}

\move(1.95 -2.2) \htext{$C_{\wp}$}

 }
\end{center}
\vspace{0.2cm} \caption{The cone $C_{\wp}$ spanned at $1$}
\end{figure}

\begin{lem}\label{cone}
There is a $\wp_0 > 0$ such that for all $0< \wp < \wp_0$, the
following property  holds: Let  $z, z' \in \Bbb C \setminus
\overline{\Delta}$ be  two distinct points such that $1 < |z|, |z'|
< 1 + \wp$,
 and  $g_{\theta}(z) = g_{\theta}(z')$.   Then both $z$ and $z'$
belong to $B_{6 \wp}(1)$. In particular, if there is a $\zeta$
satisfying $1< |\zeta|< 1 + \wp$ and $
g_{\theta}(\zeta)=g_{\theta}(z) = g_{\theta}(z') $, then either
$\zeta = z$ or $\zeta = z'$.
\end{lem}
\begin{proof}
 The proof is by contradiction.
Suppose it were not true. Then we would have a sequence $\wp_n \to
0$ and two sequences $z_n, z_n'$ such that $1< |z_n|, |z_n| < 1 +
\wp_n$, $z_n \ne z_n'$, $g_{\theta}(z_n) = g_{\theta}(z_n')$, and
$|z_n' - 1| \ge 6 \wp_n$. By taking a subsequence, we may assume
that both $z_n$ and $z_n'$ converge. Since $g_{\theta}$ is a
homeomorphism when restricted on $\Bbb T$ and  is a local
homeomorphism at every point in $\Bbb T$ except the critical point
$1$,    the two sequences must both converge to $1$.

Let $V_n$ be the vector starting from $1$ and pointing to $z_n$, and
$W_n$  the vector starting from $1$ and pointing to $z_n'$. Since
$1$ is a double critical point, $g_{\theta}$ behaves like the cubic
map $z \mapsto g_{\theta}(1) + \lambda \cdot (z - 1)^3$ where
$\lambda  = \frac{1}{6} g_{\theta}^{'''}(1) \ne 0$.  So the angle
spanned  by $V_n$ and $W_n$ at $1$ is approximately equal to $2
\pi/3$.  Since both $V_n$ and $W_n$ are outside the unit disk, it
follows that  one of them is above $C_\wp$ and the other one is
below $C_\wp$. Let $\alpha_n$ and $\beta_n$ denote the angle formed
by $V_n$ and $\Bbb T$, $W_n$ and $\Bbb T$, respectively.  Then
$$
\alpha_n + \beta_n = \pi/3 + o(1)  \hbox{ and } |W_n| = (1 +
o(1))|V_n|.
$$ See Figure 3 for an illustration.
Note that $$\frac{|z_n| -1}{|V_n|}= \sin(\alpha_n)  (1 + o(1))
\hbox{ and  } \frac{|z_n'| -1}{|W_n|}= \sin(\beta_n)  (1 + o(1)).$$
Since $\alpha_n + \beta_n = \pi/3 + o(1)$, either $\alpha_n$ or
$\beta_n$ is greater than $\pi/8$ when $\wp_n$ is small enough.
Without loss of generality, let us assume that $|\beta_n| >  \pi/8$.
So $\sin(\beta_n) > (\pi/8) \cdot (2/\pi) = 1/4$. By assumption we
have $|W_n| = |z_n' -1| > 6 \wp_n$. Then
$$|z_n'| - 1 = |W_n| \sin(\beta_n) (1 + o(1)) \ge  6 \wp_n \cdot
\frac{1}{4} \cdot (1 + o(1))= \frac{3}{2} \wp_n (1 + o(1)).$$ This
is a contradiction with $1 < |z_n'| < 1 + \wp_n$. This implies that
both $z_n$ and $z_n'$ belong to $B_{6 \wp}(1)$ provided that $\wp >
0$ is small enough.

Repeating the same argument as above to $z$ and $\zeta$, we get
$\zeta \in B_{6 \wp}(1)$.  Let $\omega = g_{\theta}(z) =
g_{\theta}(z')$.  Then $\omega$ is near $v = g_{\theta}(1)$ if $\wp
> 0$ is small. Since the point $1$ is a double critical point of $g_{\theta}$,
there are exactly three pre-images of $\omega$ in  a  small
neighborhood of $1$: two of them are outside the unit disk, which
are $z$ and $z'$, the third one is inside the unit disk. Since
$\zeta$ is outside the unit disk and belong to $B_{6 \wp}(1)$,
$\zeta$ must be one of $z$ and $z'$. This proves Lemma~\ref{cone}.
\end{proof}

Let $v = g_\theta(1)$ and \begin{equation}\label{defdom} W =
\{\omega\:|\: |\omega - v| < 1/2 \hbox{ and }|\omega| \ge
1\}.\end{equation} Since $g_{\theta}$ has no asymptotic values and
$v$ is the unique critical value of $g_{\theta}$ in $\Bbb C^*$,  one
can define two branches of the inverse of $g_{\theta}$, say
$\Psi_{+}$ and $\Psi_{-}$ in $W$, which map $v$ to $1$ and map $W$
to the two angle domains (both the angles are equal to $\pi/3$)
spanned at $1$ and outside the unit disk. Let $\wp_0 > 0$ be the
constant guaranteed by Lemma~\ref{cone}.

Now let us  take a small $$0< \wp < \wp_0$$ such that for any $\zeta
\in B_{6 \wp}(1)$ with $\omega = |g_{\theta}(\zeta)| > 1$, we have
$\omega \in W$. Thus we have either $\zeta = \Psi_{+}(\omega)$ or
$\zeta = \Psi_{-}(\omega)$.

\begin{lem}\label{after-K} For any $d \ge 1$, $m > 0$ and $K > 1$,
there exists an $L(d, m, K) > 1$ depending only on $d$, $m$ and $K$
such that the following holds.  Suppose $\widetilde{U}$ and
$\widetilde{A}$ are two Jordan domains and $f: \widetilde{U} \to
\widetilde{A}$  is a holomorphic and properly branched covering map
of degree $d \ge 1$. Let $B \subset A \Subset \widetilde{A}$  be a
pair of Jordan domains such that $A$ contains no critical values of
$f$ and ${\rm mod}(\widetilde{A} \setminus \overline{A}) \ge m$ for
some $m > 0$. Suppose $(A, B)$ has $K$-bounded geometry. Let $V
\subset U$$(\Subset \widetilde{U})$ be a pair of Jordan domains
which are the pull backs of $B \subset A$, respectively.   Then $(U,
V)$ has $L(d, m, K)$-bounded geometry.
\end{lem}
\begin{proof}
 First note the following fact (cf. Page 959 of \cite{Sh}):  The space $$\{h: \Delta \to \Delta \:|\: h
\hbox{  is holomorphic and proper of degree } d \hbox{  and  } h(0)
= 0\}$$ is a compact set in the topology of uniform convergence on
compact sets. The lemma  then follows from this fact  and Koebe's
distortion theorem. We leave the (easy) details to the reader.
\end{proof}

\begin{lem}\label{cdr} Let  $1 < M < \infty$.
Then there exists an  $1 < L < \infty$ depending only on $M$   such
that for all  $\zeta \in X_{n+2}$ $(\subset \Bbb C \setminus
\overline{\Delta})$ with ${\rm dist}(\zeta, \Bbb T) \ge \wp$  and
all $n$ large enough, if there exist a Jordan domain $A$ and a
Euclidean disk $B$ satisfying
\begin{itemize}
\item[(1)]  $\omega = g_{\theta}(\zeta) \in A$  and  $v \notin A$,
\item[(2)] $B \subset (Y_{n} \setminus Y_{n+2})$,
\item[(3)] $B \subset A$ and ${\rm diam}(A) \le M \cdot  {\rm diam}(B)$,
\end{itemize}
then  $\zeta$ is associated to some $L$-admissible pair $(U, V)$.
\end{lem}
\begin{proof}
Assume that $n$ is large enough. By the condition (2) $B$ is a small
disk near $\Bbb T$. By  the condition (3) $A$ is a small domain near
$\Bbb T$.
 Let $\widetilde{A} =
B_{1/2}(\omega)$.  Since $\omega \in A$ and $A$ is a small Jordan
domain near $\Bbb T$, $\widetilde{A}$ does not contain $0$ and
$\infty$. Since $v$, $0$ and $\infty$ are the only critical values
of $g_\theta$ (cf. Lemma~\ref{rotation number}), $\widetilde{A}$
contains at most one critical value of $g_{\theta}$, that is, the
point $v$.  Since $g_{\theta}$ has no asymptotic values (cf.
Lemma~\ref{rotation number}), there is a Jordan domain
$\widetilde{U}$ containing $\zeta$ such that $g_{\theta}:
\widetilde{U} \to \widetilde{A}$ is either a holomorphic isomorphism
or a branched and proper covering map. In the later case,
$\widetilde{A}$ contains a single critical value and thus
$\widetilde{U}$ contains a single critical point by Riemann-Hurwitz
formula.  Note that except the critical point $1$ which has local
degree three, all other critical points of $g_{\theta}$ have local
degree two. Thus the degree of the map $g_{\theta}: \widetilde{U}
\to \widetilde{A}$ is either one, two, or three.

Let $V \subset U $ ($\Subset \widetilde{U}$) be the pull backs of $B
\subset A$ such that $\zeta \in U$.  Since $v \notin A$ and
$g_{\theta}$ has no asymptotic values, both $U$ and $V$ must be
Jordan domains. By the conditions (2) and (3) we see the Jordan
domain $A$ can be arbitrarily small provided that $n$ is large
enough. Since $\omega \in A$ by the condition (1), it follows that
$$
{\rm mod}(\widetilde{A} \setminus \overline{A}) = {\rm
mod}(B_{1/2}(\omega) \setminus \overline{A})
$$ can be arbitrarily large provided that $n$ is large enough.
Since $(A, B)$ has $M$-bounded geometry by the condition (3), and
the degree of the covering map $g_{\theta}: \widetilde{U} \to
\widetilde{A}$ is either one, two, or three, by Lemma~\ref{after-K},
the pair $(U, V)$ has $L$-bounded geometry where $L > 1$ is some
constant depending only on $M$.

Since the degree of the map $g_{\theta}: \widetilde{U} \to
\widetilde{A}$ is either one,  two, or three, it follows that
$${\rm mod}(\widetilde{U}\setminus \overline{U}) \ge \frac{1}{3}
\cdot  {\rm mod}(\widetilde{A} \setminus \overline{A})$$ can be
arbitrarily large provided that $n$ is large enough. Note that $0$
and $\infty$ are essential singularities of $g_{\theta}$. Thus $\{0,
\infty\} \cap \widetilde{U} = \emptyset$. This means that the
annulus $\widetilde{U}\setminus \overline{U}$ separates $U$ and
$\{0, \infty\}$. Since ${\rm mod}(\widetilde{U}\setminus
\overline{U})$ can be arbitrarily large provided that $n$ is large
enough,  we have

$\bold{Fact}$: The ratio   $$\frac{{\rm diam}(U) }{{\rm dist}(U,
0)}$$ can be arbitrarily small provided that $n$ is large enough.

Now  we have  two  cases.  (1)  ${\rm dist}(U, 0) \ge 2$ and (2)
${\rm dist}(U, 0) < 2$.

 In the first case, we have
$${\rm dist}(U, \Bbb T) \ge  \frac{1}{2} \cdot {\rm dist}(U, 0).$$
From the above $\bold{Fact}$,  it follows that $${\rm diam}(U) <
{\rm dist}(U, \Bbb T)
$$ provided that $n$ is large enough.

In the second case,   we have ${\rm dist}(U, 0) < 2$. From the above
$\bold{Fact}$, it follows that ${\rm diam}(U)$ can be arbitrarily
small provided that $n$ is large enough. Since $\zeta \in U$ and
${\rm dist}(\zeta, \Bbb T) \ge \wp$, we have ${\rm dist}(U, \Bbb T)
> \wp/2$ provided that $n$ is large enough.   This implies that
$$
{\rm diam}(U) < {\rm dist}(U, \Bbb T)
$$  provided that $n$ is large enough.

So in both the cases, we have ${\rm diam}(U) < {\rm dist}(U, \Bbb
T)$. In particular, $U \cap \Bbb T = \emptyset$.  Since $\zeta \in
X_{n+2} \subset \Bbb C \setminus \overline{\Delta}$ and $\zeta \in
U$, we have $U \subset \Bbb C \setminus \overline{\Delta}$. Since $V
\subset U$, we have $V \subset \Bbb C \setminus \overline{\Delta}$.
Since $V$ is the pull back of $B$ by $g_{\theta}$  and  $B \subset
Y_{n} \setminus Y_{n+2}$, we have $V \subset X_n \setminus X_{n+2}$.
Since $(U, V)$ has $L$-bounded geometry  and ${\rm diam}(U) < {\rm
dist}(U, \Bbb T)$, it follows that the pair $(U,V)$ is an
$L$-admissible pair to which $\zeta$ is associated. The proof of
Lemma~\ref{cdr} is completed.
\end{proof}

For an interval $I \subset \Bbb T$, let ${\rm int}(I)$ denote the
interior of $I$.
\begin{lem}\label{crt} Let $n \ge 0$. Then we have
\begin{itemize}
 \item[1.]  for $-1 \le i \le   q_{n+1}-2$,
$v\notin {\rm int}(I_{n}^i)$ and \item[2.]  for $-1 \le i \le
q_{n}-2$, $v \notin {\rm int}(I_{n+1}^i)$.
\end{itemize}
\end{lem}
\begin{proof}   Note that the collection of the intervals
\begin{equation}\label{refl}
I_n^i, -1 \le i \le q_{n+1}-2, \hbox{  and  } I_{n+1}^i, -1 \le i
\le q_{n} -2, \end{equation} is the $g_{\theta}$-images of the
intervals of the dynamical partition  of level $n$ (see $\S4$ for
the definition of the dynamical partition). Thus all these intervals
have disjoint interiors. Lemma~\ref{crt} then follows since  the
critical point $1$ is the common boundary point of $I_{n}^0$ and
$I_{n+1}^0$, and thus $v = g_{\theta}(1)$ is the common boundary
point of $I_n^{-1}$ and $I_{n+1}^{-1}$.
\end{proof}

Let $I$ be one of the intervals in the collection (\ref{refl}). Let
$J \subset \Bbb T$ be the interval such that $g_{\theta}(J) = I$.
Then $J$ is an interval in the dynamical partition of level $n$. So
the interior of $J$ does not contain the critical point $1$ of
$g_{\theta}$. Let $\Psi$ be the branch of the inverse of
$g_{\theta}$ which maps $I$ to $J$. For $d
> 0$ recall that $\Omega_d(I)$ denotes the hyperbolic neighborhood
of $I$ in $\Omega_I$ (cf (\ref{dhs}) for the definition of
$\Omega_{d}(I)$). By Lemma~\ref{hn-ps}, $\Omega_{d}(I)$ is a Jordan
domain for all $n$ large enough. By Lemma~\ref{rotation number}, $v$
is the unique critical value of $g_{\theta}$ in $\Bbb C^*$. By
Lemma~\ref{crt} $v \notin int(I)$ and thus $v \notin \Omega_{d}(I)$.
It follows that $\Omega_{d}(I)$ contains no critical value of
$g_{\theta}$.  Since $g_{\theta}$ has no asymptotic values and
$\Omega_d(I)$ is a Jordan domain,  $\Psi$ can be holomorphically
extended to $\Omega_{d}(I)$ for all $n$ large enough.  Recall that
$H_{\alpha}(I)$, where  $d = \log \cot \frac{\alpha}{4}$,  is the
part of $\Omega_{d}(I)$ which belongs to the outside of $\Delta$
(cf.  (\ref{hndd}) for the definition of $H_{\alpha}(I)$).

\begin{lem}\label{sch} There exist an $N_1$ such that for all $n \ge
N_1$ we have $$\Psi(H_{\alpha}(I)) \subset H_{\alpha}(J).$$
\end{lem}
\begin{proof}
Since $v \notin \Omega_I$ and $v$ is the unique critical value of
$g_{\theta}$ in $\Bbb C^*$ ($\supset \Omega_I$), $\Omega_{I}$
contains no critical values of $g_\theta$.  By Lemma~\ref{rotation
number} $g_{\theta}$ has no asymptotic values. Thus $\Psi$ can be
holomorphically continued along any path in $\Omega_{I}$. It is
clear that $\Psi$ takes values in $\Omega_J$. But since $\Omega_{I}$
is not simply connected (because it is punctured at $0$ and
$\infty$), the map
$$\Psi: \Omega_{I} \to \Omega_{J}$$ obtained in
this way may be a multi-valued holomorphic function.  To avoid this
problem, let us consider the holomorphic universal covering map
$\pi: \Delta \to \Omega_{I}$.    Note that  $g_\theta$ has no
asymptotic values and   critical values in $\Omega_I$.  Thus
 $\Psi \circ \pi$ can be holomorphically extended along any path
 in $\Delta$. Since $\Delta$ is simply connected, $\Psi \circ \pi$
 can be
 holomorphically extended to a map
 $\widetilde{\Psi}: \Delta \to \Omega_{J}$.  Since $\Psi$ is locally
 homeomorphic,
 $\widetilde{\Psi}$ is locally homeomorphic.

 Since
 $\Omega_d(I)$ is simply connected and $\pi: \Delta \to \Omega_I$ is a universal holomorphic covering map, there is a simply
 connected domain $U \subset \Delta$ such that $\pi: U \to \Omega_d(I)$ is a holomorphic isomorphism. Let $\Gamma \subset U$ be the curve segment
 such that
 $\pi(\Gamma) = I$ (By symmetry, one can show that $I$ is a geodesic in $\Omega_I$ and thus $\Gamma$ is a geodesic in $\Delta$). Since $\pi$ is a local isometry, it follows that
 $$
 U = \{z\:|\: d_{\Delta}(z, \Gamma) < d\}.
 $$

Since $\widetilde{\Psi}: \Delta \to \Omega_{J}$ is holomorphic and
$\widetilde{\Psi}(\Gamma) = \Psi \circ \pi(\Gamma) = J$, by Schwarz
Contraction Principle, it follows that
$$
\widetilde{\Psi}(U) \subset \Omega_d(J).
$$
Since $\widetilde{\Psi}(U) = \Psi \circ \pi (U) =
\Psi(\Omega_d(I))$, it follows that
$$\Psi(\Omega_d(I)) \subset \Omega_d(J).$$   Note that
$H_{\alpha}(I)$ and $H_{\alpha}(J)$ are respectively the parts of
$\Omega_{\alpha}(I)$ and $\Omega_{\alpha}(J)$ which are contained in
the outside of the unit disk (see (\ref{hndd})). Since $\Psi$ is a
univalent function on $\Omega_d(I)$ and maps $I$ to $J$ with the
orientation preserved, $\Psi$ maps $H_{\alpha}(I)$ into
$H_{\alpha}(J)$. That is,
$$\Psi(H_{\alpha}(I)) \subset H_{\alpha}(J).$$
Lemma~\ref{sch} follows.
\end{proof}

\begin{lem} \label{cm}
Let $1< L < \infty$.   Then there is a $1 < K < \infty$ depending
only on $L$ such that   for any  $z \in X_{n+2}$ and $m \ge 1$, if
 $g_{\theta}^i(z) \in \Bbb C \setminus \overline{\Delta}$ for all $1 \le i \le m$, and   $\zeta =
{g}_{\theta}^{m}(z)$ is associated to some $L$-admissible pair $(U',
V')$, then $z$ is associated to some $K$-admissible pair.
\end{lem}
\begin{proof}
 Assume that $\zeta$ is
associated to some $L$-admissible pair $(U', V')$ for some $1 < L <
\infty$.  By the third condition of Definition~\ref{admiss},  we
have ${\rm diam}(U') < L \cdot {\rm dist}(U', \Bbb T)$.  This
implies that there is a Jordan domain $U''$ with the following
properties:
\begin{itemize} \item[1.] $\overline{U''} \subset \Bbb C \setminus
\overline{\Delta}$,
\item[2.] $\overline{U'} \subset U''$,
\item[3.]
${\rm mod}(U'' \setminus \overline{U'}) \ge M(L)$ where $M(L) > 0$
is some constant depending only on $L$.  \end{itemize}
 Since $v \in \Bbb T$ is the only critical value of
$g_{\theta}$ in $\Bbb C^{*}$ (the other two critical values are $0$
and $\infty$, see Lemma~\ref{rotation number}), the post-critical
set of $g_{\theta}$ is equal to $\Bbb T$. Since $\overline{U''}
\subset \Bbb C \setminus \overline{\Delta}$, $U''$ does not
intersect the post-critical set of $g_{\theta}$. Since $g_{\theta}$
has no asymptotic values, the inverse branch of $g_{\theta}^{m}$,
say $\Psi$, which maps $\zeta$ to $z$, can be holomorphically
extended to a univalent function on $U''$.  Let $$W = \Psi(U''), U =
\Psi(U') \hbox{   and } V = \Psi(V').$$  Then both $U'$ and $V'$ are
Jordan domains. Since $V' \subset X_{n} \setminus X_{n+2}$ by the
first condition of Definition~\ref{admiss}, we have $V \subset X_{n}
\setminus X_{n+2}$ by Remark~\ref{rk1}. Since ${\rm mod}(U''
\setminus \overline{U'}) \ge M(L)>0$ and $(U', V')$ has $L$-bounded
geometry, by Koebe's distortion theorem, $(U, V)$ has $K_1$-bounded
geometry with $1< K_1 < \infty$ depending only on $L$. Since ${\rm
mod}(W \setminus \overline{U}) = {\rm mod}(U'' \setminus
\overline{U'}) \ge M(L)$ and the annulus $W\setminus \overline{U}$
separates $U$ from $\Bbb T$, we have ${\rm diam}(U) < K_2 \cdot {\rm
dist}(U, \Bbb T)$ where $1 < K_2 < \infty$ is a constant depending
only on $L$. Let $K = \max\{K_1, K_2\}$. Then $z$ is associated to
the $K$-admissible pair $(U, V)$. Lemma~\ref{cm} has been proved.
\end{proof}


$\bold{Proof \:\: of\: \:Lemma~\ref{pre-lem}}$. In the following $n$
is assumed to be large enough.

Let $z \in X_{n+2}$. Suppose $z \notin Z_n$. It suffices to prove
that there is a $1 < K <\infty$ independent of $n$ and $z$ such that
$z$ is associated to some $K$-admissible pair $(U, V)$.

First recall that $k_z \ge 1$ is the least positive integer such
that $g_{\theta}^{k_z}(z) \in \Delta$. Since $z \in X_{n+2}$, we
have $$g_{\theta}^{k_z}(z) \in Y_{n+2}.$$  Let us denote
$$
z_l = g_{\theta}^{l}(z),\:\: 0 \le l \le k_z.
$$
Since $z_{0} = z \notin Z_n$, the set
$$
\Pi = \{k \in \Bbb Z\:|\: 0\le k < k_{z} \hbox{  and  }
g_{\theta}^{k}(z) \notin Z_n\}$$ is not empty. It is clear that
$\Pi$ contains at most $k_z$ elements and  is thus  a finite set.
Let
$$
k_0 = \max _{k\in \Pi}\{k\}.
$$  Then $0 \le k_0 \le k_z - 1$.
Set  $$\zeta = z_{k_0} \hbox{   and   }\omega = g_{\theta}(\zeta) =
z_{k_0+1}.$$ Then $\zeta \notin Z_n$, and  moreover,
\begin{equation}\label{two-d}\omega \in Z_n  \hbox{ if } k_0 < k_z -1 \hbox{ and }\omega \in Y_{n+2}
\hbox{  if  }k_0 = k_z -1.\end{equation}

By Lemma~\ref{cm} it suffices to prove that $\zeta$ is associated to
some $L$-admissible pair $(U', V')$ for some uniform $1 < L <
\infty$. The proof is divided into two cases.

\vspace{0.3cm}
 Case I: ${\rm dist}(\zeta, \Bbb T) \ge \wp$.
 \vspace{0.3cm}

By Lemma~\ref{cdr},  it suffices to prove that there exist a uniform
$1 < M < \infty$, a Jordan domain $A$  and  a Eucldiean disk $B$
such that
\begin{itemize}
\item[(1)]  $\omega \in A$ and $v \notin A$,
\item[(2)] $B \subset (Y_{n} \setminus Y_{n+2})$,
\item[(3)] $B \subset A$ and ${\rm diam}(A) \le M \cdot  {\rm
diam}(B)$.
\end{itemize}
There are two subcases.

In the first subcase, $\omega \in Z_n$. Then from  (\ref{ZN}) there
is an interval $I$ in the dynamical partition of level $n$ such that
$\omega \in H_{\alpha}(I)$. By Lemma~\ref{dyn-cell},  $I$ is
contained in an interval $J$  of the cell partition of level $n$,
and moreover, $J$ is either equal to $I$, or is the union of $I$ and
one of its adjacent intervals of the dynamical partition. By
Theorem~\ref{SH} we  have
$$|J|\asymp |I|.$$ Let $E$ be the Yoccoz's cell of level $n$
attached to $J$.  Then ${\rm diam}(E) \asymp |J|$ by Lemma~\ref{good
geometry}.

In the second subcase, $\omega \in Y_{n+2}$. Then there is a Yoccoz
cell of level $n$ containing $\omega$. Let us still use $E$ to
denote the cell.

In both the two subcases, by Lemma~\ref{good geometry}, we have a
Eucldiean disk $B \subset E \setminus Y_{n+2}$ such that
$${\rm diam}(B)  \asymp {\rm diam}(E).$$ It is easy to see that in
the first subcase,   one can construct a Jordan domain $A$
containing $\omega$ such that $v \notin A$ and
$$
B \subset A \subset E \cup H_{\alpha}(I);
$$
and in the second subcase,  one can  construct a Jordan domain $A$
containing $\omega$  such that $v \notin A$ and
$$
B \subset A \subset E.
$$
In the first subcase, since ${\rm diam}(H_{\alpha}(I)) \asymp |I|$ (
cf. Lemma~\ref{hn-ps}) and  $|I| \asymp |J| \asymp {\rm diam}(E)$,
we have ${\rm diam}(A) \le {\rm diam}(E \cup H_{\alpha}(I)£© \preceq
|J| \asymp{\rm diam}(B)$. In the second subcase, ${\rm diam}(A) \le
{\rm diam}(E) \preceq {\rm diam}(B)$.  So in both the two subcases,
${\rm diam}(A) \le M \cdot {\rm diam}(B)$ with $1 < M < \infty$
being some uniform constant.  This implies that the conditions
(1)-(3) hold in both the two subcases.   So  Lemma~\ref{pre-lem} in
Case I follows from Lemma~\ref{cdr}.

\vspace{0.3 cm}  Case II. ${\rm dist}(\zeta, \Bbb T) <  \wp$.
\vspace{0.3cm}

By (\ref{two-d}) we have  $\omega = g_{\theta}(\zeta) \in Z_n$ or
$\omega \in Y_{n+2}$.   If $\omega \in Z_n$,  from  (\ref{ZN}) there
is an interval $I$ in the dynamical partition of level $n$ such that
$\omega \in H_{\alpha}(I)$.  If  $\omega \in Z_n$, we have  two
subcases: Subcase I: $v \in I$  and Subcase II:  $v \notin I$.  The
situation that $\omega \in Y_{n+2}$ is considered in Subcase III.

Subcase I. $\omega \in H_{\alpha}(I_n^{q_{n+1}-1})$.  Note that $v
\in I_{n}^{q_{n+1}-1}$.  Since $1$ is a double critical point of
$g_{\theta}$ and ${\rm dist}(\omega, v) \preceq {\rm
diam}(H_{\alpha}(I_n^{q_{n+1}-1})) \asymp |I_n^{q_{n+1}-1}|$, there
are two distinct points $\zeta'$ and $\zeta''$ outside the unit disk
such that
\begin{itemize}
\item[i.] ${\rm dist}(\zeta', 1) \asymp {\rm dist}(\zeta'',
1)\preceq|I_n^{q_{n+1}-1}|^{1/3}$,
\item [ii.] $g_{\theta}(\zeta') = g_{\theta}(\zeta'') = \omega$.
\end{itemize}
 From (i) we have $1< |\zeta'|, |\zeta''| < 1 + \wp$ for all $n$ large enough.  Since ${\rm dist}(\zeta, \Bbb T) < \wp$, by Lemma~\ref{cone}  we have either  $\zeta= \zeta'$  or  $\zeta =
 \zeta''$.
 Let $\Psi_{+}$ and $\Psi_{-}$ be the two branches of the inverse
of $g_{\theta}$ defined right after Lemma~\ref{cone}. Without loss
of generality, we may assume that
$$
\Psi_{+}([x_{q_{n+1}-1}, v]) = [x_{q_{n+1}}, 1]  \hbox{  and  }
\Psi_{-}([v, x_{q_n + q_{n+1}-1}]) = [1, x_{q_n + q_{n+1}}].
$$
See Figures 4 and 5 for an illustration. Then  we have either $
\zeta=\Psi_{+}(\omega)$ or $\zeta = \Psi_{-}(\omega)$.

\vspace{0.3cm}
\begin{figure}
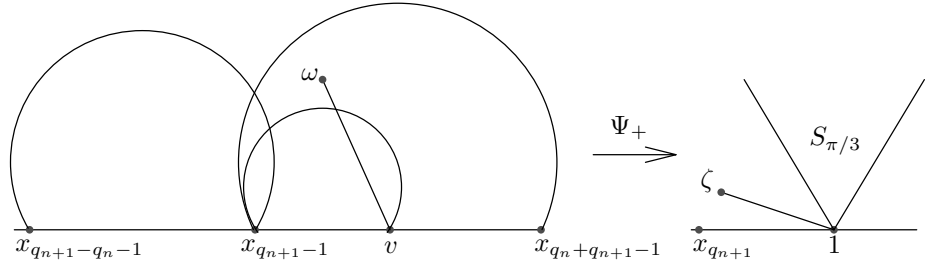

\bigskip
\begin{center}
\centertexdraw { \drawdim cm \linewd 0.02 \move(-2 2)

\move(-10 -2)    \lvec(-3 -2) \move(-9.8 -2) \fcir f:0.3 r:0.05
\move(-10 -2.4) \htext{$x_{q_{n+1}-q_{n} -1}$}

\move(-6.8 -2) \fcir f:0.3 r:0.05 \move(-7 -2.4) \htext{$x_{q_{n+1}
-1}$}

\move(-5 -2) \fcir f:0.3 r:0.05   \lvec(-5.9 0)\move(-5.1 -2.3)
\htext{$v$}

\move(-3 -2) \fcir f:0.3 r:0.05 \move(-3.1 -2.4) \htext{$x_{q_n +
q_{n+1} -1}$}

\move(-1 -2)    \lvec(2 -2)

\move(-0.9 -2) \fcir f:0.3 r:0.05 \move(-1 -2.4)
\htext{$x_{q_{n+1}}$}

\move(0.9 -2) \fcir f:0.3 r:0.05  \lvec(-0.6 -1.5)\move(0.8 -2.3)
\htext{$1$}   \move(0.9 -2) \lvec(-0.3 0) \move(0.9 -2) \lvec(2.1 0)

\move(-4.9 -1.1) \larc r: 2.12 sd:335 ed:205

\move(-8.3 -1.1) \larc r: 1.75 sd:330 ed:210

\move(-5.9 -1.43) \larc r: 1.05 sd:325 ed:215

\move(-5.9 0)  \fcir f:0.3 r:0.05  \move(-6.2 0)  \htext{$\omega$}

\move(-0.6 -1.5)  \fcir f:0.3 r:0.05  \move(-0.9 -1.5)
\htext{$\zeta$}

\move(0.55 -1) \htext{$S_{\pi/3}$}

\move(-2.3  -1) \arrowheadtype t:V \avec(-1.2 -1)

\move(-2.1 -0.8) \htext{$\Psi_{+}$}

 }

\end{center}
\vspace{0.2cm} \caption{The map $\Psi_{+}$}
\end{figure}
First let us suppose $ \zeta= \Psi_{+}(\omega) $.  Note that
$$[x_{q_{n+1}-1}, x_{q_{n} +q_{n+1}-1}] =  I_{n}^{q_{n+1}-1} \hbox{
and  } [x_{q_{n+1} -q_{n} -1}, x_{q_{n+1}-1}] =
I_{n}^{q_{n+1}-q_{n}-1} $$  are two adjacent intervals in the
dynamical partition of level $n$ and are thus commensurable. By
Lemma~\ref{sch} we have
$$
\Psi_{+}(H_{\alpha}(I_{n}^{q_{n+1}-q_{n}-1})) \subset
H_{\alpha}(I_{n}^{q_{n+1}-q_n}) \subset Z_n
$$
and
$$
\Psi_{+}(H_{\alpha}([x_{q_{n+1}-1}, v])) \subset H_{\alpha}(I_{n+1})
\subset Z_n,
$$
These, together with $\zeta \notin Z_n$,  implies that
\begin{equation}\label{gr}
\omega \notin H_{\alpha}(I_{n}^{q_{n+1}-q_{n}-1}) \cup
H_{\alpha}([x_{q_{n+1}-1}, v]).\end{equation}    From (\ref{gr}) and
the geometry of $H_{\alpha}(I_{n}^{q_{n+1}-q_{n}-1})$,
$H_{\alpha}([x_{q_{n+1}-1}, v])$ and $H_{\alpha}(I_{n}^{q_{n+1}-1})$
(cf.  Lemma~\ref{hn-ps} and Figure 4), there is a $0< \beta < \pi$
such that for all $n$ large enough, the angle between  $[v, \omega]$
and $[v, x_{q_{n+1}-1}])$ is greater than $\beta$. Since $\Psi_{+}$
is like a branch of the cubic root map, the angle between $[1,
\zeta]$ and $[1, x_{q_{n+1}}]$ is approximately greater  than
$\beta/3$ and thus strictly greater than $\beta / 4$.

Next let us  suppose $\Psi_{-}(\omega) = \zeta$. By Lemma~\ref{sch}
we have
$$
\Psi_{-}(H_{\alpha}[v, x_{q_{n}-1}])) \subset H_{\alpha}(I_n)
\subset Z_n.
$$
This, together with $\zeta \notin Z_n$,  implies that $\omega \notin
H_{\alpha}([v, x_{q_{n}-1}])$.  Note that   $$|[v, x_{q_{n}-1}]| >
|[v, x_{q_{n}+q_{n+1}-1}]|\succeq |I_{n}^{q_{n+1}-1}|.$$   The first
inequality comes from $[v, x_{q_{n}+q_{n+1}-1}] \subset [v,
x_{q_{n}-1}]$ and  the second approximate inequality is guaranteed
by the first assertion of Lemma~\ref{real bound}.   Since  $\omega
\in H_{\alpha}(I_n^{q_{n+1}-1})$  and $\omega \notin H_{\alpha}([v,
x_{q_{n}-1}])$,   by the geometry of $H_{\alpha}(I_{n}^{q_{n+1}-1})$
and $H_{\alpha}([v, x_{q_{n}-1}])$ (cf.  Lemma~\ref{hn-ps} and
Figure 5), there is a $0 < \gamma < \pi$ such that for all $n$ large
enough, the angle between $[v, \omega]$ and $[v, x_{q_{n} -1}]$ is
greater than $0 < \gamma< \pi$.  Since $\Psi_{-}$ is like a branch
of the cubic root map, the angle between $[1, \zeta]$ and $[1,
x_{q_{n}}]$ is approximately greater than $\gamma/3$ and thus
strictly greater than $\gamma/4$.

 \vspace{0.3cm}
\begin{figure}
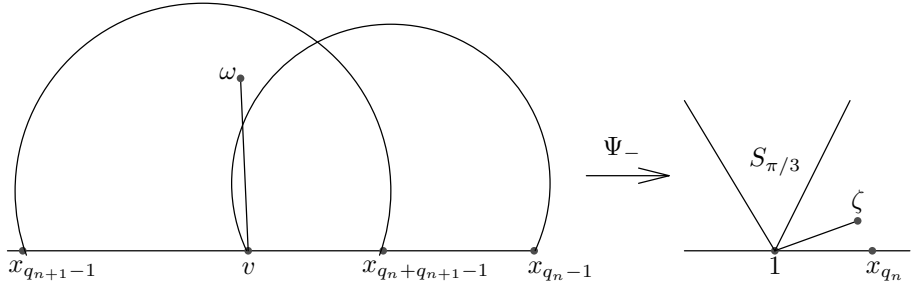

\bigskip
\begin{center}
\centertexdraw { \drawdim cm \linewd 0.02 \move(-2 2)

\move(-10 -2)    \lvec(-3 -2)

\move(-9.8 -2) \fcir f:0.3 r:0.05 \move(-10 -2.4)
\htext{$x_{q_{n+1}-1}$}

\move(-6.8 -2) \fcir f:0.3 r:0.05 \lvec(-6.9 0.3)   \move(-6.9 -2.3)
\htext{$v$}

\move(-5 -2) \fcir f:0.3 r:0.05   \move(-5.3 -2.4) \htext{$x_{q_n +
q_{n+1} -1}$}

\move(-3 -2) \fcir f:0.3 r:0.05 \move(-3.1 -2.4) \htext{$x_{q_n
-1}$}

\move(-1 -2)    \lvec(2 -2)

\move(1.5 -2) \fcir f:0.3 r:0.05 \move(1.4 -2.4) \htext{$x_{q_{n}}$}

\move(0.2 -2) \fcir f:0.3 r:0.05  \lvec(1.3 -1.6)\move(0.1 -2.3)
\htext{$1$}   \move(0.2 -2) \lvec(-1 0) \move(0.2 -2) \lvec(1.2 0)

\move(-4.9 -1.1) \larc r: 2.12 sd:335 ed:205

\move(-7.4 -1.2) \larc r: 2.5 sd:340 ed:200

\move(-6.9 0.3)  \fcir f:0.3 r:0.05  \move(-7.2 0.3)
\htext{$\omega$}

\move(1.3 -1.6)  \fcir f:0.3 r:0.05  \move(1.2 -1.5) \htext{$\zeta$}

\move(-0.15 -1) \htext{$S_{\pi/3}$}

\move(-2.3  -1) \arrowheadtype t:V \avec(-1.2 -1)

\move(-2.1 -0.8) \htext{$\Psi_{-}$}

 }

\end{center}
\vspace{0.2cm} \caption{The map $\Psi_{-}$}
\end{figure}
Let $\delta = \min\{\beta/4, \gamma/4\}$.  Then from the above
arguments,  $\zeta$ always  belongs to a cone spanned at $1$ and
bounded by two rays which form an angle $\delta > 0$ with $\Bbb T$.
Let us denote this cone by $S_{\delta}$. Since $\delta < \pi/3$, we
have  $S_{\pi/3} \subset S_{\delta}$. See Figure 6 for an
illustration.

Note that  $H_{\alpha}(I_n) \cup H_{\alpha}(I_{n+1})$ ($\subset
Z_n$)
 contains the part,
 which is outside the unit disk, of a  disk centered at $1$
 and with radius $\asymp |I_{n+1}|$.  Since $\zeta \notin Z_n$, we have
   ${\rm dist}(1, \zeta) \succeq
 |I_{n+1}|$. Thus
\begin{equation}\label{nngg}{\rm dist}(\omega,
 v) \asymp {\rm dist}(1, \zeta)^{3}\succeq |I_{n+1}|^{3} \asymp |[v,
 x_{q_{n+1}-1}]| \asymp |I_n^{q_{n+1}-1}|.\end{equation}  The last $\asymp$ comes from the
 first  assertion of Lemma~\ref{real bound}.   On the other hand,
 we have ${\rm dist}(\omega, v) \le
 {\rm diam}(H_{\alpha}(I_n^{q_{n+1}-1})) \asymp |I_n^{q_{n+1}-1}|$.
This, together with (\ref{nngg}), implies
\begin{equation}\label{pds}
{\rm dist}(\omega, v)\asymp |I_n^{q_{n+1}-1}|.
\end{equation}

Now let $I = I_{n}^{q_{n+1}-1} \cup I_{n+1}^{q_{n}-1}$. Then $I$ is
the interval in the cell partition of level $n$ which contains
$I_n^{q_{n+1}-1}$ (cf. Lemma~\ref{dyn-cell}). Since
$|I_{n}^{q_{n+1}-1} |\asymp|I_{n+1}^{q_{n}-1}|$, we have $|I| \asymp
|I_{n}^{q_{n+1}-1}|$.  Let $E$ be the cell of level $n$ which is
attached to $I$. By Lemma~\ref{good geometry}, we have ${\rm
diam}(E) \asymp |I|$, and moreover, there is a Euclidean disk $B$
contained in $E \setminus Y_{n+2}$ such that
\begin{equation}\label{ce}
{\rm dist}(B, \Bbb T) \asymp {\rm diam}(B)\asymp |I|\asymp
|I_{n}^{q_{n+1}-1}|.\end{equation}   Since $v \in I_{n}^{q_{n+1} -1}
\subset I$,  we have ${\rm dist}( B, v) \le {\rm diam}(E) \preceq
|I|$.  Since  ${\rm dist}(B, v) \ge {\rm dist}(B, \Bbb T)\asymp |I|$
and ${\rm diam}(B) \asymp |I|$ by (\ref{ce}), we have
\begin{equation} \label{de}{\rm dist}(B, v) \asymp |I| \asymp{\rm
diam}(B).\end{equation}  Let $\Psi$ be the branch of the inverse of
$g_{\theta}$ which maps $E$ into $S_{\delta}$ (cf. Lemma~\ref{obv}).
Since $1$ is a double critical point of $g_{\theta}$ and $B$  is a
Euclidean disk, from (\ref{de}) and  Koebe's distortion theorem, it
follows that there is a Euclidean disk $V \subset \Psi(B) \subset
S_{\delta}$ such that
\begin{equation}\label{red}
 {\rm dist}(V, 1) \asymp {\rm diam}(V) \asymp {\rm diam}(B)^{1/3}
\asymp |I_{n}^{q_{n+1}-1}|^{1/3} \asymp |I_{n}|.
\end{equation}

From (\ref{pds}) we have  \begin{equation}\label{aha} {\rm
dist}(\zeta, 1) \asymp |I_{n}^{q_{n+1}-1}|^{1/3}\asymp
|I_n|.\end{equation}  Since $V \subset S_{\delta}$  and  $\zeta \in
S_{\delta}$ (see Figure 6 for an illustration),  from (\ref{red})
and (\ref{aha}) there is a Jordan domain $U$ contained in
$S_{\delta}$ such that
\begin{itemize}
\item[1.]  $\zeta \in U$, \item[2.]  $V \subset U$,
\item[3.] ${\rm diam}(U) \asymp {\rm dist}(U, \Bbb T)\asymp {\rm diam}(V)$. \end{itemize}
Since $B \subset E \setminus Y_{n+2} \subset Y_{n} \setminus
Y_{n+2}$, we have $V \subset \Psi(B) \subset X_n \setminus X_{n+2}$.
Thus $(U, V)$ is a desired pair of Jordan domains. This proves
Lemma~\ref{pre-lem} in Subcase I of Case II.

\vspace{0.3cm}
\begin{figure}
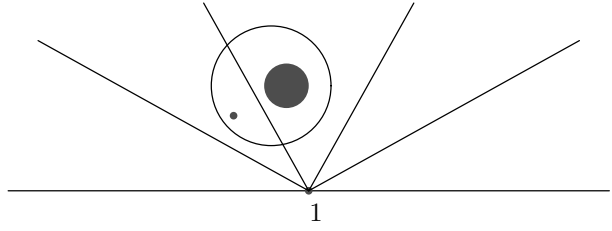

\bigskip
\begin{center}
\centertexdraw { \drawdim cm \linewd 0.02 \move(-2 1)

\move(-4 -2)    \lvec(4 -2)

\move(0 -2)  \fcir f:0.3 r:0.05  \lvec(-1.4 0.5)\move(0 -2.4)
\htext{$1$} \move(0 -2) \lvec(1.4 0.5) \move(0 -2) \lvec(-3.6 0)
\move(0 -2) \lvec(3.6 0)

\move(-0.3 -0.6) \fcir f:0.3 r:0.3 \move(-0.5 -0.6) \lcir r:0.8
\move(-1 -1) \fcir f:0.3 r:0.05 }
\end{center}
\vspace{0.2cm} \caption{The pair $(U, V)$ associated to $\zeta$}
\end{figure}
 Subcase II. $\omega \in H_{\alpha}(J)$  where $J = I_n^i$ for some  $0
\le i < q_{n+1} -1$ or $J = I_{n+1}^{i}$ for some $0 \le i \le q_{n}
-1$.  Then $v \notin J$.  Let $J' \subset \Bbb T$ such that
$g_{\theta}(J') = J$.  Then $1 \notin J'$. Let $\Psi$ denote the
branch of the inverse of $g_{\theta}$ which maps $J$ to $J'$.  By
Lemma~\ref{sch} we have
\begin{equation}\label{spz}
\Psi(H_{\alpha}(J)) \subset H_{\alpha}(J').
\end{equation}
Let us now prove  \begin{equation}\label{lab} H_{\alpha}(J') \subset
Z_n.\end{equation} In the case that  $J = I_n^{i}$ for $0 \le i <
q_{n+1} -1$ or $J = I_{n+1}^{i}$ for $0 \le i < q_{n}-1$,  $J' =
I_n^{i+1}$ or $J' = I_{n+1}^{i+1}$ is still an interval in the
dynamical partition of level $n$, and thus (\ref{lab}) holds by the
definition of $Z_n$ (cf. (\ref{ZN})). In the case that $J =
I_{n+1}^{q_{n}-1}$, $J' = I_{n+1}^{q_{n}} = [x_{q_{n}},
x_{q_{n}+q_{n+1}}] \subset [1, x_{q_{n}}] = I_n$. Thus
$H_{\alpha}(J') \subset H_{\alpha}(I_n) \subset Z_n$ and (\ref{lab})
also holds. This proves (\ref{lab}).

Let $\zeta' = \Psi(\omega)$. Since $\omega \in H_{\alpha}(J)$, from
(\ref{spz}) we have $$\zeta' \in H_{\alpha}(J').$$ See Figure 7 for
an illustration.   From (\ref{lab}) we have $\zeta' \in Z_n$. Thus
$\zeta \ne \zeta'$.   Since $g_{\theta}(\zeta) = g_{\theta}(\zeta')
= \omega$ and $1 < |\zeta|, |\zeta'| < 1 + \wp$, by Lemma~\ref{cone}
both $\zeta$ and $\zeta'$ belong to $B_{6\wp}(1)$.

 Now we claim
\begin{equation}\label{pcm}
{\rm dist}(v,  \omega) \succeq |J|.
\end{equation} Let us prove the claim.  Since $\omega \in H_{\alpha}(J)$, we need only to
prove that  \begin{equation}\label{can} {\rm dist}(v, H_{\alpha}(J))
\succeq |J|.\end{equation} By the geometry of $H_{\alpha}(J)$ (cf.
Lemma~\ref{hn-ps}), we have ${\rm dist}(v, H_{\alpha}(J)) \asymp
{\rm dist}(v, J)$. Since $v \notin J$, by the second  assertion of
Lemma~\ref{real bound}, we have ${\rm dist}(v, J) \succeq |J|$. This
proves (\ref{can}) and (\ref{pcm}) then follows.

Now let $I$ be the interval in the cell partition of level $n$ which
contains $J$. By Lemma~\ref{dyn-cell}, $I$ either is $J$ itself or
is the union of $J$ and one of its adjacent intervals in the
dynamical partition.  The situation illustrated in Figure 7 is the
case that $I = J$.  In both the cases, by Theorem~\ref{SH} we have
$$|I| \asymp |J|.$$ Let $E$ be the cell of level $n$ which is
attached to $I$. By Lemma~\ref{good geometry}, there is a Euclidean
disk $B \subset E \setminus Y_{n+2}$ such that
$$
{\rm diam}(B) \asymp {\rm dist}(\Bbb T, B) \asymp |I|.
$$
In particular,  we have \begin{equation}\label{prs} {\rm dist}(v, B)
\succeq {\rm dist}(\Bbb T, B) \asymp |I|. \end{equation}   We thus
have
\begin{equation}\label{nnss}
{\rm dist}(v, B) \succeq {\rm diam}(B) \asymp |I|\asymp|J|.
\end{equation}
Note that  ${\rm diam}(H_{\alpha}(J)) \asymp |J|$ and ${\rm diam}(E)
\asymp |I|$.  This, together  with (\ref{pcm}) and (\ref{nnss}),
implies that there exists a Jordan domain $A\subset H_{\alpha}(J)
\cup E$ such that
\begin{itemize}
\item[1.] $\omega \in A$,
\item[2.] $B \subset A$,
\item[3.] ${\rm dist}(v, A) \succeq {\rm diam}(A)\asymp {\rm diam}(B)$.
\end{itemize}
\begin{figure}
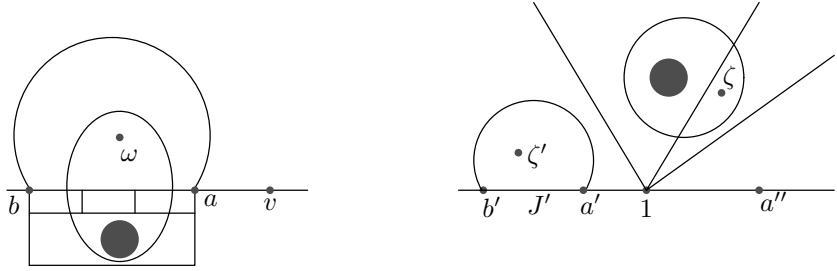

\bigskip
\begin{center}
\centertexdraw { \drawdim cm \linewd 0.02 \move(-2 2)

\move(-10 -1.5)    \lvec(-6 -1.5)  \move(-6.5 -1.5)  \fcir f:0.3
r:0.05  \move(-6.6 -1.75)  \htext{$v$}

\move(-8.6 -0.77) \larc r: 1.3 sd:325 ed:215

\move(-9.7 -2.5)    \lvec(-7.5 -2.5) \lvec(-7.5 -1.5) \move(-9.7
-2.5) \lvec(-9.7 -1.5)  \move(-9.7 -1.8) \lvec(-7.5 -1.8)

\move(-9 -1.8) \lvec(-9 -1.5) \move(-8.3 -1.8) \lvec(-8.3 -1.5)

\move(-8.5 -2.15) \fcir f:0.3 r:0.25

\move(-7.5 -1.5)  \fcir f:0.3 r:0.05    \move(-7.4 -1.7)\htext{$a$}

\move(-10 -1.8) \htext{$b$}

 \move(-9.7 -1.5) \fcir f:0.3 r:0.05

 \move(-3.69 -1.85) \htext{$b'$}

 \move(-3.67 -1.5) \fcir f:0.3 r:0.05

\move(-8.5 -1.45) \lellip rx:0.7 ry:1

\move(-4 -1.5)    \lvec(1 -1.5)

\move(-8.5 -0.8) \fcir f:0.3 r:0.05  \move(-8.5 -1.1)
\htext{$\omega$}

\move(-3 -1.1) \larc r: 0.8 sd:330 ed:210

\move(-1.5 -1.5)  \fcir f:0.3 r:0.05

\lvec (-3  1)  \move(0  1) \lvec(-1.5 -1.5)

\lvec(1 0.3)  \move(-1.6 -1.85) \htext{$1$}

\move(-1.2 -0) \fcir f:0.3 r:0.25

\move(-3.2 -1)  \fcir f:0.3 r:0.05  \move(-3.1 -1.2)
\htext{$\zeta'$} \move(-3.1 -1.8) \htext{$J'$}

\move(-2.4 -1.8) \htext{$a'$}

\move(-3.1 -1.2)
 \move(-2.34 -1.5) \fcir f:0.3 r:0.05

\move(-0.5 -0.2)  \fcir f:0.3 r:0.05  \move(-0.5 -0.15)
\htext{$\zeta$}

\move(-1 0) \lcir r:0.8

\move(0 -1.77) \htext{$a''$}

\move(0 -1.5) \fcir f:0.3 r:0.05

 }
\end{center}
\vspace{0.2cm} \caption{Subcase II of Case II }
\end{figure}
Since $\zeta$ and $\zeta'$ belong to $B_{6\wp}(1)$ and $g_{\theta}$
is like a cubic map near $1$, we have ${\rm dist}(\omega, v) =
O(\wp^3)$.  Let $a$ and $b$ denote the two end points of $J$ and $a$
be the one which is nearer to $v$. Since ${\rm
diam}(H_{\alpha}(J))\asymp |J|$ and $J$ can be arbitrarily small
provided that $n$ is large enough, we have
\begin{equation}\label{sups} {\rm dist}(a, v) =
O(\wp^3).\end{equation}   Since $\wp > 0$ is small, from
(\ref{sups}) the small arc $[a, v]$ can be regarded as a straight
segment.
 Let $a', b'  \in \Bbb T$ be the points such that $g_{\theta}(a') = a$ and $g_{\theta}(b') = b$.
 Then $a'$ is close to $1$. Let $a'' \in \Bbb T$ be a point near $1$ which is on
the other side of $1$.  See Figure 7 for an illustration.

In Figure 7  consider the triangle with $\omega$, $a$ and $v$ being
the three vertices.
 By the geometry of $H_{\alpha}(J)$ (cf. Lemma~\ref{hn-ps}),  it
follows that the angle formed by $[a, v]$ and $[a, \omega]$ is
 greater than $\alpha$.  Thus the angle formed by $[v,
\omega]$ and $[v, a]$ is  less than $\pi - \alpha$.   Note that
${\rm diam}(H_{\alpha}(J)) \asymp |J|$  and by (\ref{sups}) the
points in $H_{\alpha}(J) \cup E$ can be arbitrarily close to $v$
provided that $\wp > 0$ is small enough and $n$ is large enough ($J$
can be arbitrarily small provided that $n$ is large enough). By
assuming $\wp
> 0$ is small enough and $n$ is large enough, we may regard $\Psi$ as  a cubic root map near $v$.
Since  $\Psi$  maps $J$ to $J'$ and $\Psi(\omega) = \zeta'$, it
follows that the angle formed by $[1, \zeta']$ and $[1, a']$ is
approximately less than $(\pi - \alpha)/3$. Since the angle formed
by $[1, \zeta]$ and $[1, \zeta']$ is approximately equal to $2
\pi/3$, the angle formed by $[1, \zeta]$ and $[1, a'']$ is
approximately greater than $ \alpha/3$.

Now let $\Psi_{+}$ and $\Psi_{-}$ be the two branches of the inverse
of $g_{\theta}$ defined right after the proof of Lemma~\ref{cone}.
Without loss of generality, let us assume that $\Psi_{-} = \Psi$.
That is, $\Psi_{-}(\omega) = \zeta'$. Then $\Psi_{+}$ is the branch
such that $\Psi_{+}(\omega) = \zeta$.  Let $S_{\alpha/4}$ denote the
cone spanned at $1$ and outside the unit disk such that the two
exterior angles formed by the two sides of $S_{\alpha/4}$ and $\Bbb
T$ are both equal to $\alpha/4$.    In the last paragraph, by
replacing $\omega$ with an arbitrary point $\xi \in H_{\alpha}(J)$,
it follows  that the angle formed by $[1, \Psi_{+}(\xi)]$ and $[1,
a'']$ is approximately greater than $\alpha/3$ and thus strictly
greater than $\alpha/4$.  This implies
\begin{equation}\label{refee1}
\Psi_{+}(H_{\alpha}(J)) \subset S_{\alpha/4}.\end{equation}

Note that $\Psi_{-} = \Psi$ maps $J$ to $J'$ and thus maps $E$ to a
domain contained in $\Delta$ and attached to $\Bbb T$.  Thus
$\Psi_{+}$ maps $E$ into the outside of the unit disk. From
(\ref{sups}) it follows that $\Psi_{+}(E)$ is contained in an
arbitrarily small neighborhood of $1$ provided that $\wp> 0$ is
small enough and $n$ is large enough. By Lemma~\ref{obv} and by
assuming $\wp > 0$ is small enough and $n$ is large enough, we have
\begin{equation}\label{refee2}
\Psi_{+}(E) \subset  S_{\alpha/4}.\end{equation} Since  $A \subset
H_{\alpha}(J) \cup E$, from (\ref{refee1}) and (\ref{refee2}) we
have
$$\Psi_{+}(A) \subset S_{\alpha/4}.$$   This implies that
\begin{equation}\label{ioo}
{\rm dist}(\Psi_{+}(A), 1) \asymp  {\rm dist}(\Psi_{+}(A), \Bbb T).
\end{equation} From the property (3) above we have  ${\rm dist}(v, A) \succeq
{\rm diam}(A)$. Because $\Psi_{+}$ is like a branch of the cubic
root map near $v$, we have \begin{equation}\label{poo}
 {\rm
dist}(\Psi_{+}(A), 1) \succeq {\rm diam}(\Psi_{+}(A)).
\end{equation} From (\ref{ioo}) and (\ref{poo}) we have $$ {\rm dist}(\Psi_{+}(A), \Bbb T) \succeq {\rm diam}(\Psi_{+}(A)).$$ Since ${\rm diam}(A) \asymp {\rm diam}(B)$, the pair
$(A, B)$ has $M$-bounded geometry for some universal $M
> 1$. From ${\rm dist}(v, A) \succeq {\rm diam}(A)$  and Koebe's
distortion theorem, the distortion of $\Psi_{+}$ is universally
bounded in $A$.  Thus  $(\Psi_{+}(A), \Psi_{+}(B))$ has $K$-bounded
geometry with $K> 1$ being  some universal constant.  Since $B
\subset E \setminus Y_{n+2} \subset Y_n \setminus Y_{n+2}$, we have
$\Psi_{+}(B) \subset X_{n} \setminus X_{n+2}$. Let $U = \Psi_{+}(A)$
and $V = \Psi_{+}(B)$. Then the pair $(U, V)$ is the desired pair.
This proves  Lemma~\ref{pre-lem} in Subcase II of Case II.

 Subcase III. $\omega \in Y_{n+2}$.
 Let $E$ be the cell of level $n$ containing $\omega$. Let $\Psi$ be
 the branch of the inverse of $g_{\theta}$ which maps $\omega$ to $\zeta$.
  Then  $\zeta \in C_{\wp} \subset S_{\pi/4}$ provided that $\wp > 0$ is small enough (cf. Lemma~\ref{obv}).   It follows that $\Psi(E) \subset
 S_{\pi/4}$ for all $n$ large enough.

Since $0< \alpha < \pi/3$, by the geometry of $H_{\alpha}(I)$ (cf.
Lemma~\ref{hn-ps}),  $H_{\alpha}(I_n) \cup H_{\alpha}(I_{n+1})
(\subset Z_n)$ contains the part, which is outside the unit disk, of
a disk centered at $1$ and with radius $\asymp|I_n|$.  Since $\zeta
\notin Z_n$,  we have ${\rm dist}(1, \zeta) \succeq |I_n|\asymp
|I_{n+1}|$. Thus
\begin{equation}\label{cr}{\rm dist}(v, \omega) \succeq |g_\theta
(I_{n+1})| = |[v, x_{q_{n+1}-1}]| \asymp |[x_{q_{n+1}-1},
x_{q_{n}+q_{n+1}-1}]| = |I_{n}^{q_{n+1}-1}|\end{equation} where
$\asymp$ comes from the first assertion of Lemma~\ref{real bound}.
Let $J = \overline{E} \cap \Bbb T$ be the arc interval to which $E$
is attached. Let us prove that
\begin{equation}\label{db} {\rm dist}(v, \omega) \succeq
|J|.\end{equation} In the case that $J$ contains
$I_{n}^{q_{n+1}-1}$, $J$ is the union of $I_{n}^{q_{n+1}-1}$ and one
of its adjacent intervals in the dynamical partition of level $n$
(cf. Lemma~\ref{dyn-cell}). That is,
$$J = I_{n+1}^{q_n -1} \cup I_{n}^{q_{n+1}-1} = [x_{q_{n}-1},
x_{q_{n+1} +q_{n} -1}] \cup [x_{q_{n+1} + q_{n} -1},
x_{q_{n+1}-1}].$$  Since $|I_{n+1}^{q_n -1}| \asymp
|I_{n}^{q_{n+1}-1}|$ by Theorem~\ref{SH}, we get  $|J| \asymp
|I_{n}^{q_{n+1}-1}|$. Then (\ref{db}) follows from (\ref{cr}). In
the other cases,  since $v \in I_{n}^{q_{n+1} -1}$, we have $v
\notin J$. Then (\ref{db}) follows from the second assertion of
Lemma~\ref{real bound}. See Figure 8 for an illustration for the
case that $v \notin J$.
\begin{figure}
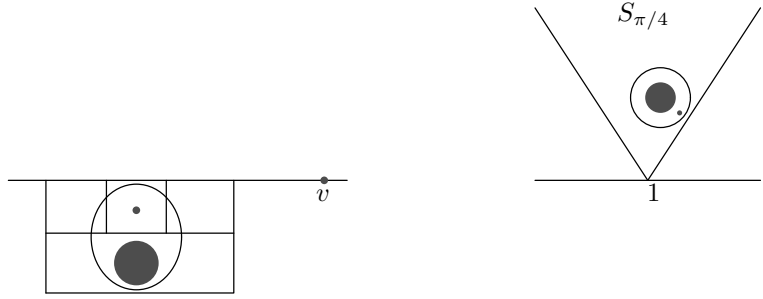

\bigskip
\begin{center}
\centertexdraw { \drawdim cm \linewd 0.02 \move(-2 2)

\move(-8 -1.5)    \lvec(-3.5 -1.5) \move(-7.5 -3)    \lvec(-5 -3)
\lvec(-5 -1.5) \move(-7.5 -3) \lvec(-7.5 -1.5) \move(-6.7 -2.2)
\lvec(-6.7 -1.5) \move(-5.9 -2.2) \lvec(-5.9 -1.5) \move(-7.5
-2.2)\lvec(-5 -2.2)

\move(-6.3 -1.9)  \fcir f:0.3 r:0.05  \move(-3.8 -1.5)  \fcir f:0.3
r:0.05  \move(-3.9 -1.78) \htext{$v$}

\move(-6.3 -2.6)  \fcir f:0.3 r:0.3

\move(-6.3 -2.25)  \lellip  rx:0.6 ry:0.7

\move(-1 -1.5)    \lvec(2 -1.5)  \move(0.5 -1.5) \lvec(-1 0.8)
\move(0.5 -1.5) \lvec(2 0.8)

\move(0.67 -0.4)  \fcir f:0.3 r:0.2  \move(0.67 -0.4)  \lcir r:0.4

\move(0.92 -0.6)  \fcir f:0.3 r:0.03

\move(0.5 -1.78) \htext{$1$}

\move(0.1 0.5) \htext{$S_{\pi/4}$}

}
\end{center}
\vspace{0.2cm} \caption{Subcase III of Case II}
\end{figure}

 Let $B \subset E$ be the Euclidean disk  guaranteed by Lemma~\ref{good
geometry}.  From  Lemma~\ref{good geometry} we have
\begin{equation}\label{crf}
 {\rm dist}(v, B) \ge {\rm dist}(\Bbb T, B) \asymp
|J|. \end{equation} From (\ref{db}), (\ref{crf}) and the fact that
${\rm diam}(E) \asymp |J|$ (cf. Lemma~\ref{good geometry}),  there
is a Jordan domain $A$ contained in $E$ such that
\begin{itemize} \item[i] $\omega \in A$,  \item[ii] $B \subset A$, \item[iii] ${\rm dist}(v, A) \succeq {\rm diam}(A)$. \end{itemize}

Let $U = \Psi(A)$ and $V = \Psi(B)$. Since $\Psi$ is like a branch
of the cubic root map, from (iii) we have
$$
{\rm dist}(1, U) \succeq {\rm diam}(U).
$$
Since $U = \Psi(A) \subset \Psi(E) \subset S_{\pi/4}$, we have ${\rm
dist}(U, \Bbb T)  \succeq {\rm dist}(U, 1)$. Thus we get
\begin{equation}\label{lse}
{\rm dist}(U, \Bbb T) \succeq {\rm diam}(U).
\end{equation}
Since $B \subset A \subset E$ and ${\rm diam}(B) \asymp {\rm
diam}(E)$, the pair $(A, B)$ has $L$-bounded geometry for some
universal  $L > 1$.  From (iii), it follows that the distortion of
$\Psi$ in $A$ is universally bounded. Thus $(U, V)$ has $K$-bounded
geometry for some universal constant $K > 1$.   Since $B \subset Y_n
\setminus Y_{n+2}$, $V  = \Psi(B) \subset X_{n}\setminus X_{n+2}$.
It follows that $(U, V)$ is a desired pair. This proves
Lemma~\ref{pre-lem} in Subcase III of Case II.

The proof of  Lemma~\ref{pre-lem} is completed.


\end{document}